\documentclass{article}
\usepackage{amsmath}
\usepackage{amssymb}

\newtheorem{theorem}{Theorem}[section]

\newtheorem{condition}{Condition}[section]

\newtheorem{corollary}{Corollary}

\newtheorem{definition}{Definition}[section]

\newtheorem{proposition}{Proposition}[section]

\newenvironment{proof}[1][Proof]{\noindent\textbf{#1.} }{\ \rule{0.5em}{0.5em}}

\input{tcilatex}
\begin{document}

\title{Automorphic Equivalence of Many-Sorted Algebras.}
\author{A. Tsurkov. \\
%EndAName
Institute of Mathematics and Statistics.\\
University S\~{a}o Paulo. \\
Rua do Mat\~{a}o, 1010 \\
Cidade Universit\'{a}ria \\
S\~{a}o Paulo - SP - Brasil - CEP 05508-090 \\
arkady.tsurkov@gmail.com}
\maketitle

\begin{abstract}
For every variety $\Theta $ of universal algebras we can consider the
category $\Theta ^{0}$\ of the finite generated free algebras of this
variety. The quotient group $\mathfrak{A/Y}$, where $\mathfrak{A}$ is a
group of the all automorphisms of the category $\Theta ^{0}$ and $\mathfrak{Y%
}$ is a subgroup of the all inner automorphisms of this category measures
difference between the geometric equivalence and automorphic equivalence of
algebras from the variety $\Theta $.

In \cite{PlotkinZhitAutCat} the simple and strong method of the verbal
operations was elaborated on for the calculation of the group $\mathfrak{A/Y}
$ in the case when the $\Theta $ is a variety of one-sorted algebras. In the
first part of our paper (Sections \ref{intro}, \ref{automorphisms} and \ref%
{verbal}) we prove that this method can be used in the case of many-sorted
algebras.

In the second part of our paper (Section \ref{alggeometry}) we apply the
results of the first part to the universal algebraic geometry of many-sorted
algebras and refine and reprove results of \cite{PlotkinSame} and \cite%
{TsurAutomEqAlg} for these algebras. For example we prove in the Theorem \ref%
{reduction} that the automorphic equivalence of algebras can be reduced to
the geometric equivalence if we change the operations in the one of these
algebras.

In the third part of this paper (Section \ref{examples}) we consider some
varieties of many-sorted algebras. We prove that automorphic equivalence
coincide with geometric equivalence in the variety of the all actions of
semigroups over sets and in the variety of the all automatons, because the
group $\mathfrak{A/Y}$ is trivial for this varieties. We also consider the
variety of the all representations of groups and the all representations of
Lie algebras. The group $\mathfrak{A/Y}$ is not trivial for these varieties
and for both these varieties we give an examples of the representations
which are automorphically equivalent but not geometrically equivalent.
\end{abstract}

\section{Introduction.\label{intro}}

\setcounter{equation}{0}

In this paper we consider many-sorted algebras. We suppose that there is a
finite set of names of sorts $\Gamma $. Many-sorted algebra, first of all,
is a set $A$ with the "sorting": mapping $\eta _{A}:A\rightarrow \Gamma $.
We call the set $\eta _{A}^{-1}\left( i\right) $ for $i\in \Gamma $ - set of
elements of the sort $i$ of the algebra $A$. We denote $\eta _{A}^{-1}\left(
i\right) =A^{\left( i\right) }$. If $a\in A^{\left( i\right) }$, then we
will many time denote $a=a^{\left( i\right) }$, with a view to emphasize
that $a$ is an element of the sort $i$. Contrary to the common approach we
allow that $A^{\left( i\right) }=\varnothing $. We denote $\mathrm{im}\eta
_{A}=\left\{ i\in \Gamma \mid A^{\left( i\right) }\neq \varnothing \right\}
=\Gamma _{A}$.

Also we suppose that there is a set of operations (signature) $\Omega $.
Every operation $\omega \in \Omega $ has a type $\tau _{\omega }=\left(
i_{1},\ldots ,i_{n};j\right) $, where $n\in 
%TCIMACRO{\U{2115} }%
%BeginExpansion
\mathbb{N}
%EndExpansion
$, $i_{1},\ldots ,i_{n},j\in \Gamma $. Operation $\omega \in \Omega $ of the
type $\left( i_{1},\ldots ,i_{n};j\right) $ is a partially defined mapping $%
\omega :A^{n}\rightarrow A$. This mapping is defined only for tuples $\left(
a_{1},\ldots ,a_{n}\right) \in A^{n}$ such that $a_{k}\in $ $A^{\left(
i_{k}\right) }$, $1\leq k\leq n$. The images of these tuples are elements of
the sort $j$: $\omega \left( a_{1},\ldots ,a_{n}\right) \in A^{\left(
j\right) }$. We suppose that all operations $\omega \in \Omega $ are closed.
It means that for all $a_{k}\in $ $A^{\left( i_{k}\right) }$, $1\leq k\leq n$%
, there exists $\omega \left( a_{1},\ldots ,a_{n}\right) \in A^{\left(
j\right) }$.

If there is at least one $k\in \left\{ 1,\ldots ,n\right\} $ such that $%
A^{\left( i_{k}\right) }=\varnothing $, then the operation $\omega \in
\Omega $ with the type $\tau _{\omega }=\left( i_{1},\ldots ,i_{n};j\right) $
defined only on empty set. But we still consider algebra $A$ as algebra with
operation $\omega $. It is possible that $n=0$. In this case the operation $%
\omega \in \Omega $ with the type $\tau _{\omega }=\left( i_{1},\ldots
,i_{n};j\right) $ is the operation of the taking a constant $\omega
=c^{\left( j\right) }$ of the sort $j$.

Now we will define the notion of the homomorphism of the two many-sorted
algebras. We assume that two many-sorted algebras $A$ and $B$ have the same
set of names of sorts $\Gamma $. We denote the set of operations in the
algebra $A$ by $\Omega ^{A}=\left\{ \omega _{\alpha }^{A}\mid \alpha \in
I\right\} $ and the set of operations in the algebra $B$ by $\Omega
^{B}=\left\{ \omega _{\alpha }^{B}\mid \alpha \in I\right\} $. We assume
that between these sets there is a one-to-one and onto correspondence such
that operations $\omega _{\alpha }^{A}$ and $\omega _{\alpha }^{B}$ have the
same type $\tau _{\alpha }=\left( i_{1},\ldots ,i_{n};j\right) $.
Homomorphism from $A$ to $B$ is the mapping $\varphi :A\rightarrow B$, which
conforms with the "sorting" $\eta _{A}$ and $\eta _{B}$ and conforms with
operations $\Omega ^{A}$ and $\Omega ^{B}$. "Conforms with the sorting" it
means that 
\begin{equation}
\eta _{A}=\eta _{B}\varphi  \label{homom_cond_1}
\end{equation}%
(the diagram%
\begin{equation*}
\begin{array}{ccc}
A & \overset{\varphi }{\rightarrow } & B \\ 
\eta _{A}\searrow &  & \swarrow \eta _{B} \\ 
& \Gamma & 
\end{array}%
\end{equation*}%
is commutative). "Conforms with the operations" it means that for every $%
\alpha \in I$ and every $a^{\left( i_{k}\right) }\in A^{\left( i_{k}\right)
} $, $1\leq k\leq n$, fulfills 
\begin{equation}
\varphi \left( \omega _{\alpha }^{A}\left( a^{\left( i_{1}\right) },\ldots
,a^{\left( i_{n}\right) }\right) \right) =\omega _{\alpha }^{B}\left(
\varphi \left( a^{\left( i_{1}\right) }\right) ,\ldots ,\varphi \left(
a^{\left( i_{n}\right) }\right) \right) .  \label{homom_cond_2}
\end{equation}%
If for any $k$ such that $1\leq k\leq n$ the $A^{\left( i_{k}\right)
}=\varnothing $ holds then this equality fulfills by the principle of the
empty set. From (\ref{homom_cond_1}) we can conclude that if $i\in \Gamma $, 
$B^{\left( i\right) }=\varnothing $, $A^{\left( i\right) }\neq \varnothing $
then homomorphisms from $A$ to $B$ are not defined: $\mathrm{Hom}\left(
A,B\right) =\varnothing $. In other words, if $\Gamma _{A}\nsubseteq \Gamma
_{B}$ then $\mathrm{Hom}\left( A,B\right) =\varnothing $.

The notions of a congruence and a quotient algebra we define by natural way.
The congruence must be conform with the "sorting", so for every algebra $A$
and every congruence $T\subseteq A^{2}$ the $T\subseteq
\dbigcup\limits_{i\in \Gamma _{A}}\left( A^{\left( i\right) }\right) ^{2}$
holds. We denote by $\Delta _{A}$ the minimal congruence in the algebra $A$: 
$\Delta _{A}=\left\{ \left( a,a\right) \mid a\in A\right\} $.

Now we will define the notion of the varieties of the many-sorted algebras.
We fix the set of names of sorts $\Gamma $ and the signature $\Omega $. We
take a set $X$, which we will call an alphabet. We suppose that this set has
a "sorting": a mapping $\chi :X\rightarrow \Gamma $. After this we define an
algebra of terms over the alphabet $X$. The notion of the term over the
alphabet $X$ we define by the induction by the construction: if $x=x^{\left(
i\right) }\in \chi ^{-1}\left( i\right) =X^{\left( i\right) }$, where $i\in
\Gamma $, then $x^{\left( i\right) }$ is a term of the sort $i$. If $\omega
\in \Omega $ has the type $\tau _{\omega }=\left( i_{1},\ldots
,i_{n};j\right) $ and $t_{k}$ is a term of the sort $i_{k}$, $1\leq k\leq n$%
, then $\omega \left( t_{1},\ldots ,t_{n}\right) $ is a term of the sort $j$%
. We denote the algebra of terms over the alphabet $X$ by $\widetilde{F}=%
\widetilde{F}\left( X\right) $. It is clear that $\chi _{\mid X}=\left( \eta
_{\widetilde{F}}\right) _{\mid X}$.

Pairs $\left( w_{1},w_{2}\right) \in \left( \left( \widetilde{F}\left(
X\right) \right) ^{\left( i\right) }\right) ^{2}$, $i\in \Gamma $, which we
denote as $w_{1}=w_{2}$, will be identities. We denote by $X^{\prime }$ the
finite subset $X^{\prime }\subset X$ of the letters from alphabet $X$ which
are really included in the term $w_{1}$ or in the term $w_{2}$. The $%
w_{1},w_{2}\in \widetilde{F}\left( X^{\prime }\right) $ holds. Now\ we
consider an arbitrary algebra $A$ with the set of names of sorts $\Gamma $
and the signature $\Omega $. We say that the identity $w_{1}=w_{2}$ fulfills
in the algebra $A$ if for every $\varphi \in \mathrm{Hom}\left( \widetilde{F}%
\left( X^{\prime }\right) ,A\right) $ the $\varphi \left( w_{1}\right)
=\varphi \left( w_{2}\right) $ holds. If there is a $i\in \Gamma $ such that 
$\left( X^{\prime }\right) ^{\left( i\right) }=X^{\prime }\cap X^{\left(
i\right) }\neq \varnothing $ and $A^{\left( i\right) }=\varnothing $ then $%
w_{1}=w_{2}$ fulfills in the algebra $A$ by the principle of the empty set.

If $\mathfrak{I}\subset \dbigcup\limits_{i\in \Gamma }\left( \left( 
\widetilde{F}\left( X\right) \right) ^{\left( i\right) }\right) ^{2}$ then
the family of algebras in which fulfill all identities from $\mathfrak{I}$
called the variety of algebras defined by the identities $\mathfrak{I}$. We
denote this variety by $\Theta \left( \mathfrak{I}\right) =\Theta $.

We say that the algebra $F\in \Theta $ is a free algebra of this variety
with the set of free generators $X\subset F$ if for every algebra $A\in
\Theta $, such that $\Gamma _{A}\supseteq \eta _{F}\left( X\right) $, and
every mapping $f:X\rightarrow A$, such that $\eta _{A}f\left( x\right) =\eta
_{F}\left( x\right) $ for every $x\in X$, there exists only one homomorphism 
$\varphi :F\rightarrow A$, such that $\varphi \left( x\right) =f\left(
x\right) $ for every $x\in X$. This algebra we denote by $F=F\left( X\right) 
$. It is clear that the algebra $\widetilde{F}\left( X\right) $ of terms
over the alphabet $X$ is a free algebra with the free generators $X$ in the
variety $\Theta \left( \varnothing \right) $ defined by the empty set of
identities. When we will construct the free algebras $F\left( X\right) $ in
arbitrary variety $\Theta $, we must consider the set $\mathfrak{I}_{\Theta
}\left( X\right) $ of the all identities from $\dbigcup\limits_{i\in \Gamma
_{A}}\left( \left( \widetilde{F}\left( X\right) \right) ^{\left( i\right)
}\right) ^{2}$ which fulfill in the all algebras $A\in \Theta $. The set $%
\mathfrak{I}_{\Theta }\left( X\right) $ is a congruence and $F\left(
X\right) =\widetilde{F}\left( X\right) /\mathfrak{I}_{\Theta }\left(
X\right) $.

In the end of this section we will say that the first and the second
theorems of homomorphisms, the projective propriety of free algebras fulfill
according to our approach to the notions of many-sorted algebras, their
homomorphisms and their varieties. If $A\in \Theta $ and there exists $i\in
\Gamma $ such that $A^{\left( i\right) }=\varnothing $ then there are free
algebras $F\left( X\right) $ of the variety $\Theta $, such that $\mathrm{Hom%
}\left( F\left( X\right) ,A\right) =\varnothing $. But for every $A\in
\Theta $ there exists free algebra $F\left( X\right) \in \Theta $, such that 
$A\cong F\left( X\right) /T$, where $T$ is a congruence. The Birkhoff
theorem about varieties can be proved according to our approach.

Also we will say that our approach to the notions of many-sorted algebras,
their homomorphisms and their varieties coincide with the common approach,
for example, when our signature $\Omega $ has operation of the taking a
constant $c^{\left( j\right) }$ for the every sort $j\in \Gamma $. It means
that our approach coincide with the common one in the cases of
representations of groups, representations of linear algebras, actions of
monoids over sets with stationary points and many other cases.

\section{Category $\Theta ^{0}$ and its automorphisms. Decomposition
theorem. \label{automorphisms}}

\setcounter{equation}{0}

Now we consider an arbitrary variety $\Theta =\Theta \left( \mathfrak{I}%
\right) $ of the many-sorted algebras with the set of names of sorts $\Gamma 
$, the set of operations $\Omega $ and defined by the identities $\mathfrak{I%
}$. We take $X_{0}$ and $\chi :X_{0}\rightarrow \Gamma $ such that $\chi
^{-1}\left( i\right) =X_{0}^{\left( i\right) }$ is an infinite countable set
for every $i\in \Gamma $. Free algebras $F\left( X\right) $ such that $%
X\subset X_{0}$, $\left\vert X\right\vert <\infty $ will be objects of the
category $\Theta ^{0}$, the homomorphisms of these algebras will be
morphisms of this category.

From now on we assume that the following condition holds in our variety $%
\Theta $:

\begin{condition}
\label{monoiso}$\Phi \left( F\left( x^{\left( i\right) }\right) \right)
\cong F\left( x^{\left( i\right) }\right) $ for every automorphism $\Phi $
of the category $\Theta ^{0}$, every sort $i\in \Gamma $ and every element
of this sort $x^{\left( i\right) }\in X_{0}^{\left( i\right) }\subset X_{0}$.
\end{condition}

The following theorem is a generalization of the first part of the Theorem 1
from \cite{PlotkinZhitAutCat}.

\begin{theorem}
\label{potinere}If $\Phi $ is an automorphism of the category $\Theta ^{0}$
then for every $A\in \mathrm{Ob}\Theta ^{0}$ there exists a bijection $%
s_{A}:A\rightarrow \Phi \left( A\right) $ such that $\eta _{A}=\eta _{\Phi
\left( A\right) }s_{A}$ and for every $\mu \in \mathrm{Mor}_{\Theta
^{0}}\left( A,B\right) $ the 
\begin{equation}
\Phi \left( \mu \right) =s_{B}\mu s_{A}^{-1}  \label{pot_inner}
\end{equation}
holds.
\end{theorem}

\begin{proof}
We take $a^{\left( i\right) }\in A^{\left( i\right) }\subset A$, $x^{\left(
i\right) }\in X_{0}^{\left( i\right) }\subset X_{0}$ and $F\left( x^{\left(
i\right) }\right) \in \mathrm{Ob}\Theta ^{0}$. There exists one homomorphism 
$\alpha :F\left( x^{\left( i\right) }\right) \rightarrow A$ such that $%
\alpha \left( x^{\left( i\right) }\right) =a^{\left( i\right) }$. $\Phi
\left( \alpha \right) :\Phi \left( F\left( x^{\left( i\right) }\right)
\right) \rightarrow \Phi \left( A\right) $. By Condition \ref{monoiso} there
exists isomorphism $\sigma :F\left( x^{\left( i\right) }\right) \rightarrow
\Phi \left( F\left( x^{\left( i\right) }\right) \right) $. We define $%
s_{A}\left( a^{\left( i\right) }\right) =\Phi \left( \alpha \right) \sigma
\left( x^{\left( i\right) }\right) $. $\Phi \left( \alpha \right) \sigma
:F\left( x^{\left( i\right) }\right) \rightarrow \Phi \left( A\right) $ is a
homomorphism, so by (\ref{homom_cond_1}) $s_{A}\left( a^{\left( i\right)
}\right) \in \left( \Phi \left( A\right) \right) ^{\left( i\right) }$.
Therefore $\eta _{A}=\eta _{\Phi \left( A\right) }s_{A}$.

We take $b^{\left( i\right) }\in \left( \Phi \left( A\right) \right)
^{\left( i\right) }\subset \Phi \left( A\right) $. There exists homomorphism 
$\beta :F\left( x^{\left( i\right) }\right) \rightarrow \Phi \left( A\right) 
$ such that $\beta \left( x^{\left( i\right) }\right) =b^{\left( i\right) }$
and homomorphism $\beta \sigma ^{-1}:\Phi \left( F\left( x^{\left( i\right)
}\right) \right) \rightarrow \Phi \left( A\right) $. Therefore exists a
homomorphism $\Phi ^{-1}\left( \beta \sigma ^{-1}\right) :F\left( x^{\left(
i\right) }\right) \rightarrow A$. $\Phi ^{-1}\left( \beta \sigma
^{-1}\right) \left( x^{\left( i\right) }\right) =a^{\left( i\right) }\in
A^{\left( i\right) }$, so $\Phi ^{-1}\left( \beta \sigma ^{-1}\right)
=\alpha $, such that $\alpha \left( x^{\left( i\right) }\right) =a^{\left(
i\right) }$. Hence $\beta =\Phi \left( \alpha \right) \sigma $, $\Phi \left(
\alpha \right) \sigma \left( x^{\left( i\right) }\right) =b^{\left( i\right)
}$ and $b^{\left( i\right) }=s_{A}\left( a^{\left( i\right) }\right) $. So $%
s_{A}$ is a surjection.

We assume that $a_{1}^{\left( i\right) },a_{2}^{\left( i\right) }\in
A^{\left( i\right) }$ and $s_{A}\left( a_{1}^{\left( i\right) }\right)
=s_{A}\left( a_{2}^{\left( i\right) }\right) \in \left( \Phi \left( A\right)
\right) ^{\left( i\right) }$. We consider the homomorphisms $\alpha
_{1},\alpha _{2}:F\left( x^{\left( i\right) }\right) \rightarrow A$ such
that $\alpha _{j}\left( x^{\left( i\right) }\right) =a_{j}^{\left( i\right)
} $, $j=1,2$. We have that $\Phi \left( \alpha _{1}\right) \sigma \left(
x^{\left( i\right) }\right) =\Phi \left( \alpha _{2}\right) \sigma \left(
x^{\left( i\right) }\right) $. $\Phi \left( \alpha _{1}\right) \sigma ,\Phi
\left( \alpha _{2}\right) \sigma :F\left( x^{\left( i\right) }\right)
\rightarrow \Phi \left( A\right) $, so $\Phi \left( \alpha _{1}\right)
\sigma =\Phi \left( \alpha _{2}\right) \sigma $. Therefore $\alpha
_{1}=\alpha _{2}$ and $a_{1}^{\left( i\right) }=a_{2}^{\left( i\right) }$.
So $s_{A}$ is a injection.

We consider $\mu \in \mathrm{Mor}_{\Theta ^{0}}\left( A,B\right) $ and $%
a^{\left( i\right) }\in A^{\left( i\right) }\subset A$. $s_{B}\mu \left(
a^{\left( i\right) }\right) =\Phi \left( \beta \right) \sigma \left(
x^{\left( i\right) }\right) $, where $\sigma $ is an isomorphism $\sigma
:F\left( x^{\left( i\right) }\right) \rightarrow \Phi \left( F\left(
x^{\left( i\right) }\right) \right) $, $\beta $ is a homomorphism $\beta
:F\left( x^{\left( i\right) }\right) \rightarrow B$ such that $\beta \left(
x^{\left( i\right) }\right) =\mu \left( a^{\left( i\right) }\right) $. $\Phi
\left( \mu \right) s_{A}\left( a^{\left( i\right) }\right) =\Phi \left( \mu
\right) \Phi \left( \alpha \right) \sigma \left( x^{\left( i\right) }\right) 
$, where $\alpha $ is a homomorphism $\alpha :F\left( x^{\left( i\right)
}\right) \rightarrow A$ such that $\alpha \left( x^{\left( i\right) }\right)
=a^{\left( i\right) }$. $\mu \alpha \left( x^{\left( i\right) }\right) =\mu
\left( a^{\left( i\right) }\right) $, so $\mu \alpha =\beta $ and $\Phi
\left( \beta \right) =\Phi \left( \mu \right) \Phi \left( \alpha \right) $.
Therefore $s_{B}\mu \left( a^{\left( i\right) }\right) =\Phi \left( \mu
\right) s_{A}\left( a^{\left( i\right) }\right) $ and $\Phi \left( \mu
\right) =s_{B}\mu s_{A}^{-1}$.
\end{proof}

From this theorem we conclude that for every automorphism $\Phi $ of the
category $\Theta ^{0}$ and every $\mu \in \mathrm{Mor}_{\Theta ^{0}}\left(
A,B\right) $ the diagram%
\begin{equation*}
\begin{array}{ccc}
A & \underset{s_{A}}{\rightarrow } & \Phi \left( A\right) \\ 
\downarrow \mu &  & \Phi \left( \mu \right) \downarrow \\ 
B & \overset{s_{B}}{\rightarrow } & \Phi \left( B\right)%
\end{array}%
\end{equation*}%
commutes.

The following theorem is a generalization of the second part of the Theorem
1 from \cite{PlotkinZhitAutCat}.

\begin{theorem}
\label{isom}If $\Phi $ is an automorphism of the category $\Theta ^{0}$ then
for every $A\in \mathrm{Ob}\Theta ^{0}$ the $\Phi \left( A\right) \cong A$.
\end{theorem}

\begin{proof}
We denote $\Phi \left( A\right) =B$, $\Phi ^{-1}\left( A\right) =C$. $%
A=F\left( X\right) $ such that $X\subset X_{0}$, $\left\vert X\right\vert
<\infty $. By Theorem \ref{potinere} there exist bijections $%
s_{C}:C\rightarrow A$, $s_{A}:A\rightarrow B$., so $\Gamma _{A}=\Gamma
_{B}=\Gamma _{C}$. Hence there exists one homomorphism $\sigma :A\rightarrow
B$, such that $\sigma \left( x\right) =s_{A}\left( x\right) $ and there
exists one homomorphism $\tau :A\rightarrow C$, such that $\tau \left(
x\right) =s_{C}^{-1}\left( x\right) $ for every $x\in X$. $\Phi \left( \tau
\right) :\Phi \left( A\right) =B\rightarrow \Phi \left( C\right) =A$. $\Phi
\left( \tau \right) \sigma \left( x\right) =s_{C}\tau s_{A}^{-1}\sigma
\left( x\right) =s_{C}\tau \left( x\right) =x$ holds for every $x\in X$.
Therefore $\Phi \left( \tau \right) \sigma =id_{A}$.

$\Phi ^{-1}\left( \sigma \right) :\Phi ^{-1}\left( A\right) =C\rightarrow
\Phi ^{-1}\left( B\right) =A$. $\Phi ^{-1}\left( \sigma \right) \tau \left(
x\right) =s_{A}^{-1}\sigma s_{C}\tau \left( x\right) =s_{A}^{-1}\sigma
\left( x\right) =x$ holds for every $x\in X$. Therefore $\Phi ^{-1}\left(
\sigma \right) \tau =id_{A}$. Automorphism $\Phi $ provides an isomorphisms
of monoids $\mathrm{End}A\rightarrow \mathrm{End}\Phi \left( A\right) $, $%
id_{A}$ is an unit of $\mathrm{End}A$, $id_{\Phi \left( A\right) }$ is an
unit of $\mathrm{End}\Phi \left( A\right) $, so $\Phi \left( \Phi
^{-1}\left( \sigma \right) \tau \right) =\sigma \Phi \left( \tau \right)
=id_{B}$.
\end{proof}

\begin{definition}
\label{inner}An automorphism $\Upsilon $ of an arbitrary category $\mathfrak{%
K}$ is \textbf{inner}, if it is isomorphic as a functor to the identity
automorphism of the category $\mathfrak{K}$.
\end{definition}

This means that for every $A\in \mathrm{Ob}\mathfrak{K}$ there exists an
isomorphism $s_{A}^{\Upsilon }:A\rightarrow \Upsilon \left( A\right) $ such
that for every $\mu \in \mathrm{Mor}_{\mathfrak{K}}\left( A,B\right) $ the
diagram%
\begin{equation*}
\begin{array}{ccc}
A & \overrightarrow{s_{A}^{\Upsilon }} & \Upsilon \left( A\right) \\ 
\downarrow \mu &  & \Upsilon \left( \mu \right) \downarrow \\ 
B & \underrightarrow{s_{B}^{\Upsilon }} & \Upsilon \left( B\right)%
\end{array}%
\end{equation*}%
\noindent commutes. It is clear that the set of all inner automorphisms of
an arbitrary category $\mathfrak{K}$ is a normal subgroup of the group of
all automorphisms of this category.

\begin{definition}
\label{str_stab_aut}\textbf{\hspace{-0.08in} }\textit{An automorphism $\Phi $
of the category }$\Theta ^{0}$\textit{\ is called \textbf{strongly stable}
if it satisfies the conditions:}

\begin{enumerate}
\item[A1)] $\Phi $\textit{\ preserves all objects of }$\Theta ^{0}$\textit{,}

\item[A2)] \textit{there exists a system of bijections }$S=\left\{
s_{F}:F\rightarrow F\mid F\in \mathrm{Ob}\Theta ^{0}\right\} $\textit{\ such
that all these bijections }conform with the sorting:%
\begin{equation*}
\eta _{F}=\eta _{F}s_{F}
\end{equation*}

\item[A3)] $\Phi $\textit{\ acts on the morphisms }$\mu \in \mathrm{Mor}%
_{\Theta ^{0}}\left( A,B\right) $\textit{\ of }$\Theta ^{0}$\textit{\ by
this way: }%
\begin{equation}
\Phi \left( \mu \right) =s_{B}\mu s_{A}^{-1},  \label{biject_action}
\end{equation}

\item[A4)] $s_{F}\mid _{X}=id_{X},$ \textit{\ for every }$F\left( X\right)
\in \mathrm{Ob}\Theta ^{0}$.
\end{enumerate}
\end{definition}

It is clear that the set of all strongly stable automorphisms of the
category $\Theta ^{0}$ is a subgroup of the group of all automorphisms of
this category. We denote by $\mathfrak{A}$ the group of the all
automorphisms of the category $\Theta ^{0}$, by $\mathfrak{Y}$ the group of
the all inner automorphisms of this category, by $\mathfrak{S}$ the group of
the all strongly stable automorphisms of this category.

The following decomposition theorem is a is a generalization of the Theorem
2 from \cite{PlotkinZhitAutCat}.

\begin{theorem}
\label{factorisation}$\mathfrak{A=YS=SY}$, where $\mathfrak{A}$ is a group
of the all automorphisms of the category $\Theta ^{0}$,
\end{theorem}

\begin{proof}
We consider $\Phi \in \mathfrak{A}$. In the Theorem \ref{potinere} we prove
that there is a system of bijections $\left\{ s_{A}\mid A\in \mathrm{Ob}%
\Theta ^{0}\right\} $ such that $s_{A}:A\rightarrow \Phi \left( A\right) $
and $\Phi \left( \mu \right) =s_{B}\mu s_{A}^{-1}$ if $\mu \in \mathrm{Mor}%
_{\Theta ^{0}}\left( A,B\right) $. In the Theorem \ref{isom} we prove that
for every $A=F\left( X\right) \in \mathrm{Ob}\Theta ^{0}$ there is an
isomorphism $\sigma _{A}:A\rightarrow \Phi \left( A\right) $, such that $%
\sigma _{A}\left( x\right) =s_{A}\left( x\right) $ for every $x\in X$.

Now we will define two automorphisms $\Upsilon $ and $\Psi $. $\Upsilon $ we
define by this way: $\Upsilon \left( A\right) =\Phi \left( A\right) $ for
every $A\in \mathrm{Ob}\Theta ^{0}$, $\Upsilon \left( \mu \right) =\sigma
_{B}\mu \sigma _{A}^{-1}$ for every $\mu \in \mathrm{Mor}_{\Theta
^{0}}\left( A,B\right) $. $\Psi $ we define by this way: $\Psi \left(
A\right) =A$ for every $A\in \mathrm{Ob}\Theta ^{0}$, $\Psi \left( \mu
\right) =\sigma _{B}^{-1}s_{B}\mu s_{A}^{-1}\sigma _{A}$ for every $\mu \in 
\mathrm{Mor}_{\Theta ^{0}}\left( A,B\right) $. It is easy to check that $%
\Upsilon $ and $\Psi $ are functors. Also we can remark that $\Upsilon $ and 
$\Psi $ have inverted functors $\Upsilon ^{-1}$ and $\Psi ^{-1}$: if we
define $\Upsilon ^{-1}\left( A\right) =\Phi ^{-1}\left( A\right) $ for every 
$A\in \mathrm{Ob}\Theta ^{0}$, $\Upsilon ^{-1}\left( \mu \right) =\sigma
_{\Phi ^{-1}\left( B\right) }^{-1}\mu \sigma _{\Phi ^{-1}\left( A\right) }$
for every $\mu \in \mathrm{Mor}_{\Theta ^{0}}\left( A,B\right) $ and $\Psi
^{-1}\left( A\right) =A$ for every $A\in \mathrm{Ob}\Theta ^{0}$, $\Psi
^{-1}\left( \mu \right) =s_{B}^{-1}\sigma _{B}\mu \sigma _{A}^{-1}s_{A}$ for
every $\mu \in \mathrm{Mor}_{\Theta ^{0}}\left( A,B\right) $ - then $%
\Upsilon ^{-1}\Upsilon =\Upsilon \Upsilon ^{-1}=\Psi ^{-1}\Psi =\Psi \Psi
^{-1}=id_{\Theta ^{0}}$. Therefore $\Upsilon $ and $\Psi $ are automorphisms.

If $A=F\left( X\right) \in \mathrm{Ob}\Theta ^{0}$ then $s_{A}^{-1}\sigma
_{A}\left( x\right) =x$ for every $x\in X$, so $\Psi $ is a strongly stable
automorphism.

We can conclude from $\Phi \left( \mu \right) =s_{B}\mu s_{A}^{-1}=\sigma
_{B}\left( \sigma _{B}^{-1}s_{B}\mu s_{A}^{-1}\sigma _{A}\right) \sigma
_{A}^{-1}$ for every $\mu \in \mathrm{Mor}_{\Theta ^{0}}\left( A,B\right) $
that $\Phi =\Upsilon \Psi $. It is well known that the group $\mathfrak{Y}$
is a normal subgroup of the group $\mathfrak{A}$. This completes the proof.
\end{proof}

\section{Strongly stable automorphisms and systems of verbal operations. 
\label{verbal}}

\setcounter{equation}{0}

\subsection{Strongly stable automorphisms and systems of bijections.}

If we have a strongly stable automorphism $\Phi $ of the category $\Theta
^{0}$, then by Definition \ref{str_stab_aut} there exists a system of
bijections $S=\left\{ s_{F}:F\rightarrow F\mid F\in \mathrm{Ob}\Theta
^{0}\right\} $, such that

\begin{enumerate}
\item[B1)] for every\textit{\ }$F\in \mathrm{Ob}\Theta ^{0}$ the $\eta
_{F}=\eta _{F}s_{F}$ holds,

\item[B2)] for every $A,B\in \mathrm{Ob}\Theta ^{0}$ and every $\mu \in 
\mathrm{Mor}_{\Theta ^{0}}\left( A,B\right) $ the mappings $s_{B}\mu
s_{A}^{-1},s_{B}^{-1}\mu s_{A}:A\rightarrow B$ are homomorphisms,

\item[B3)] for every\textit{\ }$F\left( X\right) \in \mathrm{Ob}\Theta ^{0}$
the $s_{F}\mid _{X}=id_{X}$ holds.
\end{enumerate}

\begin{proposition}
For strongly stable automorphism $\Phi $ of the category $\Theta ^{0}$ there
is only one system of bijections $S=\left\{ s_{F}:F\rightarrow F\mid F\in 
\mathrm{Ob}\Theta ^{0}\right\} $ such that $\Phi $ acts on homomorphisms by
these bijections and the system $S$ fulfills conditions B1) - B3).
\end{proposition}

\begin{proof}
By Definition \ref{str_stab_aut} there exists a system of bijections $%
S=\left\{ s_{F}:F\rightarrow F\mid F\in \mathrm{Ob}\Theta ^{0}\right\} $\
such that for every\textit{\ }$F\in \mathrm{Ob}\Theta ^{0}$ the $\eta
_{F}=\eta _{F}s_{F}$ holds, for every $A,B\in \mathrm{Ob}\Theta ^{0}$ and
every $\mu \in \mathrm{Mor}_{\Theta ^{0}}\left( A,B\right) $ the $\Phi
\left( \mu \right) =s_{B}\mu s_{A}^{-1}$ holds and $s_{F}\mid _{X}=id_{X}$, 
\textit{\ }for every\textit{\ }$F\left( X\right) \in \mathrm{Ob}\Theta ^{0}$%
. $\Phi ^{-1}\left( \mu \right) =s_{B}^{-1}\mu s_{A}$ is also homomorphism,
so $S$ fulfills conditions B1) - B3). For every $F\in \mathrm{Ob}\Theta ^{0}$
and every $f^{\left( i\right) }\in F^{\left( i\right) }$, $i\in \Gamma $, we
take $A=F\left( x^{\left( i\right) }\right) $ where $x^{\left( i\right) }\in
X_{0}^{\left( i\right) }$. There exist homomorphism $\alpha :A\ni x^{\left(
i\right) }\rightarrow f^{\left( i\right) }\in F^{\left( i\right) }$. The $%
s_{F}\left( f^{\left( i\right) }\right) =s_{F}\alpha \left( x^{\left(
i\right) }\right) =s_{F}\alpha s_{A}^{-1}\left( x^{\left( i\right) }\right)
=\Phi \left( \alpha \right) \left( x^{\left( i\right) }\right) $ holds, so
all bijections of $S$ uniquely defined by automorphism $\Phi $.
\end{proof}

The systems of bijections $S$ which is subject of this Proposition we denote
by $S^{\Phi }$.

\begin{proposition}
\label{automorphismbijections}The mapping $\Phi \rightarrow S^{\Phi }$ is
one to one and onto correspondence between the family of the all strongly
stable automorphisms of the category $\Theta ^{0}$ and the family of the all
systems of bijections $S=\left\{ s_{F}:F\rightarrow F\mid F\in \mathrm{Ob}%
\Theta ^{0}\right\} $ which fulfills conditions B1) - B3).
\end{proposition}

\begin{proof}
If we have a system of bijections $S$ which fulfills conditions B1) - B3) we
can define a functor $\Phi :\Theta ^{0}\rightarrow \Theta ^{0}$, such that
preserves all objects of $\Theta ^{0}$ and for every $A,B\in \mathrm{Ob}%
\Theta ^{0}$ and every $\mu \in \mathrm{Mor}_{\Theta ^{0}}\left( A,B\right) $
the $\Phi \left( \mu \right) =s_{B}\mu s_{A}^{-1}$ holds. There is an
inverse functor $\Phi :\Theta ^{0}\rightarrow \Theta ^{0}$, such that the $%
\Phi ^{-1}\left( \mu \right) =s_{B}^{-1}\mu s_{A}$ holds. Therefore $\Phi $
is a strongly stable automorphism of the category $\Theta ^{0}$ and $S^{\Phi
}=S$. Hence our correspondence is onto.

If $\Phi $ and $\Psi $ are strongly stable automorphisms of the category $%
\Theta ^{0}$ and $S^{\Phi }=S^{\Psi }=S=\left\{ s_{F}:F\rightarrow F\mid
F\in \mathrm{Ob}\Theta ^{0}\right\} $, then $\Phi =\Psi $, because, for $%
F\in \mathrm{Ob}\Theta ^{0}$ the $\Phi \left( F\right) =\Psi \left( F\right)
=F$ holds and for every $A,B\in \mathrm{Ob}\Theta ^{0}$ and every $\mu \in 
\mathrm{Mor}_{\Theta ^{0}}\left( A,B\right) $ the $\Phi \left( \mu \right)
=s_{B}\mu s_{A}^{-1}=\Psi \left( \mu \right) $ holds. Hence our
correspondence is one to one.
\end{proof}

\subsection{Systems of bijections and systems of verbal operations.}

We take the word $w=w\left( x_{1},\ldots ,x_{n}\right) \in F\left(
x_{1},\ldots ,x_{n}\right) =F\in \mathrm{Ob}\Theta ^{0}$. For every algebra $%
H\in \Theta $ we can define an operation $w_{H}^{\ast }$: if $h_{1},\ldots
,h_{n}\in H$ such that $\eta _{H}\left( h_{i}\right) =\eta _{F}\left(
x_{i}\right) $, where $1\leq i\leq n$, then $w_{H}^{\ast }\left(
h_{1},\ldots ,h_{n}\right) =\alpha \left( w\right) $, where $\alpha
:F\rightarrow H$ homomorphism such that $\alpha \left( x_{i}\right) =h_{i}$.
If $\eta _{F}\left\{ x_{1},\ldots ,x_{n}\right\} \nsubseteq \Gamma _{H}$
then the operation $w_{H}^{\ast }$ is defined on the empty subset of $H^{n}$%
. The operation $w_{H}^{\ast }$ is called the verbal operation defined by
the word $w$. This operation we consider as the operation of the type $%
\left( \eta _{F}\left( x_{1}\right) ,\ldots ,\eta _{F}\left( x_{n}\right)
;\eta _{F}\left( w\right) \right) $ even if not all free generators $%
x_{1},\ldots ,x_{n}$ really enter to the word $w$.

If we have a system of words $W=\left\{ w_{i}\mid i\in I\right\} $ then for
every $H\in \Theta $ we denote by $H_{W}^{\ast }$ the universal algebra
which coincide with $H$ as a set with the "sorting", but has only verbal
operations defined by the words from $W$. It is easy to prove that if $%
H_{1},H_{2}\in \Theta $ and $\varphi :H_{1}\rightarrow H_{2}$ is a
homomorphism then $\varphi \left( w_{H_{1}}^{\ast }\left( h_{1},\ldots
,h_{n}\right) \right) =w_{H_{2}}^{\ast }\left( \varphi \left( h_{1}\right)
,\ldots ,\varphi \left( h_{n}\right) \right) $ if both sides of this
equality are defined. So, if $W$ is a system of words and $\varphi
:H_{1}\rightarrow H_{2}$ is a homomorphism then $\varphi :\left(
H_{1}\right) _{W}^{\ast }\rightarrow \left( H_{2}\right) _{W}^{\ast }$ is a
homomorphism too.

Now we assume that there exists a system of bijections $S=\left\{
s_{A}:A\rightarrow A\mid A\in \mathrm{Ob}\Theta ^{0}\right\} $ which
fulfills conditions B1) - B3).

For the operation $\omega \in \Omega $ which has a type $\tau _{\omega
}=\left( i_{1},\ldots ,i_{n};j\right) $, we take $A_{\omega }=F\left(
X_{\omega }\right) \in \mathrm{Ob}\Theta ^{0}$ such that $X_{\omega
}=\left\{ x^{\left( i_{1}\right) },\ldots ,x^{\left( i_{n}\right) }\right\} $%
, $\eta _{A_{\omega }}\left( x^{\left( i_{k}\right) }\right) =i_{k}$, $1\leq
k\leq n$. 
\begin{equation}
s_{A_{\omega }}\left( \omega \left( x^{\left( i_{1}\right) },\ldots
,x^{\left( i_{n}\right) }\right) \right) =w_{\omega }\left( x^{\left(
i_{1}\right) },\ldots ,x^{\left( i_{n}\right) }\right) \in A_{\omega }.
\label{wordformula}
\end{equation}%
For every $H\in \Theta $ we define by the word $w_{\omega }$ the verbal
operation $\omega _{H}^{\ast }$. We denote $W=\left\{ w_{\omega }\mid \omega
\in \Omega \right\} $. The types of the operations $\omega $ and $\omega
^{\ast }$ coincides.

\begin{proposition}
\label{*isom}For every $F\in \mathrm{Ob}\Theta ^{0}$ the bijection $s_{F}$
is an isomorphism $s_{F}:F\rightarrow F_{W}^{\ast }$.
\end{proposition}

\begin{proof}
We consider the operation $\omega \in \Omega $ which has a type $\tau
_{\omega }=\left( i_{1},\ldots ,i_{n};j\right) $. We need to prove that for
every $f^{\left( i_{1}\right) },\ldots ,f^{\left( i_{n}\right) }\in F$ such
that $\eta _{F}\left( f^{\left( i_{r}\right) }\right) =i_{r}$, $1\leq r\leq
n $, the $s_{F}\omega \left( f^{\left( i_{1}\right) },\ldots ,f^{\left(
i_{n}\right) }\right) =\omega ^{\ast }\left( s_{F}\left( f^{\left(
i_{1}\right) }\right) ,\ldots ,s_{F}\left( f^{\left( i_{n}\right) }\right)
\right) $ holds. We consider the homomorphisms $\alpha ,\beta :A_{\omega
}\rightarrow F$ such that $\alpha \left( x^{\left( i_{r}\right) }\right)
=f^{\left( i_{r}\right) }$, $\beta \left( x^{\left( i_{r}\right) }\right)
=s_{F}\left( f^{\left( i_{r}\right) }\right) $, $1\leq r\leq n$. By our
assumption $s_{F}\alpha s_{A_{\omega }}^{-1}$ is also homomorphism from $A$
to $F$. $s_{F}\alpha s_{A_{\omega }}^{-1}\left( x^{\left( i_{r}\right)
}\right) =s_{F}\alpha \left( x^{\left( i_{r}\right) }\right) =s_{F}\left(
f^{\left( i_{r}\right) }\right) $, therefore $\beta =s_{F}\alpha
s_{A_{\omega }}^{-1}$. So 
\begin{equation*}
s_{F}\omega \left( f^{\left( i_{1}\right) },\ldots ,f^{\left( i_{n}\right)
}\right) =s_{F}\alpha \omega \left( x^{\left( i_{1}\right) },\ldots
,x^{\left( i_{n}\right) }\right) =s_{F}\alpha s_{A_{\omega
}}^{-1}s_{A_{\omega }}\omega \left( x^{\left( i_{1}\right) },\ldots
,x^{\left( i_{n}\right) }\right) =
\end{equation*}%
\begin{equation*}
\beta w_{\omega }\left( x^{\left( i_{1}\right) },\ldots ,x^{\left(
i_{n}\right) }\right) =\omega ^{\ast }\left( s_{F}\left( f^{\left(
i_{1}\right) }\right) ,\ldots ,s_{F}\left( f^{\left( i_{n}\right) }\right)
\right) .
\end{equation*}
\end{proof}

Now we assume that there exists a system of words $W$ such that

\begin{enumerate}
\item[Op1)] $W=\left\{ w_{\omega }\in A_{\omega }\mid \omega \in \Omega
\right\} $, where $A_{\omega }$ is defined as above, and

\item[Op2)] for every $F\left( X\right) \in \mathrm{Ob}\Theta ^{0}$ there
exists a bijection $s_{F}:F\rightarrow F$ such that $\left( s_{F}\right)
_{\mid X}=id_{X}$ and $s_{F}:F\rightarrow F_{W}^{\ast }$ is an isomorphism.
\end{enumerate}

The verbal operations defined by the word $w_{\omega }$ we denote by $\omega
^{\ast }$. We have a new signature $\Omega ^{\ast }=\left\{ \omega ^{\ast
}\mid \omega \in \Omega \right\} $.

Because $s_{F}:F\rightarrow F_{W}^{\ast }$ is an isomorphism, $F_{W}^{\ast }$
such that $F\in \mathrm{Ob}\Theta ^{0}$ is also a free algebra in the
variety $\Theta $.

By our assumption there is a symmetry between the signatures $\Omega $ and $%
\Omega ^{\ast }$.

\begin{proposition}
\label{inverswords}Every $\omega \in \Omega $ is a verbal operations defined
by the some word written in the signature $\Omega ^{\ast }$.
\end{proposition}

\begin{proof}
$s_{A_{\omega }}$ is a bijection, so there is $u_{\omega }\left( x^{\left(
i_{1}\right) },\ldots ,x^{\left( i_{n}\right) }\right) \in A_{\omega }$ such
that $s_{A_{\omega }}\left( u_{\omega }\left( x^{\left( i_{1}\right)
},\ldots ,x^{\left( i_{n}\right) }\right) \right) =\omega \left( x^{\left(
i_{1}\right) },\ldots ,x^{\left( i_{n}\right) }\right) $. $s_{A_{\omega }}$
is an isomorphism and $\left( s_{F}\right) _{\mid X}=id_{X}$, so $\omega
\left( x^{\left( i_{1}\right) },\ldots ,x^{\left( i_{n}\right) }\right)
=u_{\omega }^{\ast }\left( x^{\left( i_{1}\right) },\ldots ,x^{\left(
i_{n}\right) }\right) $ where $u_{\omega }^{\ast }\left( x^{\left(
i_{1}\right) },\ldots ,x^{\left( i_{n}\right) }\right) $ we achieve from the
word $u_{\omega }\left( x^{\left( i_{1}\right) },\ldots ,x^{\left(
i_{n}\right) }\right) $ when we change the all operations from $\Omega $ by
corresponding operation from $\Omega ^{\ast }$. So for every $H\in \Theta $
and every $h^{\left( i_{k}\right) }\in H^{\left( i_{k}\right) }$, $1\leq
k\leq n$, $\omega _{H}\left( h^{\left( i_{1}\right) },\ldots ,h^{\left(
i_{n}\right) }\right) =\alpha \omega \left( x^{\left( i_{1}\right) },\ldots
,x^{\left( i_{n}\right) }\right) =\alpha u_{\omega }^{\ast }\left( x^{\left(
i_{1}\right) },\ldots ,x^{\left( i_{n}\right) }\right) $, where $\alpha
:A_{\omega }\ni x^{\left( i_{k}\right) }\rightarrow h^{\left( i_{k}\right)
}\in H$, $1\leq k\leq n$, is a homomorphism. But also $\alpha $ is a
homomorphism from $\left( A_{\omega }\right) _{W}^{\ast }$ to $H_{W}^{\ast }$%
.
\end{proof}

\begin{corollary}
\label{invershomomorphism}If $H_{1},H_{2}\in \Theta $ and $\mu :\left(
H_{1}\right) _{W}^{\ast }\rightarrow \left( H_{2}\right) _{W}^{\ast }$ is a
homomorphism then $\mu :H_{1}\rightarrow H_{2}$ is also a homomorphism.
\end{corollary}

\begin{corollary}
\label{bijectionsproprieties}For every $A,B\in \mathrm{Ob}\Theta ^{0}$ and
every $\mu \in \mathrm{Mor}_{\Theta ^{0}}\left( A,B\right) $ the mappings $%
s_{B}\mu s_{A}^{-1},s_{B}^{-1}\mu s_{A}:A\rightarrow B$ are also
homomorphisms.
\end{corollary}

\setcounter{corollary}{0}

\begin{proposition}
\label{inthevariety}If $H\in \Theta $ then $H_{W}^{\ast }\in \Theta $.
\end{proposition}

\begin{proof}
There is $F\in \mathrm{Ob}\Theta ^{0}$ such that exists epimorphism $\varphi
:F\rightarrow H$. $\varphi $ is also an epimorphism from $F_{W}^{\ast }$ to $%
H_{W}^{\ast }$. Therefore $\varphi s_{F}$ is also an epimorphism from $F$ to 
$H_{W}^{\ast }$.
\end{proof}

If we have a system of words $W$ which fulfills the conditions Op1) and Op2)
then by Proposition \ref{inthevariety} the bijections $\left\{
s_{F}:F\rightarrow F\mid F\in \mathrm{Ob}\Theta ^{0}\right\} $ which are
subjects of the condition Op2) uniquely defined by the system of words $W$.
This system of bijections we denote by $S^{W}$.

\begin{proposition}
\label{wordsbijections}The mapping $W\rightarrow S^{W}$ is one to one and
onto correspondence between the family of the all systems of words which
fulfill the conditions Op1) and Op2) and the family of the all systems of
bijections \linebreak $S=\left\{ s_{F}:F\rightarrow F\mid F\in \mathrm{Ob}%
\Theta ^{0}\right\} $ which fulfills conditions B1) - B3).
\end{proposition}

\begin{proof}
We consider a system of words $W$ which fulfills the conditions Op1) and
Op2). By condition Op2) for every\textit{\ }$F\left( X\right) \in \mathrm{Ob}%
\Theta ^{0}$ the $s_{F}\mid _{X}=id_{X}$ holds and all\ the bijections $%
s_{F}:F\rightarrow F$ conform with the sorting of the algebras $F$. So the
system of bijections $S^{W}$ fulfills the conditions B1) and B2). By
Corollary \ref{bijectionsproprieties} from the Proposition \ref{inverswords}
the system of bijections $S^{W}$ fulfills the condition B2).

If we have the systems of bijections $S$ which fulfills conditions B1) - B3)
we can define the system of words $W=\left\{ w_{\omega }\in A_{\omega }\mid
\omega \in \Omega \right\} $ by formula (\ref{wordformula}). This system of
words fulfills the condition Op1) and by Proposition \ref{*isom} - the
condition Op2). The $S^{W}=S$ holds, so our correspondence is onto.

We assume that we have two systems of words $W=\left\{ w_{\omega }^{W}\in
A_{\omega }\mid \omega \in \Omega \right\} $ and $V=\left\{ w_{\omega
}^{V}\in A_{\omega }\mid \omega \in \Omega \right\} $ which fulfill the
conditions Op1) and Op2) and $S^{W}=S=S^{V}$. We take $\omega \in \Omega $
such that $\tau _{\omega }=\left( i_{1},\ldots ,i_{n};j\right) $, $%
i_{1},\ldots ,i_{n},j\in \Gamma $, $A_{\omega }=F\left( X_{\omega }\right) $
as above and $x^{\left( i_{k}\right) }$ such that $\eta _{A_{\omega }}\left(
x^{\left( i_{k}\right) }\right) =i_{k}$, $1\leq k\leq n$.. By condition Op2) 
$s_{A_{\omega }}^{W}\left( \omega \left( x^{\left( i_{1}\right) },\ldots
,x^{\left( i_{n}\right) }\right) \right) =w_{\omega }^{W}\left( x^{\left(
i_{1}\right) },\ldots ,x^{\left( i_{n}\right) }\right) =s_{A_{\omega
}}^{V}\left( \omega \left( x^{\left( i_{1}\right) },\ldots ,x^{\left(
i_{n}\right) }\right) \right) =w_{\omega }^{V}\left( x^{\left( i_{1}\right)
},\ldots ,x^{\left( i_{n}\right) }\right) $. Therefore $W=V$ and our
correspondence is one to one.
\end{proof}

From Propositions \ref{automorphismbijections} and \ref{wordsbijections} we
conclude the

\begin{theorem}
\label{methodofverbaloperations}There is one to one and onto correspondence $%
\Phi \rightarrow W^{\Phi }$ between the family of the all strongly stable
automorphisms of the category $\Theta ^{0}$ and the family of the all
systems of words which fulfill the conditions Op1) and Op2). The systems of
words $W^{\Phi }=W$ defined by formula (\ref{wordformula}) where bijections $%
s_{A_{\omega }}$ are subjects of the items A2) - A4) of the Definition \ref%
{str_stab_aut}.
\end{theorem}

We will denote by $W\rightarrow \Phi ^{W}$ the correspondence which inverse
to the correspondence $\Phi \rightarrow W^{\Phi }$.

\subsection{Strongly stable automorphisms and inner automorphisms.}

\begin{proposition}
\label{intersectioncriterion}The strongly stable automorphism $\Phi $ which
corresponds to the system of words $W^{\Phi }=W$ is inner if and only if
there is a system of isomorphisms $\left\{ \tau _{F}:F\rightarrow
F_{W}^{\ast }\mid F\in \mathrm{Ob}\Theta ^{0}\right\} $ such that for every $%
A,B\in \mathrm{Ob}\Theta ^{0}$ and every $\mu \in \mathrm{Mor}_{\Theta
^{0}}\left( A,B\right) $ the 
\begin{equation}
\tau _{B}\mu =\mu \tau _{A}  \label{comm}
\end{equation}%
holds.
\end{proposition}

\begin{proof}
If automorphism $\Phi $ is inner then by Definition \ref{inner} there is a
system of isomorphisms $\left\{ \sigma _{F}:F\rightarrow F\mid F\in \mathrm{%
Ob}\Theta ^{0}\right\} $ such that for every $A,B\in \mathrm{Ob}\Theta ^{0}$
and every $\mu \in \mathrm{Mor}_{\Theta ^{0}}\left( A,B\right) $ the $\Phi
\left( \mu \right) =\sigma _{B}\mu \sigma _{A}^{-1}=s_{B}\mu s_{A}^{-1}$
holds, where bijections $s_{A},s_{B}$ are subjects of the items A2) - A4) of
the Definition \ref{str_stab_aut}. So the $\mu \sigma _{A}^{-1}s_{A}=\sigma
_{B}^{-1}s_{B}\mu $ holds. By Proposition \ref{*isom} $\tau _{A}=\sigma
_{A}^{-1}s_{A}:A\rightarrow A_{W}^{\ast }$ is an isomorphism.

Now we assume that $\Phi $ is a strongly stable automorphism and (\ref{comm}%
) holds. $s_{A},s_{B}$ are subjects of the items A2) - A4) of the Definition %
\ref{str_stab_aut}. From (\ref{comm}) we conclude $\Phi \left( \mu \right)
=s_{B}\mu s_{A}^{-1}=s_{B}\tau _{B}^{-1}\mu \tau _{A}s_{A}^{-1}$. By
Proposition \ref{*isom} $\tau _{A}s_{A}^{-1}:A_{W}^{\ast }\rightarrow
A_{W}^{\ast }$ is an isomorphism. So, by Corollary \ref{invershomomorphism}
from Proposition \ref{inverswords} $\tau _{A}s_{A}^{-1}:A\rightarrow A$ is
an isomorphism. Hence $\sigma _{A}=\left( \tau _{A}s_{A}^{-1}\right)
^{-1}:A\rightarrow A$ is an isomorphism and $\Phi \left( \mu \right) =\sigma
_{B}\mu \sigma _{A}^{-1}$.
\end{proof}

\section{Application to the universal algebraic geometry of many-sorted
algebras.\label{alggeometry}}

The basic notion of the universal algebraic geometry defined in \cite%
{PlotkinVarCat}, \cite{PlotkinNotions} and \cite{PlotkinSame}. These notion
can by defined for many-sorted algebras. The universal algebraic geometry of
same classes of many-sorted algebras was considered in \cite{Repofgr} and 
\cite{ShestTsur}. We can consider the universal algebraic geometry over the
arbitrary variety of algebras $\Theta $. Ours equations will be pairs of the
elements of the free algebras of the variety $\Theta $. It is natural that
in the case of many-sorted algebras we compare elements of the same sort, so
for the equation $\left( t_{1},t_{2}\right) \in F^{2}$, where $F\in \mathrm{%
Ob}\Theta ^{0}$, the $\left( t_{1},t_{2}\right) \in \left( F^{\left(
i\right) }\right) ^{2}$ holds, where $i\in \Gamma $. Therefore, $T\subseteq
\dbigcup\limits_{i\in \Gamma }\left( \left( F\right) ^{\left( i\right)
}\right) ^{2}$ holds for every system of equations $T$. Solutions of this
system of equations we can find in arbitrary algebra $H\in \Theta $. It will
be a homomorphisms $\varphi \in \mathrm{Hom}\left( F,H\right) $, such that $%
\varphi \left( t_{1}\right) =\varphi \left( t_{2}\right) $ holds for every $%
\left( t_{1},t_{2}\right) \in T$. We denote by $T_{H}^{\prime }$ the set of
the all solutions of the system of equations $T$ in the algebra $H$. The $%
T_{H}^{\prime }=\left\{ \varphi \in \mathrm{Hom}\left( F,H\right) \mid
T\subseteq \ker \varphi \right\} $ holds. Also we can consider the algebraic
closure of the system $T$ over the algebra $H$: $T_{H}^{\prime \prime
}=\bigcap\limits_{\varphi \in T_{H}^{\prime }}\ker \varphi =\bigcap\limits 
_{\substack{ \varphi \in \mathrm{Hom}\left( F,H\right) ,  \\ T\subseteq \ker
\varphi }}\ker \varphi $. $T_{H}^{\prime \prime }$ is a set of all equations
which we can conclude from the system $T$ in the algebra $H$ or the maximal
system of equations which has in the algebra $H$ same solutions as the
system $T$. It is clear that algebraic closure of arbitrary system $%
T\subseteq \dbigcup\limits_{i\in \Gamma }\left( \left( F\right) ^{\left(
i\right) }\right) ^{2}$ is a congruence and also the $T_{H}^{\prime \prime
}\subseteq \dbigcup\limits_{i\in \Gamma }\left( \left( F\right) ^{\left(
i\right) }\right) ^{2}$ must fulfill. If $\mathrm{Hom}\left( F,H\right)
=\varnothing $, then for every $T\subseteq \dbigcup\limits_{i\in \Gamma
}\left( \left( F\right) ^{\left( i\right) }\right) ^{2}$ the $T_{H}^{\prime
}=\varnothing $ holds and $T_{H}^{\prime \prime }=\bigcap\limits_{\varphi
\in T_{H}^{\prime }}\ker \varphi =\dbigcup\limits_{i\in \Gamma }\left(
\left( F\right) ^{\left( i\right) }\right) ^{2}$ fulfills by the principle
of the empty set.

\begin{definition}
\cite{PlotkinVarCat}The set $T\subseteq \dbigcup\limits_{i\in \Gamma }\left(
\left( F\right) ^{\left( i\right) }\right) ^{2}$ is called \textbf{%
algebraically closed} over the algebra $H$ if $T_{H}^{\prime \prime }=T$.
\end{definition}

If $F\in \mathrm{Ob}\Theta ^{0}$ then the family of the subsets of $%
\dbigcup\limits_{i\in \Gamma }\left( \left( F\right) ^{\left( i\right)
}\right) ^{2}$ which are algebraically closed over the algebra $H\in \Theta $
we denote by $Cl_{H}(F)$. The minimal subset in the $Cl_{H}(F)$ is $\left(
\Delta _{F}\right) _{H}^{\prime \prime }=\bigcap\limits_{\varphi \in \mathrm{%
Hom}\left( F,H\right) }\ker \varphi $.$\left( {}\right) $

\begin{definition}
\cite{PlotkinVarCat}We say that \textit{algebras }$H_{1},H_{2}\in \Theta $%
\textit{\ are \textbf{geometrically equivalent}} if for every $F\in \mathrm{%
Ob}\Theta ^{0}$ and every $T\subseteq \dbigcup\limits_{i\in \Gamma }\left(
\left( F\right) ^{\left( i\right) }\right) ^{2}$ the $T_{H_{1}}^{\prime
\prime }=T_{H_{2}}^{\prime \prime }$.
\end{definition}

It is clear that algebras $H_{1},H_{2}$\ are geometrically equivalent if and
only if the $Cl_{H_{1}}(F)=Cl_{H_{2}}(F)$ holds for every $F\in \mathrm{Ob}%
\Theta ^{0}$.

\subsection{Automorphic equivalence of many-sorted algebras.\label%
{Automorphic_equivalence}}

\begin{definition}
\label{Autom_equiv}\cite{PlotkinSame}We say that \textit{algebras }$%
H_{1},H_{2}\in \Theta $\textit{\ are \textbf{automorphically equivalent} if
there exist an automorphism }$\Phi :\Theta ^{0}\rightarrow \Theta ^{0}$%
\textit{\ and the bijections}%
\begin{equation*}
\alpha (\Phi )_{A}:Cl_{H_{1}}(A)\rightarrow Cl_{H_{2}}(\Phi (A))
\end{equation*}%
for every $A\in \mathrm{Ob}\Theta ^{0}$, \textit{coordinated in the
following sense: if }$A_{1},A_{2}\in \mathrm{Ob}\Theta ^{0}$\textit{, }$\mu
_{1},\mu _{2}\in \mathrm{Hom}\left( A_{1},A_{2}\right) $\textit{, }$T\in
Cl_{H_{1}}(A_{2})$\textit{\ then}%
\begin{equation*}
\tau \mu _{1}=\tau \mu _{2},
\end{equation*}%
\textit{if and only if }%
\begin{equation*}
\widetilde{\tau }\Phi \left( \mu _{1}\right) =\widetilde{\tau }\Phi \left(
\mu _{2}\right) ,
\end{equation*}%
\textit{where }$\tau :A_{2}\rightarrow A_{2}/T$\textit{, }$\widetilde{\tau }%
:\Phi \left( A_{2}\right) \rightarrow \Phi \left( A_{2}\right) /\alpha (\Phi
)_{A_{2}}\left( T\right) $\textit{\ are the natural epimorphisms.}
\end{definition}

If $A,B\in \mathrm{Ob}\Theta ^{0}$, $f:A\rightarrow B$ is a mapping, $%
T\subseteq A^{2}$ then we denote $f\left( T\right) =\left\{ \left( f\left(
t_{1}\right) ,f\left( t_{2}\right) \right) \mid \left( t_{1},t_{2}\right)
\in T\right\} $.

This result is a refinement of \cite[Proposition 8]{PlotkinSame}.

\begin{theorem}
\label{alfaformula}If \textit{an automorphism }$\Phi :\Theta ^{0}\rightarrow
\Theta ^{0}$ provides an automorphic equivalence of algebras $H_{1},H_{2}\in
\Theta $ and $\Phi $ acts on the morphisms of $\Theta ^{0}$ by formula (\ref%
{pot_inner}) with the system of bijections $\left\{ s_{A}:A\rightarrow \Phi
\left( A\right) \mid A\in \mathrm{Ob}\Theta ^{0}\right\} $, then $\alpha
(\Phi )_{A}\left( T\right) =s_{A}\left( T\right) $ holds for every $T\in
Cl_{H_{1}}(A)$.

Vice versa, if \textit{an automorphism }$\Phi :\Theta ^{0}\rightarrow \Theta
^{0}$ acts on the morphisms of $\Theta ^{0}$ by formula (\ref{pot_inner})
and for every $A\in \mathrm{Ob}\Theta ^{0}$ the $s_{A}:Cl_{H_{1}}(A)%
\rightarrow Cl_{H_{2}}(\Phi (A))$ is a bijection then the \textit{%
automorphism }$\Phi :\Theta ^{0}\rightarrow \Theta ^{0}$ provides an
automorphic equivalence of algebras $H_{1},H_{2}\in \Theta $ with $\alpha
(\Phi )_{A}=s_{A}$.
\end{theorem}

\begin{proof}
We assume that $T\in Cl_{H_{1}}(A)$. We suppose that $\left(
t_{1},t_{2}\right) \in \left( A^{\left( i\right) }\right) ^{2}\cap T$, where 
$i\in \Gamma $. We consider $F=F\left( x^{\left( i\right) }\right) \in 
\mathrm{Ob}\Theta ^{0}$ and two homomorphisms $\mu _{1},\mu
_{2}:F\rightarrow A$, such that $\mu _{j}\left( x^{\left( i\right) }\right)
=t_{j}$, $j=1,2$. $T$ is a congruence, so we can prove by induction by the
length of the terms in the algebra $F$, that $\left( \mu _{1}\left( f\right)
,\mu _{2}\left( f\right) \right) \in T$ holds for every $f\in F$. Therefore $%
\tau \mu _{1}=\tau \mu _{2}$ and $\widetilde{\tau }\Phi \left( \mu
_{1}\right) =\widetilde{\tau }\Phi \left( \mu _{2}\right) $, where $\tau
:A\rightarrow A/T$\textit{, }$\widetilde{\tau }:\Phi \left( A\right)
\rightarrow \Phi \left( A\right) /\alpha (\Phi )_{A}\left( T\right) $\textit{%
\ }are the natural epimorphisms. So $\left( s_{A}\mu _{1}s_{F}^{-1}\left(
u\right) ,s_{A}\mu _{2}s_{F}^{-1}\left( u\right) \right) \in \alpha (\Phi
)_{A}\left( T\right) $ holds for every $u\in \Phi \left( F\right) $. We take 
$u=s_{F}\left( x^{\left( i\right) }\right) \in \Phi \left( F\right) $ and
conclude that $\left( s_{A}\left( t_{1}\right) ,s_{A}\left( t_{2}\right)
\right) \in \alpha (\Phi )_{A}\left( T\right) $. So $s_{A}\left( T\right)
\subseteq \alpha (\Phi )_{A}\left( T\right) $.

Now we suppose that $\left( u_{1},u_{2}\right) \in \left( \left( \Phi \left(
A\right) \right) ^{\left( i\right) }\right) ^{2}\cap \alpha (\Phi
)_{A}\left( T\right) $, $i\in \Gamma $. We consider $F=F\left( x^{\left(
i\right) }\right) \in \mathrm{Ob}\Theta ^{0}$ and two homomorphisms $\varphi
_{1},\varphi _{2}:F\rightarrow \Phi \left( A\right) $, such that $\varphi
_{j}\left( x^{\left( i\right) }\right) =u_{j}$, $j=1,2$. By proof of the
Theorem \ref{isom} there is an isomorphism $\sigma :F\rightarrow \Phi \left(
F\right) $, such that $\sigma \left( x^{\left( i\right) }\right)
=s_{F}\left( x^{\left( i\right) }\right) $. We denote $\nu _{j}=\varphi
_{j}\sigma ^{-1}$, $j=1,2$. $\nu _{j}$ are homomorphisms from $\Phi \left(
F\right) $ to $\Phi \left( A\right) $. $\nu _{j}\left( \sigma \left(
x^{\left( i\right) }\right) \right) =u_{j}$, $j=1,2$. $\alpha (\Phi
)_{A}\left( T\right) $ is a congruence, $\sigma \left( x^{\left( i\right)
}\right) $ is a free generator of $\Phi \left( F\right) $, so we can prove
by induction by the length of the terms in the algebra $F$ that $\widetilde{%
\tau }\nu _{1}=\widetilde{\tau }\nu _{2}$. $\nu _{j}=\Phi \left( \mu
_{j}\right) $, where $j=1,2$ and $\mu _{j}$ are homomorphisms from $F$ to $A$%
. Hence, by (\ref{pot_inner}), $s_{A}\mu _{j}\left( x^{\left( i\right)
}\right) =\nu _{j}s_{F}\left( x^{\left( i\right) }\right) =\varphi
_{j}\sigma ^{-1}s_{F}\left( x^{\left( i\right) }\right) =\varphi _{j}\left(
x^{\left( i\right) }\right) =u_{j}$, $j=1,2$. We denote $t_{j}=\mu
_{j}\left( x^{\left( i\right) }\right) $, $j=1,2$. By Definition \ref%
{Autom_equiv} we have that $\tau \mu _{1}=\tau \mu _{2}$, so $\left(
t_{1},t_{2}\right) \in T$ and $\left( u_{1},u_{2}\right) \in s_{A}\left(
T\right) $. Hence, $\alpha (\Phi )_{A}\left( T\right) \subseteq s_{A}\left(
T\right) $.

Now we prove the second statement of the Theorem. We assume that $%
A_{1},A_{2}\in \mathrm{Ob}\Theta ^{0}$\textit{, }$\mu _{1},\mu _{2}\in 
\mathrm{Hom}\left( A_{1},A_{2}\right) $\textit{, }$T\in Cl_{H_{1}}(A_{2})$
and $\tau :A_{2}\rightarrow A_{2}/T$, $\widetilde{\tau }:\Phi \left(
A_{2}\right) \rightarrow \Phi \left( A_{2}\right) /s_{A_{2}}\left( T\right) $%
\ are the natural epimorphisms. If $\tau \mu _{1}=\tau \mu _{2}$ holds, then
for every $a\in A_{1}$ the $\left( \mu _{1}\left( a\right) ,\mu _{2}\left(
a\right) \right) \in T$ holds. Therefore we have for every $u\in \Phi \left(
A_{1}\right) $ that $\left( \left( \Phi \left( \mu _{1}\right) \right)
\left( u\right) ,\left( \Phi \left( \mu _{2}\right) \right) \left( u\right)
\right) =\left( s_{A_{2}}\mu _{1}s_{A_{1}}^{-1}\left( u\right) ,s_{A_{2}}\mu
_{2}s_{A_{1}}^{-1}\left( u\right) \right) \in s_{A_{2}}\left( T\right) $
holds. So $\widetilde{\tau }\Phi \left( \mu _{1}\right) =\widetilde{\tau }%
\Phi \left( \mu _{2}\right) $.

If $\widetilde{\tau }\Phi \left( \mu _{1}\right) =\widetilde{\tau }\Phi
\left( \mu _{2}\right) $, then for every $u\in \Phi \left( A_{1}\right) $
the $\left( \left( \Phi \left( \mu _{1}\right) \right) \left( u\right)
,\left( \Phi \left( \mu _{2}\right) \right) \left( u\right) \right) =\left(
s_{A_{2}}\mu _{1}s_{A_{1}}^{-1}\left( u\right) ,s_{A_{2}}\mu
_{2}s_{A_{1}}^{-1}\left( u\right) \right) \in s_{A_{2}}\left( T\right) $.
Because $s_{A_{1}}$ is a bijection, we have that for every $a\in A_{1}$ the $%
\left( \mu _{1}\left( a\right) ,\mu _{2}\left( a\right) \right) \in T$ holds
and $\tau \mu _{1}=\tau \mu _{2}$.
\end{proof}

\begin{corollary}
\label{equivalence}Automorphic equivalence of algebras in $\Theta $ is
equivalence.
\end{corollary}

\begin{proof}
We need to prove that automorphic equivalence is reflexive, symmetric and
reflexive relation. The identity automorphism provides automorphic
equivalence of the algebra $H\in \Theta $ to itself, so this relation is
reflexive.

We assume that the automorphism $\Phi \in \mathfrak{A}$ provides automorphic
equivalence of the algebra $H_{1}$ to the algebra $H_{2}$, where $%
H_{1},H_{2}\in \Theta $, and acts on the morphisms of $\Theta ^{0}$ by
formula (\ref{pot_inner}) with the system of bijections \linebreak $\left\{
s_{A}^{\Phi }:A\rightarrow \Phi \left( A\right) \mid A\in \mathrm{Ob}\Theta
^{0}\right\} $. Then by Theorem \ref{alfaformula} $s_{\Phi ^{-1}\left(
A\right) }^{\Phi }:Cl_{H_{1}}\Phi ^{-1}(A)\rightarrow Cl_{H_{2}}A$ is a
bijection. So $\left( s_{\Phi ^{-1}\left( A\right) }^{\Phi }\right)
^{-1}:Cl_{H_{2}}A\rightarrow Cl_{H_{1}}\Phi ^{-1}(A)$ is also a bijection.
But we can conclude from consideration of the formula (\ref{pot_inner}) that
automorphism $\Phi ^{-1}$ acts on the morphisms of $\Theta ^{0}$ by formula (%
\ref{pot_inner}) with the system of bijections$\left\{ s_{A}^{\Phi
^{-1}}=\left( s_{\Phi ^{-1}\left( A\right) }^{\Phi }\right)
^{-1}:A\rightarrow \Phi ^{-1}\left( A\right) \mid A\in \mathrm{Ob}\Theta
^{0}\right\} $. Therefore the automorphism $\Phi ^{-1}$ provides automorphic
equivalence of the algebra $H_{2}$ to the algebra $H_{1}$.

Now we assume that the automorphism $\Phi \in \mathfrak{A}$ provides
automorphic equivalence of the algebra $H_{1}$ to the algebra $H_{2}$ and
the automorphism $\Psi \in \mathfrak{A}$ provides automorphic equivalence of
the algebra $H_{2}$ to the algebra $H_{3}$, where $H_{1},H_{2},H_{3}\in
\Theta $. We by $\left\{ s_{A}^{\Phi }\right\} $ and $\left\{ s_{A}^{\Psi
}\right\} $ the systems of bijections which present by formula (\ref%
{pot_inner}) the acting of automorphisms $\Phi $ and $\Psi $ correspondingly
on the morphisms of $\Theta ^{0}$. By Theorem \ref{alfaformula} $s_{A}^{\Phi
}:Cl_{H_{1}}A\rightarrow Cl_{H_{2}}\Phi \left( A\right) $ and $s_{\Phi
\left( A\right) }^{\Psi }:Cl_{H_{2}}\Phi \left( A\right) \rightarrow
Cl_{H_{3}}\Psi \Phi \left( A\right) $ are bijections. So $s_{\Phi \left(
A\right) }^{\Psi }s_{A}^{\Phi }:Cl_{H_{1}}A\rightarrow Cl_{H_{3}}\Psi \Phi
\left( A\right) $ is also a bijection. But we can conclude from
consideration of the formula (\ref{pot_inner}) that automorphism $\Psi \Phi $
acts on the morphisms of $\Theta ^{0}$ by formula (\ref{pot_inner}) with the
system of bijections $\left\{ s_{A}^{\Psi \Phi }=s_{\Phi \left( A\right)
}^{\Psi }s_{A}^{\Phi }:A\rightarrow \Psi \Phi \left( A\right) \mid A\in 
\mathrm{Ob}\Theta ^{0}\right\} $. Therefore the automorphism $\Psi \Phi $
provides automorphic equivalence of the algebra $H_{1}$ to the algebra $%
H_{3} $.
\end{proof}

\setcounter{corollary}{0}

\begin{definition}
We say that the system of bijections $\left\{ c_{F}:F\rightarrow F\mid F\in 
\mathrm{Ob}\Theta ^{0}\right\} $ is a \textbf{central function} if $c_{B}\mu
=\mu c_{A}$ holds for every $\mu \in \mathrm{Mor}_{\Theta ^{0}}\left(
A,B\right) $.
\end{definition}

\begin{proposition}
\label{centrarfunction}If $A\in \mathrm{Ob}\Theta ^{0}$, $T\subseteq
\dbigcup\limits_{i\in \Gamma }\left( \left( A\right) ^{\left( i\right)
}\right) ^{2}$ is a congruence and\linebreak $\left\{ c_{F}:F\rightarrow
F\mid F\in \mathrm{Ob}\Theta ^{0}\right\} $ is a central function then $%
c_{A}\left( T\right) =T$.
\end{proposition}

\begin{proof}
We assume that $\left( a_{1},a_{2}\right) \in T\cap \left( \left( A\right)
^{\left( i\right) }\right) ^{2}$, where $i\in \Gamma $. We consider $%
F=F\left( x^{\left( i\right) }\right) \in \mathrm{Ob}\Theta ^{0}$ and two
homomorphisms $\mu _{j}:F\rightarrow A$, $j=1,2$, such that $\mu _{j}\left(
x^{\left( i\right) }\right) =a_{j}$. We can prove by induction by the length
of the terms in the algebra $F$, that $\left( \mu _{1}\left( f\right) ,\mu
_{2}\left( f\right) \right) \in T$ holds for every $f\in F$. The $%
c_{A}\left( a_{j}\right) =c_{A}\mu _{j}\left( x^{\left( i\right) }\right)
=\mu _{j}c_{F}\left( x^{\left( i\right) }\right) $ holds for $j=1,2$.
Therefore $\left( c_{A}\left( a_{1}\right) ,c_{A}\left( a_{2}\right) \right)
\in T$ and $c_{A}\left( T\right) \subseteq T$.

If $\left\{ c_{F}:F\rightarrow F\mid F\in \mathrm{Ob}\Theta ^{0}\right\} $
is a central function then $\left\{ c_{F}^{-1}:F\rightarrow F\mid F\in 
\mathrm{Ob}\Theta ^{0}\right\} $ is also a central function, so $%
c_{A}^{-1}\left( T\right) \subseteq T$.
\end{proof}

\begin{corollary}
\label{independence}The bijections $\left\{ \alpha (\Phi )_{A}\mid A\in 
\mathrm{Ob}\Theta ^{0}\right\} $ (see Definition \ref{Autom_equiv}) are
depend only from the automorphism $\Phi $ and not from the presentation of
the action of it automorphism over morphisms by bijections.
\end{corollary}

\setcounter{corollary}{0}

\begin{theorem}
\label{innerautomprop}(see \cite[Proposition 9]{PlotkinSame}) If an inner
automorphism $\Upsilon $ provides the automorphic equivalence of the
algebras $H_{1}$ and $H_{2}$, where $H_{1},H_{2}\in \Theta $, then $H_{1}$
and $H_{2}$ are geometrically equivalent.
\end{theorem}

\begin{proof}
By Definition \ref{inner} the automorphism $\Upsilon $ acts on the morphisms
of $\Theta ^{0}$ by system of isomorphisms $\left\{ \sigma _{A}:A\rightarrow
\Phi \left( A\right) \mid A\in \mathrm{Ob}\Theta ^{0}\right\} $. By Theorem %
\ref{alfaformula} $\sigma _{A}:Cl_{H_{1}}A\rightarrow Cl_{H_{2}}\Phi \left(
A\right) $ is a bijection for $A\in \mathrm{Ob}\Theta ^{0}$. We take the set 
$T\in Cl_{H_{1}}A$. $\sigma _{A}\left( T\right) \in Cl_{H_{2}}\Phi \left(
A\right) $, so $\sigma _{A}\left( T\right) =\bigcap\limits_{\substack{ %
\varphi \in \mathrm{Hom}\left( \Phi \left( A\right) ,H_{2}\right) ,  \\ %
\sigma _{A}\left( T\right) \subseteq \ker \varphi }}\ker \varphi $.
Therefore 
\begin{equation*}
T=\sigma _{A}^{-1}\sigma _{A}\left( T\right) =\bigcap\limits_{\substack{ %
\varphi \in \mathrm{Hom}\left( \Phi \left( A\right) ,H_{2}\right) ,  \\ %
\sigma _{A}\left( T\right) \subseteq \ker \varphi }}\sigma _{A}^{-1}\ker
\varphi =\bigcap\limits_{\substack{ \varphi \in \mathrm{Hom}\left( \Phi
\left( A\right) ,H_{2}\right) ,  \\ \sigma _{A}\left( T\right) \subseteq
\ker \varphi }}\ker \varphi \sigma _{A}.
\end{equation*}%
If $\sigma _{A}\left( T\right) \subseteq \ker \varphi $ then $T\subseteq
\ker \varphi \sigma _{A}$, so $T\supseteq \bigcap\limits_{\substack{ \psi
\in \mathrm{Hom}\left( A,H_{2}\right) ,  \\ T\subseteq \ker \psi }}\ker \psi
=T_{H_{2}}^{\prime \prime }$ and $T\in Cl_{H_{2}}A$. Therefore $%
Cl_{H_{1}}A\subseteq Cl_{H_{2}}A$.

Hear we must remark that if $\left\{ \varphi \in \mathrm{Hom}\left( \Phi
\left( A\right) ,H_{2}\right) \mid \sigma _{A}\left( T\right) \subseteq \ker
\varphi \right\} =\varnothing $, in particular, because $\Gamma _{\Phi
\left( A\right) }\nsubseteq \Gamma _{H_{2}}$, and $\sigma _{A}\left(
T\right) =\dbigcup\limits_{i\in \Gamma }\left( \left( \Phi \left( A\right)
\right) ^{\left( i\right) }\right) ^{2}\in Cl_{H_{2}}\Phi \left( A\right) $
then $\left\{ \psi \in \mathrm{Hom}\left( A,H_{2}\right) \mid T\subseteq
\ker \psi \right\} =\varnothing $, in particular, because $\Gamma
_{A}=\Gamma _{\Phi \left( A\right) }\nsubseteq \Gamma _{H_{2}}$, and $%
T=\dbigcup\limits_{i\in \Gamma }\left( A^{\left( i\right) }\right) ^{2}\in
Cl_{H_{2}}A$.

By Corollary \ref{equivalence} from the Theorem \ref{alfaformula} the
automorphism $\Upsilon ^{-1}$ provides the automorphic equivalence of the
algebras $H_{2}$ and $H_{1}$ and by similar consideration we conclude that $%
Cl_{H_{2}}A\subseteq Cl_{H_{1}}A$.
\end{proof}

From this Theorem we conclude that the possible difference between the
automorphic equivalence and geometric equivalence is measured by quotient
group $\mathfrak{A/Y\cong S/S\cap Y}$. The elements of the group $\mathfrak{S%
}$ we can find by using of the Theorem \ref{methodofverbaloperations}. The
elements of the group $\mathfrak{S\cap Y}$ we can find by using of the
Proposition \ref{intersectioncriterion}.

From here to the end of the Subsection we assume the system of words $%
W=\left\{ w_{\omega }\mid \omega \in \Omega \right\} $ is a subject of
conditions Op1) and Op2), $\left\{ s_{F}\mid F\in \mathrm{Ob}\Theta
^{0}\right\} $ is a system of bijection subject of condition Op2), $\Phi \in 
\mathfrak{S}$ corresponds to the system of words $W$.

\begin{proposition}
\label{sbijection}For every $H\in \Theta $ and every $F\in \mathrm{Ob}\Theta
^{0}$ the mapping $s_{F}^{-1}:Cl_{H}\left( F\right) \rightarrow
Cl_{H_{W}^{\ast }}\left( F\right) $ is a bijection.
\end{proposition}

\begin{proof}
Algebras $H$ and $H_{W}^{\ast }$ have different operations but same sets of
elements, so $\Gamma _{H}=\Gamma _{H_{W}^{\ast }}$. If $\Gamma
_{H}\nsupseteq \Gamma _{F}$ then $Cl_{H}\left( F\right) =Cl_{H_{W}^{\ast
}}\left( F\right) =\left\{ \dbigcup\limits_{i\in \Gamma _{F}}\left(
F^{\left( i\right) }\right) ^{2}\right\} $ and $s_{F}^{-1}\left(
\dbigcup\limits_{i\in \Gamma _{F}}\left( F^{\left( i\right) }\right)
^{2}\right) =\dbigcup\limits_{i\in \Gamma _{F}}\left( F^{\left( i\right)
}\right) ^{2}$, because $s_{F}^{-1}$ is a bijection which conforms with the $%
\eta _{F}$. So in this case the Proposition is proved.

Now we consider the case when $\Gamma _{H}\supseteq \Gamma _{F}$ and $%
\mathrm{Hom}\left( F,H\right) \neq \varnothing $. In this case also $\mathrm{%
Hom}\left( F,H_{W}^{\ast }\right) \neq \varnothing $. We consider the diagram%
\begin{equation*}
\begin{array}{ccc}
F & \underset{s_{F}}{\rightarrow } & F \\ 
\downarrow \psi &  & \varphi \downarrow \\ 
H_{W}^{\ast } &  & H%
\end{array}%
.
\end{equation*}%
$s_{F}:F\rightarrow F_{W}^{\ast }$ is an isomorphism. If $\varphi \in 
\mathrm{Hom}\left( F,H\right) $ then by Corollary \ref{invershomomorphism}
from Proposition \ref{inverswords} $\varphi \in \mathrm{Hom}\left(
F_{W}^{\ast },H_{W}^{\ast }\right) $ and $\varphi s_{F}\in \mathrm{Hom}%
\left( F,H_{W}^{\ast }\right) $. If $\psi \in \mathrm{Hom}\left(
F,H_{W}^{\ast }\right) $ then $\psi s_{F}^{-1}\in \mathrm{Hom}\left(
F_{W}^{\ast },H_{W}^{\ast }\right) $ and also by Corollary \ref%
{invershomomorphism} from Proposition \ref{inverswords} $\psi s_{F}^{-1}\in 
\mathrm{Hom}\left( F,H\right) $.

If $T\in Cl_{H}\left( F\right) $, then $T=\bigcap\limits_{\varphi \in
T_{H}^{\prime }}\ker \varphi $. If $\varphi \in T_{H}^{\prime }$ then $%
\varphi \in \mathrm{Hom}\left( F,H\right) $, $\ker \varphi \supseteq T$. So $%
\varphi s_{F}\in \mathrm{Hom}\left( F,H_{W}^{\ast }\right) $ and $\ker
\varphi s_{F}=s_{F}^{-1}\ker \varphi \supseteq s_{F}^{-1}T$. Hence $\varphi
s_{F}\in \left( s_{F}^{-1}T\right) _{H_{W}^{\ast }}^{\prime }$. Therefore $%
\bigcap\limits_{\psi \in \left( s_{F}^{-1}T\right) _{H_{W}^{\ast }}^{\prime
}}\ker \psi \subseteq \bigcap\limits_{\varphi \in T_{H}^{\prime }}\ker
\varphi s_{F}=s_{F}^{-1}\left( \bigcap\limits_{\varphi \in T_{H}^{\prime
}}\ker \varphi \right) =s_{F}^{-1}T$. So $s_{F}^{-1}T\in Cl_{H_{W}^{\ast
}}\left( F\right) $. Therefore $s_{F}^{-1}$ is a mapping from $Cl_{H}\left(
F\right) $ to $Cl_{H_{W}^{\ast }}\left( F\right) $.

We need to prove that $s_{F}^{-1}$ is a bijection. For this purpose we will
prove that $s_{F}$ is a mapping from $Cl_{H_{W}^{\ast }}\left( F\right) $ to 
$Cl_{H}\left( F\right) $. If $T\in Cl_{H_{W}^{\ast }}\left( F\right) $ then $%
T=\bigcap\limits_{\psi \in T_{H_{W}^{\ast }}^{\prime }}\ker \psi $. If $\psi
\in T_{H_{W}^{\ast }}^{\prime }$ then $\psi \in \mathrm{Hom}\left(
F,H_{W}^{\ast }\right) $, $\ker \psi \supseteq T$. So $\psi s_{F}^{-1}\in 
\mathrm{Hom}\left( F,H\right) $ and $\ker \psi s_{F}^{-1}=s_{F}\ker \psi
\supseteq s_{F}T$. Hence $\psi s_{F}^{-1}\in \left( s_{F}T\right)
_{H}^{\prime }$. Therefore $\bigcap\limits_{\varphi \in \left( s_{F}T\right)
_{H}^{\prime }}\ker \varphi \subseteq \bigcap\limits_{\psi \in
T_{H_{W}^{\ast }}^{\prime }}\ker \psi s_{F}^{-1}=s_{F}\left(
\bigcap\limits_{\psi \in T_{H_{W}^{\ast }}^{\prime }}\ker \psi \right)
=s_{F}T$. So $s_{F}T\in Cl_{H}\left( F\right) $. Therefore $s_{F}$ is a
mapping from $Cl_{H_{W}^{\ast }}\left( F\right) $ to $Cl_{H}\left( F\right) $%
. $s_{F}$ is an inverse mapping for the $s_{F}^{-1}$, so both of them are
bijections.
\end{proof}

\begin{corollary}
\label{hwautomequiv}For every $H\in \Theta $ the automorphism $\Phi ^{-1}$
provides an automorphic equivalence of algebras $H$ and $H_{W}^{\ast }$.
\end{corollary}

\begin{proof}
By Theorem \ref{alfaformula}.
\end{proof}

\setcounter{corollary}{0}

\begin{theorem}
\label{reduction}Algebras $H_{1},H_{2}\in \Theta $ are automorphically
equivalent if and only if there is a system of words $W=\left\{ w_{\omega
}\mid \omega \in \Omega \right\} $ subject of conditions Op1) and Op2), such
that algebras $H_{1}$ and $\left( H_{2}\right) _{W}^{\ast }$ are
geometrically equivalent.
\end{theorem}

\begin{proof}
We assume that the automorphism $\Psi \in \mathfrak{A}$ provides an
automorphic equivalence of algebras $H_{1}$ and $H_{2}$. By Theorem \ref%
{factorisation} we can present the automorphism $\Psi $ in the form $\Psi
=\Phi \Upsilon $, where $\Phi \in \mathfrak{S}$, $\Upsilon \in \mathfrak{Y}$%
. Therefore $\Upsilon =\Phi ^{-1}\Psi $. By Corollary \ref{hwautomequiv}
from the Proposition \ref{sbijection} the automorphism $\Phi ^{-1}$ provides
an automorphic equivalence of algebras $H_{2}$ and $\left( H_{2}\right)
_{W}^{\ast }$, where $W=\left\{ w_{\omega }\mid \omega \in \Omega \right\} $
is the system of the words, which is a subject of conditions Op1) and Op2)
and corresponds to the automorphism $\Phi $ by the Theorem \ref%
{methodofverbaloperations}. By the Proof of the Corollary \ref{equivalence}
from the Theorem \ref{alfaformula} we have that the inner automorphism $%
\Upsilon =\Phi ^{-1}\Psi $ provides an automorphic equivalesnce of algebras $%
H_{1}$ and $\left( H_{2}\right) _{W}^{\ast }$. By the Theorem \ref%
{innerautomprop} we conclude that algebras $H_{1}$ and $\left( H_{2}\right)
_{W}^{\ast }$ are geometrically equivalent.

Now we assume that there is a system of words $W=\left\{ w_{\omega }\mid
\omega \in \Omega \right\} $ subject of conditions Op1) and Op2), such that
algebras $H_{1}$ and $\left( H_{2}\right) _{W}^{\ast }$ are geometrically
equivalent. By definition for every $F\in \mathrm{Ob}\Theta ^{0}$ there is
system a bijection $id_{F}:Cl_{H_{1}}\left( F\right) \rightarrow Cl_{\left(
H_{2}\right) _{W}^{\ast }}\left( F\right) $. The identity automorphism of
the category $\Theta ^{0}$ acts on the morphisms of $\Theta ^{0}$ by formula
(\ref{pot_inner}) with bijections $\left\{ id_{F}:F\rightarrow F\mid F\in 
\mathrm{Ob}\Theta ^{0}\right\} $. So, by Theorem \ref{alfaformula}, the
identity automorphism of the category $\Theta ^{0}$ provides an automorphic
equivalence of algebras $H_{1}$ and $\left( H_{2}\right) _{W}^{\ast }$. We
denote by $\Phi $ the strongly stable automorphism which corresponds to the
system of words $W$ by the Theorem \ref{methodofverbaloperations}. From
Corollary \ref{hwautomequiv} from the Proposition \ref{sbijection} and from
the Proof of the Corollary \ref{equivalence} from the Theorem \ref%
{alfaformula} we conclude that automorphism $\Phi $ provides an automorphic
equivalence of algebras $\left( H_{2}\right) _{W}^{\ast }$ and $H_{2}$ and
same automorphism provides an automorphic equivalence of algebras $H_{1}$
and $H_{2}$.
\end{proof}

\subsection{Automorphic equivalence and coordinate algebras.}

Categories $C_{\Theta }\left( H\right) $ of the coordinate algebras were
defined in \cite{PlotkinVarCat}. Here $\Theta $ is an arbitrary variety of
algebras, $H\in \Theta $. The objects of the category $C_{\Theta }\left(
H\right) $ are algebras $F/T$ where $F\in \mathrm{Ob}\Theta ^{0}$, $T\in
Cl_{H}(F)$. Morphisms of the category $C_{\Theta }\left( H\right) $ are are
homomorphisms of these algebras. Now we will formulate the notion of the
automorphic equivalence of algebras in the language of the coordinate
algebras.

\begin{theorem}
Automorphism $\Phi :\Theta ^{0}\rightarrow \Theta ^{0}$ provides an
automorphic equivalence of algebras $H_{1},H_{2}\in \Theta $ if and only if
there is an isomorphism $\Psi :C_{\Theta }\left( H_{1}\right) \rightarrow
C_{\Theta }\left( H_{2}\right) $ subject of conditions:

\begin{enumerate}
\item for every $F\in \mathrm{Ob}\Theta ^{0}$ and every $T\in Cl_{H_{1}}(F)$
the $\Psi \left( F/T\right) =\Phi \left( F\right) /\widetilde{T}$ holds,
where $\widetilde{T}\in Cl_{H_{2}}(\Phi \left( F\right) )$;

\item for every $F\in \mathrm{Ob}\Theta ^{0}$ the $\Psi \left( F/\left(
\Delta _{F}\right) _{H_{1}}^{\prime \prime }\right) =\Phi \left( F\right)
/\left( \Delta _{\Phi \left( F\right) }\right) _{H_{2}}^{\prime \prime }$
holds,

\item for every $F\in \mathrm{Ob}\Theta ^{0}$ and every $T\in Cl_{H_{1}}(F)$
the isomorphism $\Psi $ transforms the natural epimorphism $\overline{\tau }%
:F/\left( \Delta _{F}\right) _{H_{1}}^{\prime \prime }\rightarrow F\left(
X\right) /T$ to the natural epimorphism $\Psi \left( \overline{\tau }\right)
:\Phi \left( F\right) /\left( \Delta _{\Phi \left( F\right) }\right)
_{H_{2}}^{\prime \prime }\rightarrow \Psi \left( F/T\right) $;

\item for every $F_{1},F_{2}\in \mathrm{Ob}\Theta ^{0}$ and every $\nu \in 
\mathrm{Mor}_{C_{\Theta }\left( H_{1}\right) }\left( F_{1}/\left( \Delta
_{F_{1}}\right) _{H_{1}}^{\prime \prime },F_{2}/\left( \Delta
_{F_{2}}\right) _{H_{1}}^{\prime \prime }\right) $ if the diagram%
\begin{equation*}
\begin{array}{ccc}
F_{1} & \underset{\delta _{1}}{\rightarrow } & F_{1}/\left( \Delta
_{F_{1}}\right) _{H_{1}}^{\prime \prime } \\ 
\downarrow \mu &  & \nu \downarrow \\ 
F_{2} & \overset{\delta _{2}}{\rightarrow } & F_{2}/\left( \Delta
_{F_{2}}\right) _{H_{1}}^{\prime \prime }%
\end{array}%
\end{equation*}%
is commutative then the diagram%
\begin{equation*}
\begin{array}{ccc}
\Phi \left( F_{1}\right) & \underset{\widetilde{\delta _{1}}}{\rightarrow }
& \Phi \left( F_{1}\right) /\left( \Delta _{\Phi \left( F_{1}\right)
}\right) _{H_{2}}^{\prime \prime } \\ 
\downarrow \Phi \left( \mu \right) &  & \Psi \left( \nu \right) \downarrow
\\ 
\Phi \left( F_{2}\right) & \overset{\widetilde{\delta _{2}}}{\rightarrow } & 
\Phi \left( F_{2}\right) /\left( \Delta _{\Phi \left( F_{2}\right) }\right)
_{H_{2}}^{\prime \prime }%
\end{array}%
\end{equation*}%
is also commutative, where $\delta _{i}$ and $\widetilde{\delta _{i}}$ are
the natural epimorphisms, $i=1,2$, and $\mu \in \mathrm{Mor}_{\Theta
^{0}}\left( F_{1},F_{2}\right) $. The isomorphism $\Psi $ is uniquely
determined by automorphism $\Phi :\Theta ^{0}\rightarrow \Theta ^{0}$.
\end{enumerate}
\end{theorem}

\begin{proof}
We suppose that there are an automorphism $\Phi :\Theta ^{0}\rightarrow
\Theta ^{0}$ and an isomorphism $\Psi :C_{\Theta }\left( H_{1}\right)
\rightarrow C_{\Theta }\left( H_{2}\right) $ subject of conditions 1.-4. We
will prove that the automorphism $\Phi :\Theta ^{0}\rightarrow \Theta ^{0}$
provide an automorphic equivalence of algebras $H_{1},H_{2}\in \Theta $. If $%
F/T\in \mathrm{Ob}C_{\Theta }\left( H_{1}\right) $ then by condition 1.
there exists $\widetilde{T}\in Cl_{H_{2}}(\Phi \left( F\right) )$ such that $%
\Psi \left( F/T\right) =\Phi \left( F\right) /\widetilde{T}$. We denote by $%
\alpha \left( \Phi \right) _{F}\left( T\right) =\widetilde{T}$. $\alpha
\left( \Phi \right) _{F}$ is a mapping $Cl_{H_{1}}(F)\rightarrow
Cl_{H_{2}}(\Phi \left( F\right) )$. We consider $A_{1},A_{2}\in \mathrm{Ob}%
\Theta ^{0}$\textit{. }If $\mathrm{Hom}\left( A_{1},A_{2}\right)
=\varnothing $, in particular, because $\Gamma _{A_{1}}\nsubseteq \Gamma
_{A_{2}}$, then condition: for every $\mu _{1},\mu _{2}\in \mathrm{Hom}%
\left( A_{1},A_{2}\right) $ and every $T\in Cl_{H_{1}}(A_{2})$ the $\tau \mu
_{1}=\tau \mu _{2}$ holds if and only if when the $\widetilde{\tau }\Phi
\left( \mu _{1}\right) =\widetilde{\tau }\Phi \left( \mu _{2}\right) $
holds, where $\tau :A_{2}\rightarrow A_{2}/T$\textit{, }$\widetilde{\tau }%
:\Phi \left( A_{2}\right) \rightarrow \Phi \left( A_{2}\right) /\alpha
\left( \Phi \right) _{A_{2}}\left( T\right) $\ are the natural epimorphisms
- fulfills. Now we consider the case when $\Gamma _{A_{1}}\subseteq \Gamma
_{A_{2}}$. We take $\mu _{1},\mu _{2}\in \mathrm{Hom}\left(
A_{1},A_{2}\right) $\textit{, }$T\in Cl_{H_{1}}(A_{2})$\textit{\ }and
suppose that 
\begin{equation}
\tau \mu _{1}=\tau \mu _{2},  \label{taucondition}
\end{equation}%
where $\tau :A_{2}\rightarrow A_{2}/T$ is the natural epimorphism. We will
consider the diagrams%
\begin{equation}
\begin{array}{ccc}
A_{1} & \underset{\delta _{1}}{\rightarrow } & A_{1}/\left( \Delta
_{A_{1}}\right) _{H_{1}}^{\prime \prime } \\ 
\downarrow \mu _{i} &  & \overline{\mu _{i}}\downarrow \\ 
A_{2} & \overset{\delta _{2}}{\rightarrow } & A_{2}/\left( \Delta
_{A_{2}}\right) _{H_{1}}^{\prime \prime }%
\end{array}%
,  \label{deltadiagram}
\end{equation}%
where $\delta _{i}$ are the natural epimorphisms, $i=1,2$. If $\left(
a_{1},a_{2}\right) \in \left( \Delta _{A_{1}}\right) _{H_{1}}^{\prime \prime
}$ then for every $\varphi \in \mathrm{Hom}\left( A_{2},H_{1}\right) $ the $%
\varphi \mu _{i}\left( a_{1}\right) =\varphi \mu _{i}\left( a_{2}\right) $
holds, because $\varphi \mu _{i}\in \mathrm{Hom}\left( A_{1},H_{1}\right) $.
So $\mu _{i}\left( \Delta _{A_{1}}\right) _{H_{1}}^{\prime \prime }\subseteq
\left( \Delta _{A_{2}}\right) _{H_{1}}^{\prime \prime }$. In the degenerate
cases, if $\Gamma _{A_{1}}\nsubseteq \Gamma _{H_{1}}$ then $\left( \Delta
_{A_{1}}\right) _{H_{1}}^{\prime \prime }=\dbigcup\limits_{i\in \Gamma
}\left( A_{1}^{\left( i\right) }\right) ^{2}$. If in this case $\Gamma
_{A_{2}}\subseteq \Gamma _{H_{1}}$ then $\Gamma _{A_{1}}\subseteq \Gamma
_{A_{2}}\subseteq \Gamma _{H_{1}}$ and we have a contradiction. So $\Gamma
_{A_{2}}\nsubseteq \Gamma _{H_{1}}$ and $\left( \Delta _{A_{2}}\right)
_{H_{1}}^{\prime \prime }=\dbigcup\limits_{i\in \Gamma }\left( A_{2}^{\left(
i\right) }\right) ^{2}$. From this equality the $\mu _{i}\left( \Delta
_{A_{1}}\right) _{H_{1}}^{\prime \prime }\subseteq \left( \Delta
_{A_{2}}\right) _{H_{1}}^{\prime \prime }$ is concluded in all cases.
Therefore in the all cases the homomorphism $\overline{\mu _{i}}$, which
closes the diagram (\ref{deltadiagram}) commutative, exists.

We will consider the natural epimorphisms $\overline{\tau }:A_{2}/\left(
\Delta _{A_{2}}\right) _{H_{1}}^{\prime \prime }\rightarrow A_{2}/T$. The $%
\tau =\overline{\tau }\delta _{2}$ holds. From commutativity of the diagrams
(\ref{deltadiagram}) and from (\ref{taucondition}) we conclude that $%
\overline{\tau }\delta _{2}\mu _{1}=\overline{\tau }\delta _{2}\mu _{2}=%
\overline{\tau }\,\overline{\mu _{1}}\delta _{1}=\overline{\tau }\,\overline{%
\mu _{2}}\delta _{1}$. So $\overline{\tau }\,\overline{\mu _{1}}=\overline{%
\tau }\,\overline{\mu _{2}}$ and $\Psi \left( \overline{\tau }\right) \Psi
\left( \overline{\mu _{1}}\right) =\Psi \left( \overline{\tau }\right) \Psi
\left( \overline{\mu _{2}}\right) $, because $\overline{\tau },\overline{\mu
_{1}},\overline{\mu _{2}}\in \mathrm{Mor}C_{\Theta }\left( H_{1}\right) $.

By condition 4. and 2. we have that the diagrams%
\begin{equation}
\begin{array}{ccc}
\Phi \left( A_{1}\right) & \underset{\widetilde{\delta _{1}}}{\rightarrow }
& \Phi \left( A_{1}\right) /\left( \Delta _{\Phi \left( A_{1}\right)
}\right) _{H_{2}}^{\prime \prime } \\ 
\downarrow \Phi \left( \mu _{i}\right) &  & \Psi \left( \overline{\mu _{i}}%
\right) \downarrow \\ 
\Phi \left( A_{2}\right) & \overset{\widetilde{\delta _{2}}}{\rightarrow } & 
\Phi \left( A_{2}\right) /\left( \Delta _{\Phi \left( A_{2}\right) }\right)
_{H_{2}}^{\prime \prime }%
\end{array}%
,  \label{postdeltadiagram}
\end{equation}%
where $\widetilde{\delta _{i}}$ are the natural epimorphisms, $i=1,2$, is
also commutative. From commutativity of these diagrams we have that $\Psi
\left( \overline{\tau }\right) \Psi \left( \overline{\mu _{1}}\right) 
\widetilde{\delta _{1}}=\Psi \left( \overline{\tau }\right) \Psi \left( 
\overline{\mu _{2}}\right) \widetilde{\delta _{1}}=\Psi \left( \overline{%
\tau }\right) \widetilde{\delta _{2}}\Phi \left( \mu _{1}\right) =\Psi
\left( \overline{\tau }\right) \widetilde{\delta _{2}}\Phi \left( \mu
_{2}\right) $. By condition 3. $\Psi \left( \overline{\tau }\right) 
\widetilde{\delta _{2}}=\widetilde{\tau }$, where $\widetilde{\tau }:\Phi
\left( A_{2}\right) \rightarrow \Phi \left( A_{2}\right) /\alpha \left( \Phi
\right) _{A_{2}}\left( T\right) $ is the natural epimorphism. So $\widetilde{%
\tau }\Phi \left( \mu _{1}\right) =\widetilde{\tau }\Phi \left( \mu
_{2}\right) $.

Vice versa, if the $\widetilde{\tau }\Phi \left( \mu _{1}\right) =\widetilde{%
\tau }\Phi \left( \mu _{2}\right) $ holds, then from commutativity of the
diagrams (\ref{postdeltadiagram}) we, as above, conclude that $\Psi \left( 
\overline{\tau }\right) \Psi \left( \overline{\mu _{1}}\right) =\Psi \left( 
\overline{\tau }\right) \Psi \left( \overline{\mu _{2}}\right) $. So $%
\overline{\tau }\,\overline{\mu _{1}}=\overline{\tau }\,\overline{\mu _{2}}$%
. And from commutativity of the diagrams (\ref{deltadiagram}) we have that $%
\tau \mu _{1}=\tau \mu _{2}$.

We also must prove that $\alpha \left( \Phi \right) _{F}$ is a bijection for
every $F\in \mathrm{Ob}\Theta ^{0}$.If $R\in Cl_{H_{2}}(\Phi \left( F\right)
)$, then $\Phi \left( F\right) /R\in \mathrm{Ob}C_{\Theta }\left(
H_{2}\right) $. There exists $F_{1}/T\in \mathrm{Ob}C_{\Theta }\left(
H_{1}\right) $ such that $\Phi \left( F\right) /R=\Psi \left( F_{1}/T\right)
=\Phi \left( F_{1}\right) /\alpha \left( \Phi \right) _{F_{1}}\left(
T\right) $. So $\Phi \left( F\right) =\Phi \left( F_{1}\right) $, $F=F_{1}$
and $R=\alpha \left( \Phi \right) _{F}\left( T\right) $. So $\alpha \left(
\Phi \right) _{F}$ is a surjection. If $T_{1},T_{2}\in Cl_{H_{1}}(F)$ and $%
\alpha \left( \Phi \right) _{F}\left( T_{1}\right) =\alpha \left( \Phi
\right) _{F}\left( T_{2}\right) =\widetilde{T}$. Then $\Psi \left(
F/T_{1}\right) =\Psi \left( F/T_{2}\right) =\Phi \left( F\right) /\widetilde{%
T}$, so $F/T_{1}=F/T_{2}$ and $T_{1}=T_{2}$. So $\alpha \left( \Phi \right)
_{F}$ is an injection.

Now suppose that the automorphism $\Phi :\Theta ^{0}\rightarrow \Theta ^{0}$
provides an automorphic equivalence of algebras $H_{1},H_{2}\in \Theta $. By
Theorem \ref{potinere} there exists a system of bijections $\left\{
s_{F}:F\rightarrow \Phi \left( F\right) \mid F\in \mathrm{Ob}\Theta
^{0}\right\} $, such that $\Phi $ acts on the morphisms of $\Theta ^{0}$
according the formula (\ref{pot_inner}). By Theorem \ref{alfaformula} $%
s_{F}:Cl_{H_{1}}(F)\rightarrow Cl_{H_{2}}(\Phi (F))$ is the bijection
subject of the Definition \ref{Autom_equiv}. We will construct the functor $%
\Psi :C_{\Theta }\left( H_{1}\right) \rightarrow C_{\Theta }\left(
H_{2}\right) $. We define for every $F/T\in \mathrm{Ob}C_{\Theta }\left(
H_{1}\right) $ that $\Psi \left( F/T\right) =\Phi (F)/s_{F}\left( T\right)
\in \mathrm{Ob}C_{\Theta }\left( H_{2}\right) $.

We consider $\nu \in \mathrm{Mor}_{C_{\Theta }\left( H_{1}\right) }\left(
F_{1}/T_{1},F_{2}/T_{2}\right) $. By projective propriety of the free
algebras there exists homomorphism $\mu $ such that the diagram%
\begin{equation}
\begin{array}{ccc}
F_{1} & \underset{\tau _{1}}{\rightarrow } & F_{1}/T_{1} \\ 
\downarrow \mu &  & \nu \downarrow \\ 
F_{2} & \overset{\tau _{2}}{\rightarrow } & F_{2}/T_{2}%
\end{array}%
,  \label{taudiagram}
\end{equation}%
is commutative. Here $\tau _{i}$ are the natural epimorphisms, $i=1,2$. If $%
\left( f_{1},f_{2}\right) \in T_{1}$ then $\tau _{2}\mu \left( f_{1}\right)
=\nu \tau _{1}\left( f_{1}\right) =\nu \tau _{1}\left( f_{2}\right) =\tau
_{2}\mu \left( f_{2}\right) $. Therefore $\mu \left( T_{1}\right) \subseteq
T_{2}$ and $\Phi \left( \mu \right) \left( s_{F_{1}}\left( T_{1}\right)
\right) =s_{F_{2}}\mu s_{F_{1}}^{-1}s_{F_{1}}\left( T_{1}\right) \subseteq
s_{F_{2}}\left( T_{2}\right) $. Hence there exists a homomorphism $\overline{%
\nu }:\Phi (F_{1})/s_{F_{1}}\left( T_{1}\right) \rightarrow \Phi
(F_{2})/s_{F_{2}}\left( T_{2}\right) $ such that the diagram%
\begin{equation}
\begin{array}{ccc}
\Phi \left( F_{1}\right) & \underset{\widetilde{\tau _{1}}}{\rightarrow } & 
\Psi \left( F_{1}/T_{1}\right) =\Phi (F_{1})/s_{F_{1}}\left( T_{1}\right) \\ 
\downarrow \Phi \left( \mu \right) &  & \overline{\nu }\downarrow \\ 
\Phi \left( F_{2}\right) & \overset{\widetilde{\tau _{2}}}{\rightarrow } & 
\Psi \left( F_{2}/T_{2}\right) =\Phi (F_{2})/s_{F_{2}}\left( T_{2}\right)%
\end{array}%
,  \label{posttaudiagram}
\end{equation}%
where $\widetilde{\tau _{i}}$ are the natural epimorphisms, $i=1,2$, is
commutative.

The homomorphism $\overline{\nu }$ is not depend on the choice of the
homomorphism $\mu $, which make the diagram (\ref{taudiagram}) commutative.
Indeed, if two homomorphisms $\mu _{1},\mu _{2}:F_{1}\rightarrow F_{2}$ make
the diagram (\ref{taudiagram}) commutative, then $\tau _{2}\mu _{1}=\nu \tau
_{1}=\tau _{2}\mu _{2}$. So by Definition \ref{Autom_equiv} and Theorem \ref%
{alfaformula} $\widetilde{\tau _{2}}\Phi \left( \mu _{1}\right) =\widetilde{%
\tau _{2}}\Phi \left( \mu _{2}\right) $. So if we have two commutative
diagrams%
\begin{equation*}
\begin{array}{ccc}
\Phi \left( F_{1}\right) & \underset{\widetilde{\tau _{1}}}{\rightarrow } & 
\Phi (F_{1})/s_{F_{1}}\left( T_{1}\right) \\ 
\downarrow \Phi \left( \mu _{1}\right) &  & \overline{\nu _{1}}\downarrow \\ 
\Phi \left( F_{2}\right) & \overset{\widetilde{\tau _{2}}}{\rightarrow } & 
\Phi (F_{2})/s_{F_{2}}\left( T_{2}\right)%
\end{array}%
\text{ and }%
\begin{array}{ccc}
\Phi \left( F_{1}\right) & \underset{\widetilde{\tau _{1}}}{\rightarrow } & 
\Phi (F_{1})/s_{F_{1}}\left( T_{1}\right) \\ 
\downarrow \Phi \left( \mu _{2}\right) &  & \overline{\nu _{2}}\downarrow \\ 
\Phi \left( F_{2}\right) & \overset{\widetilde{\tau _{2}}}{\rightarrow } & 
\Phi (F_{2})/s_{F_{2}}\left( T_{2}\right)%
\end{array}%
\end{equation*}%
then $\overline{\nu _{1}}\widetilde{\tau _{1}}=\widetilde{\tau _{2}}\Phi
\left( \mu _{1}\right) =\widetilde{\tau _{2}}\Phi \left( \mu _{2}\right) =%
\overline{\nu _{2}}\widetilde{\tau _{1}}$ and $\overline{\nu _{1}}=\overline{%
\nu _{2}}$. We denote $\overline{\nu }=\Psi \left( \nu \right) $, where $%
\overline{\nu }$ is the homomorphism from the diagram (\ref{posttaudiagram}).

Now we will check that $\Psi $ is a functor. If $\nu _{1}\in \mathrm{Mor}%
_{C_{\Theta }\left( H_{1}\right) }\left( F_{1}/T_{1},F_{2}/T_{2}\right) $, $%
\nu _{2}\in \mathrm{Mor}_{C_{\Theta }\left( H_{1}\right) }\left(
F_{2}/T_{2},F_{3}/T_{3}\right) $ then, by the consideration of the two
diagrams which are similar to the diagram (\ref{taudiagram}), we conclude
that the diagram%
\begin{equation}
\begin{array}{ccc}
F_{1} & \underset{\tau _{1}}{\rightarrow } & F_{1}/T_{1} \\ 
\downarrow \mu _{1} &  & \nu _{1}\downarrow \\ 
F_{2} & \underset{\tau _{2}}{\rightarrow } & F_{2}/T_{2} \\ 
\downarrow \mu _{2} &  & \nu _{2}\downarrow \\ 
F_{3} & \overset{\tau _{3}}{\rightarrow } & F_{3}/T_{3}%
\end{array}
\label{bigtaudiagram}
\end{equation}%
is commutative. After this, by the consideration of the two diagrams which
are similar to the diagram (\ref{posttaudiagram}), we conclude that the
diagram%
\begin{equation}
\begin{array}{ccc}
\Phi \left( F_{1}\right) & \underset{\widetilde{\tau _{1}}}{\rightarrow } & 
\Psi \left( F_{1}/T_{1}\right) \\ 
\downarrow \Phi \left( \mu _{1}\right) &  & \Psi \left( \nu _{1}\right)
\downarrow \\ 
\Phi \left( F_{2}\right) & \underset{\widetilde{\tau _{2}}}{\rightarrow } & 
\Psi \left( F_{2}/T_{2}\right) \\ 
\downarrow \Phi \left( \mu _{2}\right) &  & \Psi \left( \nu _{2}\right)
\downarrow \\ 
\Phi \left( F_{3}\right) & \overset{\widetilde{\tau _{3}}}{\rightarrow } & 
\Psi \left( F_{3}/T_{3}\right)%
\end{array}
\label{bigposttaudiagram}
\end{equation}%
is commutative. And now, by the consideration of the big squares of the
diagrams (\ref{bigtaudiagram}) and (\ref{bigposttaudiagram}), we conclude
that $\Psi \left( \nu _{1}\nu _{2}\right) =\Psi \left( \nu _{1}\right) \Psi
\left( \nu _{2}\right) $. So $\Psi $ is a functor.

Now we will prove that $\Psi $ is a an isomorphism. By Corollary \ref%
{equivalence} from the Theorem \ref{alfaformula} the automorphism $\Phi
^{-1} $ provides an automorphic equivalence of algebras $H_{2}$ and $H_{1}.$%
By using of the automorphism $\Phi ^{-1}$ and the system of bijections $%
\left\{ s_{F}^{\Phi ^{-1}}=\left( s_{\Phi ^{-1}\left( F\right) }\right)
^{-1}:F\rightarrow \Phi ^{-1}\left( F\right) \mid F\in \mathrm{Ob}\Theta
^{0}\right\} $ we construct the functor $\widetilde{\Psi }$. If $F/T\in 
\mathrm{Ob}C_{\Theta }\left( H_{1}\right) $ then $\Psi \left( F/T\right)
=\Phi (F)/s_{F}\left( T\right) \in \mathrm{Ob}C_{\Theta }\left( H_{2}\right) 
$, $s_{F}\left( T\right) \in Cl_{H_{2}}(\Phi \left( F\right) )$ and $%
\widetilde{\Psi }\Psi \left( F/T\right) =\widetilde{\Psi }\left( \Phi
(F)/s_{F}\left( T\right) \right) =\Phi ^{-1}\Phi (F)/s_{\Phi (F)}^{\Phi
^{-1}}s_{F}\left( T\right) $. But $s_{\Phi (F)}^{\Phi ^{-1}}s_{F}\left(
T\right) =\left( s_{\Phi ^{-1}\left( \Phi (F)\right) }\right)
^{-1}s_{F}\left( T\right) =T$, so $\widetilde{\Psi }\Psi \left( F/T\right)
=F/T$. If $F/T\in \mathrm{Ob}C_{\Theta }\left( H_{2}\right) $ then $%
\widetilde{\Psi }\left( F/T\right) =\Phi ^{-1}(F)/\left( s_{\Phi ^{-1}\left(
F\right) }\right) ^{-1}\left( T\right) \in \mathrm{Ob}C_{\Theta }\left(
H_{1}\right) $, \linebreak $\left( s_{\Phi ^{-1}\left( F\right) }\right)
^{-1}\left( T\right) \in Cl_{H_{1}}(\Phi ^{-1}\left( F\right) )$ and $\Psi 
\widetilde{\Psi }\left( F/T\right) =\Phi \Phi ^{-1}(F)/s_{\Phi ^{-1}\left(
F\right) }\left( s_{\Phi ^{-1}\left( F\right) }\right) ^{-1}\left( T\right)
=F/T$. Now we consider $\nu \in \mathrm{Mor}_{C_{\Theta }\left( H_{1}\right)
}\left( F_{1}/T_{1},F_{2}/T_{2}\right) $. There exists homomorphism $\mu
:F_{1}\rightarrow F_{2}$ such that the diagram (\ref{taudiagram}) and the
diagram (\ref{posttaudiagram}) with $\overline{\nu }=\Psi \left( \nu \right) 
$ are commutative. So, as above, the diagram%
\begin{equation}
\begin{array}{ccc}
\Phi ^{-1}\Phi \left( F_{1}\right) & \underset{\widetilde{\widetilde{\tau
_{1}}}}{\rightarrow } & \widetilde{\Psi }\Psi \left( F_{1}/T_{1}\right) \\ 
\downarrow \Phi ^{-1}\Phi \left( \mu \right) &  & \widetilde{\Psi }\Psi
\left( \nu \right) \downarrow \\ 
\Phi ^{-1}\Phi \left( F_{2}\right) & \overset{\widetilde{\widetilde{\tau _{2}%
}}}{\rightarrow } & \widetilde{\Psi }\Psi \left( F_{2}/T_{2}\right)%
\end{array}%
,  \label{doublediagram}
\end{equation}%
where $\widetilde{\widetilde{\tau _{i}}}$ are the natural epimorphisms, $%
i=1,2$, is also commutative. But $\Phi ^{-1}\Phi \left( F_{i}\right) =F_{i}$%
, $\Phi ^{-1}\Phi \left( \mu \right) =\mu $, $\widetilde{\Psi }\Psi \left(
F_{i}/T_{i}\right) =F_{i}/T_{i}$, $i=1,2$. So $\widetilde{\widetilde{\tau
_{i}}}=\tau _{i}$, $i=1,2$, and diagram (\ref{doublediagram}) coincide with
the diagram (\ref{taudiagram}), only instead the homomorphism $\nu $ we have
a homomorphism $\widetilde{\Psi }\Psi \left( \nu \right) $. From $\nu \tau
_{1}=\widetilde{\Psi }\Psi \left( \nu \right) \tau _{1}$ we conclude that $%
\nu =\widetilde{\Psi }\Psi \left( \nu \right) $. Now we consider $\nu \in 
\mathrm{Mor}_{C_{\Theta }\left( H_{2}\right) }\left(
F_{1}/T_{1},F_{2}/T_{2}\right) $. We denote by $\mu $ the homomorphism which
makes the diagram (\ref{taudiagram}) commutative. As above, the diagram%
\begin{equation*}
\begin{array}{ccc}
\Phi ^{-1}\left( F_{1}\right) & \underset{\widetilde{\tau _{1}}}{\rightarrow 
} & \widetilde{\Psi }\left( F_{1}/T_{1}\right) \\ 
\downarrow \Phi ^{-1}\left( \mu \right) &  & \widetilde{\Psi }\left( \nu
\right) \downarrow \\ 
\Phi ^{-1}\left( F_{2}\right) & \overset{\widetilde{\tau _{2}}}{\rightarrow }
& \widetilde{\Psi }\left( F_{2}/T_{2}\right)%
\end{array}%
,
\end{equation*}%
is commutative. So the diagram 
\begin{equation}
\begin{array}{ccc}
\Phi \Phi ^{-1}\left( F_{1}\right) & \underset{\widetilde{\widetilde{\tau
_{1}}}}{\rightarrow } & \Psi \widetilde{\Psi }\left( F_{1}/T_{1}\right) \\ 
\downarrow \Phi \Phi ^{-1}\left( \mu \right) &  & \Psi \widetilde{\Psi }%
\left( \nu \right) \downarrow \\ 
\Phi \Phi ^{-1}\left( F_{2}\right) & \overset{\widetilde{\widetilde{\tau _{2}%
}}}{\rightarrow } & \Psi \widetilde{\Psi }\left( F_{2}/T_{2}\right)%
\end{array}%
,  \label{seconddoublediagram}
\end{equation}%
where $\widetilde{\widetilde{\tau _{i}}}$ are the natural epimorphisms, $%
i=1,2$, is commutative. $\Phi \Phi ^{-1}\left( F_{i}\right) =F_{i}$, $\Psi 
\widetilde{\Psi }\left( F_{i}/T_{i}\right) =F_{i}/T_{i}$, $i=1,2$, $\Phi
^{-1}\Phi \left( \mu \right) =\mu $. So $\widetilde{\widetilde{\tau _{i}}}%
=\tau _{i}$, $i=1,2$, and diagram (\ref{seconddoublediagram}) coincide with
the diagram (\ref{taudiagram}), only instead the homomorphism $\nu $ we have
a homomorphism $\Psi \widetilde{\Psi }\left( \nu \right) $. And as above we
have that $\Psi \widetilde{\Psi }\left( \nu \right) =\nu $. So the functor $%
\Psi $ is an isomorphism, because it has an inverse functor $\widetilde{\Psi 
}$.

Now we must prove that the isomorphism $\Psi $ is a subject of the
conditions 1.-4. Condition 1. fulfills by the definition of $\Psi $.
Condition 2. follows from the monotony of the bijection $s_{F}:Cl_{H_{1}}(F)%
\rightarrow Cl_{H_{2}}(\Phi (F))$. We consider the commutative diagram%
\begin{equation*}
\begin{array}{ccc}
F & \underset{\delta }{\rightarrow } & F/\left( \Delta _{F}\right)
_{H_{1}}^{\prime \prime } \\ 
\downarrow id_{F} &  & \overline{\tau }\downarrow \\ 
F & \overset{\tau }{\rightarrow } & F/T%
\end{array}%
,
\end{equation*}%
where $F\in \mathrm{Ob}\Theta ^{0}$, $T\in Cl_{H_{1}}(F)$, $\delta ,\tau ,%
\overline{\tau }$ are the natural epimorphisms. As above, the diagram%
\begin{equation*}
\begin{array}{ccc}
\Phi \left( F\right) & \underset{\widetilde{\delta }}{\rightarrow } & \Phi
\left( F\right) /\left( \Delta _{\Phi \left( F\right) }\right)
_{H_{2}}^{\prime \prime } \\ 
\downarrow id_{\Phi \left( F\right) } &  & \Psi \left( \overline{\tau }%
\right) \downarrow \\ 
\Phi \left( F\right) & \overset{\widetilde{\tau }}{\rightarrow } & \Psi
\left( F/T\right)%
\end{array}%
,
\end{equation*}%
where $\widetilde{\delta }$ and $\widetilde{\tau }$ are the natural
epimorphisms, is commutative. From $\Psi \left( \overline{\tau }\right) 
\widetilde{\delta }=\widetilde{\tau }$ we conclude that $\Psi \left( 
\overline{\tau }\right) $ is the natural epimorphism. So, condition 3.
holds. Condition 4. we can conclude from condition 2. and the definition of $%
\Psi $.
\end{proof}

\begin{proposition}
If there is a pair: an automorphism $\Phi :\Theta ^{0}\rightarrow \Theta
^{0} $ and an isomorphism $\Psi :C_{\Theta }\left( H_{1}\right) \rightarrow
C_{\Theta }\left( H_{2}\right) $ subject of conditions 1. - 4., where $%
H_{1},H_{2}\in \Theta $, then the isomorphism $\Psi $ is uniquely determined
by the automorphism $\Phi $.
\end{proposition}

\begin{proof}
If we have an automorphism $\Phi :\Theta ^{0}\rightarrow \Theta ^{0}$ and an
isomorphism $\Psi :C_{\Theta }\left( H_{1}\right) \rightarrow C_{\Theta
}\left( H_{2}\right) $ connected with $\Phi $ by conditions 1. - 4., then as
it was proved in the previous Theorem, the automorphism $\Phi :\Theta
^{0}\rightarrow \Theta ^{0}$ provides an automorphic equivalence of algebras 
$H_{1},H_{2}$ and the mapping $\alpha \left( \Phi \right)
_{F}:Cl_{H_{1}}(F)\ni T\rightarrow \widetilde{T}\in Cl_{H_{2}}(\Phi \left(
F\right) )$ is a subject of the Definition \ref{Autom_equiv}. By Theorem \ref%
{alfaformula} and Corollary \ref{independence} from Proposition \ref%
{centrarfunction} $\alpha \left( \Phi \right) _{F}=s_{F}^{\Phi }$, when
automorphism $\Phi $ acts on the morphisms of the category $\Theta ^{0}$ by
formula (\ref{pot_inner}) with the system of bijections $\left\{ s_{F}^{\Phi
}:F\rightarrow \Phi \left( F\right) \mid F\in \mathrm{Ob}\Theta ^{0}\right\} 
$, and $\alpha \left( \Phi \right) _{F}$ uniquely determined by the
automorphism $\Phi $. So by condition 1. we have that for every $F/T\in
C_{\Theta }\left( H_{1}\right) $ the $\Psi \left( F/T\right) $ uniquely
determined.

We just have to prove that the acting of the isomorphism $\Psi $ on
morphisms of the category $C_{\Theta }\left( H_{1}\right) $ also uniquely
determined. We consider \linebreak $\nu \in \mathrm{Mor}_{C_{\Theta }\left(
H_{1}\right) }\left( F_{1}/T_{1},F_{2}/T_{2}\right) $. As it was proved in
the previous Theorem, there exists homomorphism $\mu :F_{1}\rightarrow F_{2}$
such that the diagram (\ref{taudiagram}) is commutative. Also, as it was
proved in the previous Theorem when diagrams (\ref{deltadiagram}) were
considered, there exists homomorphism $\overline{\mu }:F_{1}/\left( \Delta
_{F_{1}}\right) _{H_{1}}^{\prime \prime }\rightarrow F_{2}/\left( \Delta
_{F_{2}}\right) _{H_{1}}^{\prime \prime }$ such that 
\begin{equation}
\delta _{2}\mu =\overline{\mu }\delta _{1},  \label{deltacommut}
\end{equation}%
where $\delta _{i}:$ $F_{i}\rightarrow F_{i}/\left( \Delta _{F_{i}}\right)
_{H_{1}}^{\prime \prime }$ are the natural epimorphisms, $i=1,2$. As above
we denote by $\overline{\tau _{i}}:F_{i}/\left( \Delta _{F_{i}}\right)
_{H_{1}}^{\prime \prime }\rightarrow F_{i}/T_{i}$ the natural epimorphisms, $%
i=1,2$. The $\tau _{i}=\overline{\tau _{i}}\delta _{i}$ holds, where $\tau
_{i}:$ $F_{i}\rightarrow F_{i}/T_{i}$ the natural epimorphisms. So $\nu 
\overline{\tau _{1}}\delta _{1}=\nu \tau _{1}=\tau _{2}\mu =\overline{\tau
_{2}}\delta _{2}\mu =\overline{\tau _{2}}\overline{\mu }\delta _{1}$,
therefore 
\begin{equation}
\nu \overline{\tau _{1}}=\overline{\tau _{2}}\overline{\mu }.
\label{incategcommut}
\end{equation}%
We suppose that there are two isomorphisms $\Psi _{1},\Psi _{2}:C_{\Theta
}\left( H_{1}\right) \rightarrow C_{\Theta }\left( H_{2}\right) $ subjects
of conditions 1. - 4. and will prove that $\Psi _{1}\left( \nu \right) =\Psi
_{2}\left( \nu \right) $. All homomorphisms in the (\ref{incategcommut}) are
morphisms of the category $C_{\Theta }\left( H_{1}\right) $. We apply to (%
\ref{incategcommut}) both isomorphisms: as $\Psi _{1}$, as $\Psi _{2}$ -
achieve two commutative diagrams%
\begin{equation*}
\begin{array}{ccc}
\Psi _{i}\left( F_{1}/\left( \Delta _{F_{1}}\right) _{H_{1}}^{\prime \prime
}\right) & \underset{\Psi _{i}\left( \overline{\tau _{1}}\right) }{%
\rightarrow } & \Psi _{i}\left( F_{1}/T_{1}\right) \\ 
\downarrow \Psi _{i}\left( \overline{\mu }\right) &  & \Psi _{i}\left( \nu
\right) \downarrow \\ 
\Psi _{i}\left( F_{2}/\left( \Delta _{F_{2}}\right) _{H_{1}}^{\prime \prime
}\right) & \overset{\Psi _{i}\left( \overline{\tau _{2}}\right) }{%
\rightarrow } & \Psi _{i}\left( F_{2}/T_{2}\right)%
\end{array}%
,
\end{equation*}%
where $i=1,2$. As it was proved above, the corresponding objects of these
diagrams are coincide. The morphisms in the rows of these diagrams are also
coincide, because by condition 3. they are natural epimorphisms which exists
between the corresponding objects. That is, in fact, we have these two
commutative diagrams%
\begin{equation*}
\begin{array}{ccc}
\Phi \left( F_{1}\right) /\left( \Delta _{\Phi \left( F_{1}\right) }\right)
_{H_{2}}^{\prime \prime } & \underset{\widetilde{\tau _{1}}}{\rightarrow } & 
\Phi \left( F_{1}\right) /s_{F_{1}}^{\Phi }\left( T_{1}\right) \\ 
\downarrow \Psi _{i}\left( \overline{\mu }\right) &  & \Psi _{i}\left( \nu
\right) \downarrow \\ 
\Phi \left( F_{2}\right) /\left( \Delta _{\Phi \left( F_{2}\right) }\right)
_{H_{2}}^{\prime \prime } & \overset{\widetilde{\tau _{2}}}{\rightarrow } & 
\Phi \left( F_{2}\right) /s_{F_{2}}^{\Phi }\left( T_{2}\right)%
\end{array}%
,
\end{equation*}%
where $i=1,2$, automorphism $\Phi $ acts on the morphisms of the category $%
\Theta ^{0}$ by formula (\ref{pot_inner}) with the system of bijections $%
\left\{ s_{F}^{\Phi }:F\rightarrow \Phi \left( F\right) \mid F\in \mathrm{Ob}%
\Theta ^{0}\right\} $, $\widetilde{\tau _{1}}$, $\widetilde{\tau _{2}}$ the
natural epimorphisms. Now we consider the diagrams%
\begin{equation*}
\begin{array}{ccccc}
\Phi \left( F_{1}\right) & \underset{\widetilde{\delta _{1}}}{\rightarrow }
& \Phi \left( F_{1}\right) /\left( \Delta _{\Phi \left( F_{1}\right)
}\right) _{H_{2}}^{\prime \prime } & \underset{\widetilde{\tau _{1}}}{%
\rightarrow } & \Phi \left( F_{1}\right) /s_{F_{1}}^{\Phi }\left(
T_{1}\right) \\ 
\downarrow \Phi \left( \mu \right) &  & \Psi _{i}\left( \overline{\mu }%
\right) \downarrow &  & \Psi _{i}\left( \nu \right) \downarrow \\ 
\Phi \left( F_{2}\right) & \overset{\widetilde{\delta _{2}}}{\rightarrow } & 
\Phi \left( F_{2}\right) /\left( \Delta _{\Phi \left( F_{2}\right) }\right)
_{H_{2}}^{\prime \prime } & \overset{\widetilde{\tau _{2}}}{\rightarrow } & 
\Phi \left( F_{2}\right) /s_{F_{2}}^{\Phi }\left( T_{2}\right)%
\end{array}%
,
\end{equation*}%
where $i=1,2$, $\widetilde{\delta _{1}}$, $\widetilde{\delta _{2}}$ the
natural epimorphisms. From (\ref{deltacommut}) and condition 4. we can
conclude that the left small squares of these diagrams are commutative.
Therefore both these diagrams are commutative and $\widetilde{\tau _{2}}%
\widetilde{\delta _{2}}\Phi \left( \mu \right) =\Psi _{1}\left( \nu \right) 
\widetilde{\tau _{1}}\widetilde{\delta _{1}}=\Psi _{2}\left( \nu \right) 
\widetilde{\tau _{1}}\widetilde{\delta _{1}}$. $\widetilde{\tau _{1}}%
\widetilde{\delta _{1}}$ - is an epimorphism so $\Psi _{1}\left( \nu \right)
=\Psi _{2}\left( \nu \right) $.
\end{proof}

\section{Examples.\label{examples}}

\setcounter{equation}{0}

In this Section we will consider some varieties of many-sorted algebras. We
will calculate the group $\mathfrak{A/Y}$ for these varieties and if this
group is not trivial we will give examples of algebras which are
automorphically equivalent but are not geometrically equivalent. We must
prove in all these cosiderations that in our varieties Condition \ref%
{monoiso} fulfills.

We use for this aim the algebraic proprieties of the automorphisms of the
category $\Theta ^{0}$. First of all we must say that every automorphism $%
\Phi $ of the category $\Theta ^{0}$ transforms every identity morphism $%
id_{F}\in \mathrm{Mor}_{\Theta ^{0}}\left( F,F\right) $, where $F\in \mathrm{%
Ob}\Theta ^{0}$, to the identity morphism $id_{\Phi \left( F\right) }\in 
\mathrm{Mor}_{\Theta ^{0}}\left( \Phi \left( F\right) ,\Phi \left( F\right)
\right) $, because $id_{F}$ is an unit of the monoid $\mathrm{Mor}_{\Theta
^{0}}\left( F,F\right) $. Therefore every automorphism $\Phi $ of the
category $\Theta ^{0}$ transforms the isomorphism $\alpha :F_{1}\rightarrow
F_{2}$, where $F_{1},F_{2}\in \mathrm{Ob}\Theta ^{0}$, to the to the
isomorphism $\Phi \left( \alpha \right) :\Phi \left( F_{1}\right)
\rightarrow \Phi \left( F_{2}\right) $, because isomorphisms are the
invertible morphisms. Also every automorphism $\Phi $ of the category $%
\Theta ^{0}$ preserves the coproducts: $\Phi \left( \coprod\limits_{i\in
I}F_{i}\right) \cong \coprod\limits_{i\in I}\Phi \left( F_{i}\right) $,
where $F_{i}\in \mathrm{Ob}\Theta ^{0}$ - because the coproduct is defined
by algebraic conditions on the morphisms.

\begin{definition}
\label{IBN}We say that the variety $\Theta $ has an \textbf{IBN propriety}
if for every $F\left( X\right) ,F\left( Y\right) \in \mathrm{Ob}\Theta ^{0}$
the $F\left( X\right) \cong F\left( Y\right) $ holds if and only if the $%
\left\vert X^{\left( i\right) }\right\vert =\left\vert Y^{\left( i\right)
}\right\vert $ holds for every $i\in \Gamma $.
\end{definition}

It is clear that if $F\left( X\right) ,F\left( Y\right) \in \mathrm{Ob}%
\Theta ^{0}$, then $F\left( X\right) \sqcup F\left( Y\right) \cong F\left(
Z\right) $, where $\left\vert X^{\left( i\right) }\right\vert +\left\vert
Y^{\left( i\right) }\right\vert =\left\vert Z^{\left( i\right) }\right\vert $%
, $i\in \Gamma $. Therefore, as it was explained in \cite[Proposition 5.2]%
{ShestTsur}, if the variety $\Theta $ has an IBN propriety every
automorphism $\Phi $ of the category $\Theta ^{0}$ induces the automorphism $%
\varphi $ of the additive monoid $%
%TCIMACRO{\U{2115} }%
%BeginExpansion
\mathbb{N}
%EndExpansion
^{\left\vert \Gamma \right\vert }$, such that $\varphi \left( \left(
n_{i}\right) _{i\in \Gamma }\right) =\left( m_{i}\right) _{i\in \Gamma }$ if
and only if $\Phi \left( F\left( X\right) \right) \cong F\left( Y\right) $,
where $\left\vert X^{\left( i\right) }\right\vert =n_{i}$, $\left\vert
Y^{\left( i\right) }\right\vert =m_{i}$, $i\in \Gamma $. The automorphism $%
\varphi $ transforms the minimal set of the generators of the $%
%TCIMACRO{\U{2115} }%
%BeginExpansion
\mathbb{N}
%EndExpansion
^{\left\vert \Gamma \right\vert }$: $\left\{ e_{i}\mid i\in \Gamma \right\} $%
, where $e_{i}=\left( 0,\ldots ,0,1,0,\ldots ,0\right) $, $1$ is located on
the $i$ place - to itself. Its enough to prove Condition \ref{monoiso} if $%
\left\vert \Gamma \right\vert =1$. In the case $\left\vert \Gamma
\right\vert >1$, we must make additional efforts.

\begin{definition}
We say that homomorphism $\iota :A\rightarrow B$, where $A,B\in \Theta $, is
an \textbf{embedding} if $\ker \iota =\Delta _{A}$.
\end{definition}

We denote $\iota :A\hookrightarrow B$.

\begin{proposition}
\label{injection}We consider $F_{1},F_{2}\in \mathrm{Ob}\Theta ^{0}$. A
homomorphism $\iota :F_{1}\rightarrow F_{2}$ is an embedding if and only if
for every $F\in \mathrm{Ob}\Theta ^{0}$ and every $\alpha ,\beta \in \mathrm{%
Hom}\left( F,F_{1}\right) $ from $\iota \alpha =\iota \beta $ we can
conclude that $\alpha =\beta $ holds.
\end{proposition}

\begin{proof}
We assume that homomorphism $\iota :F_{1}\rightarrow F_{2}$ is not an
embedding, so there are $f_{1}^{\left( i\right) },f_{2}^{\left( i\right)
}\in F_{1}^{\left( i\right) }\subseteq F_{1}$, where $i\in \Gamma $, such
that $f_{1}^{\left( i\right) }\neq f_{2}^{\left( i\right) }$ but $\iota
\left( f_{1}^{\left( i\right) }\right) =\iota \left( f_{2}^{\left( i\right)
}\right) $. We consider $F=F\left( x^{\left( i\right) }\right) \in \mathrm{Ob%
}\Theta ^{0}$ and homomorphisms $\alpha ,\beta \in \mathrm{Hom}\left(
F,F_{1}\right) $ such that $\alpha \left( x^{\left( i\right) }\right)
=f_{1}^{\left( i\right) }$, $\beta \left( x^{\left( i\right) }\right)
=f_{2}^{\left( i\right) }$. $\alpha \neq \beta $, but $\iota \alpha =\iota
\beta $.

The proof of the converse is obvious.
\end{proof}

Therefore embeddings in the category $\Theta ^{0}$ can be defined by
algebraic propriety of morphisms. So every automorphism of the category $%
\Theta ^{0}$ transforms an embedding to the embedding. We will use this fact
in the prove of Condition \ref{monoiso}.

\subsection{Actions of semigroups over sets.}

In this Subsection $\Theta $ is a variety of the all actions of semigroups
over sets. $\Gamma =\left\{ 1,2\right\} $: the first sort is a sort of
elements of semigroups, the second sort is a sort of elements of sets. $%
\Omega =\left\{ \cdot ,\circ \right\} $: $\cdot $ is a multiplication in the
semigroup, $\circ $ is an action of the elements of the semigroup over the
elements of the set. $\tau _{\cdot }=\left( 1,1;1\right) $, $\tau _{\circ
}=\left( 1,2;2\right) $. A free algebras $F\left( X\right) $ in our variety
have this form: $\left( F\left( X\right) \right) ^{\left( 1\right) }=S\left(
X^{\left( 1\right) }\right) $, $\left( F\left( X\right) \right) ^{\left(
2\right) }=\left( S\left( X^{\left( 1\right) }\right) \circ X^{\left(
2\right) }\right) \cup X^{\left( 2\right) }$, where $S\left( X^{\left(
1\right) }\right) \circ X^{\left( 2\right) }=\left\{ s\circ x^{\left(
2\right) }\mid s\in S\left( X^{\left( 1\right) }\right) ,x^{\left( 2\right)
}\in X^{\left( 2\right) }\right\} $, $S\left( X^{\left( 1\right) }\right) $
is a free semigroup generated by the set of free generators $X^{\left(
1\right) }$.

\begin{proposition}
\label{ActIBN}The variety of the all actions of semigroups over sets has an
IBN propriety.
\end{proposition}

\begin{proof}
We consider $F\left( X\right) \in \mathrm{Ob}\Theta ^{0}$. We denote $\left(
F\left( X\right) \right) ^{\left( 1\right) }=S$. The quotient semigroup $%
S/S^{2}$ has $\left\vert X^{\left( 1\right) }\right\vert +1$ elements.

We consider in $\left( F\left( X\right) \right) ^{\left( 2\right) }$ the
relation $R=\left\{ \left( s\circ y,y\right) \mid s\in S,y\in \left( F\left(
X\right) \right) ^{\left( 2\right) }\right\} $ and the relation $Q$ - the
minimal equivalence in $\left( F\left( X\right) \right) ^{\left( 2\right) }$
which contain $R$. Quotient set $\left( F\left( X\right) \right) ^{\left(
2\right) }/Q$ has $\left\vert X^{\left( 2\right) }\right\vert $ elements.
\end{proof}

\begin{proposition}
\label{ActCond}Condition \ref{monoiso} fulfills in the variety of the all
actions of semigroups over sets.
\end{proposition}

\begin{proof}
By previous Proposition every automorphism $\Phi $\ of the category $\Theta
^{0}$ induces the automorphism $\varphi $ of the additive monoid $%
%TCIMACRO{\U{2115} }%
%BeginExpansion
\mathbb{N}
%EndExpansion
^{2}$, such that $\varphi \left( \left( n_{1},n_{2}\right) _{i\in \Gamma
}\right) =\left( m_{1},m_{2}\right) _{i\in \Gamma }$ if and only if $\Phi
\left( F\left( X\right) \right) \cong F\left( Y\right) $, where $\left\vert
X^{\left( i\right) }\right\vert =n_{i}$, $\left\vert Y^{\left( i\right)
}\right\vert =m_{i}$, $i\in \left\{ 1,2\right\} $. This automorphism
transforms the minimal set of the generators of the $%
%TCIMACRO{\U{2115} }%
%BeginExpansion
\mathbb{N}
%EndExpansion
^{2}$: $\left\{ \left( 1,0\right) ,\left( 0,1\right) \right\} $ to itself.
We will prove that it is not possible that $\varphi \left( 1,0\right)
=\left( 0,1\right) $, $\varphi \left( 0,1\right) =\left( 1,0\right) $.
Indeed, if $\varphi \left( 1,0\right) =\left( 0,1\right) $, then $\Phi
\left( F\left( x^{\left( 1\right) }\right) \right) \cong F\left( x^{\left(
2\right) }\right) $ and $\Phi \left( F\left( x_{1}^{\left( 1\right) },\ldots
,x_{n}^{\left( 1\right) }\right) \right) \cong F\left( x_{1}^{\left(
2\right) },\ldots ,x_{n}^{\left( 2\right) }\right) $, $n\geq 1$. $F\left(
x_{1}^{\left( 1\right) },\ldots ,x_{n}^{\left( 1\right) }\right) =S\left(
x_{1}^{\left( 1\right) },\ldots ,x_{n}^{\left( 1\right) }\right) $ is a free
semigroup with $n$ free generators. $F\left( x_{1}^{\left( 2\right) },\ldots
,x_{n}^{\left( 2\right) }\right) =\left\{ x_{1}^{\left( 2\right) },\ldots
,x_{n}^{\left( 2\right) }\right\} $ is a set with $n$ elements. There is an
embedding $\iota :F\left( x_{1}^{\left( 1\right) },\ldots ,x_{n}^{\left(
1\right) }\right) \hookrightarrow F\left( x_{1}^{\left( 1\right)
},x_{2}^{\left( 1\right) }\right) $, because in the semigroup $S\left(
x_{1}^{\left( 1\right) },x_{2}^{\left( 1\right) }\right) $, for example,
elements $x_{1}^{\left( 1\right) }x_{2}^{\left( 1\right) }x_{1}^{\left(
1\right) }$, $x_{1}^{\left( 1\right) }\left( x_{2}^{\left( 1\right) }\right)
^{2}x_{1}^{\left( 1\right) }$, $x_{1}^{\left( 1\right) }\left( x_{2}^{\left(
1\right) }\right) ^{3}x_{1}^{\left( 1\right) }$ and so on are free. By
Proposition \ref{injection} automorphism $\Phi $ transforms this embedding
to the embedding $\Phi \left( \iota \right) :\Phi \left( F\left(
x_{1}^{\left( 1\right) },\ldots ,x_{n}^{\left( 1\right) }\right) \right)
\hookrightarrow \Phi \left( F\left( x_{1}^{\left( 1\right) },x_{2}^{\left(
1\right) }\right) \right) $, where $\Phi \left( F\left( x_{1}^{\left(
1\right) },\ldots ,x_{n}^{\left( 1\right) }\right) \right) \cong F\left(
x_{1}^{\left( 2\right) },\ldots ,x_{n}^{\left( 2\right) }\right) $, $\Phi
\left( F\left( x_{1}^{\left( 1\right) },x_{2}^{\left( 1\right) }\right)
\right) \cong F\left( x_{1}^{\left( 2\right) },x_{2}^{\left( 2\right)
}\right) $. But it is not possible, because when $n>2$ we can not embed the
set with $n$ elements to the set with $2$ elements. Therefore $\varphi
\left( 1,0\right) =\left( 1,0\right) $, $\varphi \left( 0,1\right) =\left(
0,1\right) $ must fulfills. So $\Phi \left( F\left( x^{\left( 1\right)
}\right) \right) \cong F\left( x^{\left( 1\right) }\right) $, $\Phi \left(
F\left( x^{\left( 2\right) }\right) \right) \cong F\left( x^{\left( 2\right)
}\right) $.
\end{proof}

\begin{theorem}
\label{assTheorem}The automorphic equivalence in the variety of the all
actions of semigroups over sets coincides with the geometric equivalence.
\end{theorem}

\begin{proof}
By previous Proposition we can use the method of verbal operations for
calculating group $\mathfrak{A/Y}$ for the variety $\Theta $. We will find
the systems of words $W=\left\{ w_{\cdot },w_{\circ }\right\} $ subjects of
conditions Op1) and Op2). Any such sequence corresponds to the some strongly
stable automorphism of the category $\Theta ^{0}$. By Op1) $w_{\cdot }\in
F\left( x_{1}^{\left( 1\right) },x_{2}^{\left( 1\right) }\right) =S\left(
x_{1}^{\left( 1\right) },x_{2}^{\left( 1\right) }\right) $. We have only two
possibility: $w_{\cdot }\left( x_{1}^{\left( 1\right) },x_{2}^{\left(
1\right) }\right) =x_{1}^{\left( 1\right) }x_{2}^{\left( 1\right) }$ or $%
w_{\cdot }\left( x_{1}^{\left( 1\right) },x_{2}^{\left( 1\right) }\right)
=x_{2}^{\left( 1\right) }x_{1}^{\left( 1\right) }$. $w_{\cdot }\left(
x_{1}^{\left( 1\right) },x_{2}^{\left( 1\right) }\right) $ can not be
shorter or longer than these words, because, otherwise, the mapping $%
s_{F\left( x_{1}^{\left( 1\right) },x_{2}^{\left( 1\right) }\right)
}:F\left( x_{1}^{\left( 1\right) },x_{2}^{\left( 1\right) }\right)
\rightarrow F\left( x_{1}^{\left( 1\right) },x_{2}^{\left( 1\right) }\right)
_{W}^{\ast }$ which preserves elements $x_{1}^{\left( 1\right) }$ and $%
x_{2}^{\left( 1\right) }$ and will be a homomorphism can not be an
isomorphism.

$w_{\circ }\in F\left( x^{\left( 1\right) },x^{\left( 2\right) }\right) $.
Here we have only one possibility: $w_{\circ }\left( x^{\left( 1\right)
},x^{\left( 2\right) }\right) =x^{\left( 1\right) }\circ x^{\left( 2\right)
} $. Also $w_{\circ }\left( x^{\left( 1\right) },x^{\left( 2\right) }\right) 
$ can not be shorter or longer, because, otherwise, we can not construct a
bijection $s_{F\left( x^{\left( 1\right) },x^{\left( 2\right) }\right)
}:F\left( x^{\left( 1\right) },x^{\left( 2\right) }\right) \rightarrow
F\left( x^{\left( 1\right) },x^{\left( 2\right) }\right) $ subject of
condition Op2).

But the system of words $\left\{ w_{\cdot }\left( x_{1}^{\left( 1\right)
},x_{2}^{\left( 1\right) }\right) =x_{2}^{\left( 1\right) }x_{1}^{\left(
1\right) },w_{\circ }\left( x^{\left( 1\right) },x^{\left( 2\right) }\right)
=x^{\left( 1\right) }\circ x^{\left( 2\right) }\right\} $ also is not a
subject of conditions Op1) and Op2). Indeed, if we have an isomorphism $%
s_{F\left( x_{1}^{\left( 1\right) },x_{2}^{\left( 1\right) },,x^{\left(
2\right) }\right) }:F\left( x_{1}^{\left( 1\right) },x_{2}^{\left( 1\right)
},x^{\left( 2\right) }\right) \rightarrow F\left( x_{1}^{\left( 1\right)
},x_{2}^{\left( 1\right) },x^{\left( 2\right) }\right) _{W}^{\ast }$, then,
because $\left( x_{1}^{\left( 1\right) }x_{2}^{\left( 1\right) }\right)
\circ x^{\left( 2\right) }=x_{1}^{\left( 1\right) }\circ \left(
x_{2}^{\left( 1\right) }\circ x^{\left( 2\right) }\right) $, the $w_{\circ
}\left( w_{\cdot }\left( x_{1}^{\left( 1\right) },x_{2}^{\left( 1\right)
}\right) ,x^{\left( 2\right) }\right) =w_{\circ }\left( x_{1}^{\left(
1\right) },w_{\circ }\left( x_{2}^{\left( 1\right) },x^{\left( 2\right)
}\right) \right) $ must hold. But 
\begin{equation*}
w_{\circ }\left( w_{\cdot }\left( x_{1}^{\left( 1\right) },x_{2}^{\left(
1\right) }\right) ,x^{\left( 2\right) }\right) =w_{\circ }\left(
x_{2}^{\left( 1\right) }x_{1}^{\left( 1\right) },x^{\left( 2\right) }\right)
=\left( x_{2}^{\left( 1\right) }x_{1}^{\left( 1\right) }\right) \circ
x^{\left( 2\right) }
\end{equation*}
and 
\begin{equation*}
w_{\circ }\left( x_{1}^{\left( 1\right) },w_{\circ }\left( x_{2}^{\left(
1\right) },x^{\left( 2\right) }\right) \right) =w_{\circ }\left(
x_{1}^{\left( 1\right) },x_{2}^{\left( 1\right) }\circ x^{\left( 2\right)
}\right) =x_{1}^{\left( 1\right) }\circ \left( x_{2}^{\left( 1\right) }\circ
x^{\left( 2\right) }\right) =\left( x_{1}^{\left( 1\right) }x_{2}^{\left(
1\right) }\right) \circ x^{\left( 2\right) }.
\end{equation*}
This contradiction proves that there is only one possibility for the system $%
W$: 
\begin{equation*}
W=\left\{ w_{\cdot }\left( x_{1}^{\left( 1\right) },x_{2}^{\left( 1\right)
}\right) =x_{1}^{\left( 1\right) }x_{2}^{\left( 1\right) },w_{\circ }\left(
x^{\left( 1\right) },x^{\left( 2\right) }\right) =x^{\left( 1\right) }\circ
x^{\left( 2\right) }\right\} .
\end{equation*}
So $\mathfrak{S=}\left\{ 1\right\} $ in our variety $\Theta $ and by Theorem %
\ref{factorisation} $\mathfrak{A/Y=}\left\{ 1\right\} $. By Theorem \ref%
{innerautomprop} the proof is complete.
\end{proof}

\subsection{Automatons.}

In this Subsection $\Theta $ is a variety of the all automatons. $\Gamma
=\left\{ 1,2,3\right\} $: the first sort is a sort of input signals, the
second sort is a sort of statements of automatons, the third sort is a sort
of output signals. $\Omega =\left\{ \ast ,\circ \right\} $. Operation $\ast $
gives as a new statement of automaton according an input signal and a
previous statement of automaton. $\tau _{\ast }=\left( 1,2;2\right) $.
Operation $\circ $ gives as an output signal according an input signal and a
statement of automaton. $\tau _{\circ }=\left( 1,2;3\right) $. A free
algebras $F\left( X\right) $ in our variety has this form: $\left( F\left(
X\right) \right) ^{\left( 1\right) }=X^{\left( 1\right) }$, $\left( F\left(
X\right) \right) ^{\left( 2\right) }=\left( X^{\left( 1\right) }\right)
^{\infty }\ast X^{\left( 2\right) }=\left\{ x_{i_{n}}^{\left( 1\right) }\ast
\left( \ldots \ast \left( x_{i_{1}}^{\left( 1\right) }\ast x_{j}^{\left(
2\right) }\right) \right) \mid x_{j}^{\left( 2\right) }\in X^{\left(
2\right) },x_{i_{n}}^{\left( 1\right) },\ldots ,x_{i_{1}}^{\left( 1\right)
}\in X^{\left( 1\right) },n\in 
%TCIMACRO{\U{2115} }%
%BeginExpansion
\mathbb{N}
%EndExpansion
\right\} $. Here, when $n=0$, we must understand $x_{i_{n}}^{\left( 1\right)
}\ast \left( \ldots \ast \left( x_{i_{1}}^{\left( 1\right) }\ast
x_{j}^{\left( 2\right) }\right) \right) =x_{j}^{\left( 2\right) }$. $\left(
F\left( X\right) \right) ^{\left( 3\right) }=\left( \left( X^{\left(
1\right) }\right) ^{\infty }\circ \left( F\left( X\right) \right) ^{\left(
2\right) }\right) \cup X^{\left( 3\right) }$, where $\left( X^{\left(
1\right) }\right) ^{\infty }\circ \left( F\left( X\right) \right) ^{\left(
2\right) }=\left\{ x_{i_{n}}^{\left( 1\right) }\circ \left( \ldots \circ
\left( x_{i_{1}}^{\left( 1\right) }\circ y^{\left( 2\right) }\right) \right)
\mid y^{\left( 2\right) }\in \left( F\left( X\right) \right) ^{\left(
2\right) },x_{i_{n}}^{\left( 1\right) },\ldots ,x_{i_{1}}^{\left( 1\right)
}\in X^{\left( 1\right) },n\geq 1\right\} $.

\begin{proposition}
The variety of the all automatons has an IBN propriety.
\end{proposition}

\begin{proof}
We consider $F\left( X\right) \in \mathrm{Ob}\Theta ^{0}$. $\left\vert
X^{\left( 1\right) }\right\vert =\left\vert \left( F\left( X\right) \right)
^{\left( 1\right) }\right\vert $. As in Proposition \ref{ActIBN} we consider
in $\left( F\left( X\right) \right) ^{\left( 2\right) }$ the relation $%
R=\left\{ \left( s\ast y,y\right) \mid s\in \left( F\left( X\right) \right)
^{\left( 1\right) },y\in \left( F\left( X\right) \right) ^{\left( 2\right)
}\right\} $ and the relation $Q$ - the minimal equivalence in $\left(
F\left( X\right) \right) ^{\left( 2\right) }$ which contain $R$. Quotient
set $\left( F\left( X\right) \right) ^{\left( 2\right) }/Q$ has $\left\vert
X^{\left( 2\right) }\right\vert $ elements. \linebreak $\left\vert X^{\left(
3\right) }\right\vert =\left\vert \left( F\left( X\right) \right) ^{\left(
3\right) }\setminus \left\langle \left( F\left( X\right) \right) ^{\left(
1\right) }\cup \left( F\left( X\right) \right) ^{\left( 2\right)
}\right\rangle ^{\left( 3\right) }\right\vert $, where $\left\langle \left(
F\left( X\right) \right) ^{\left( 1\right) }\cup \left( F\left( X\right)
\right) ^{\left( 2\right) }\right\rangle $ is a subalgebra of $F\left(
X\right) $ generated by set $\left( F\left( X\right) \right) ^{\left(
1\right) }\cup \left( F\left( X\right) \right) ^{\left( 2\right) }$.
\end{proof}

\begin{proposition}
Condition \ref{monoiso} fulfills in the variety of the all automatons.
\end{proposition}

\begin{proof}
The automorphism $\varphi $ of the additive monoid $%
%TCIMACRO{\U{2115} }%
%BeginExpansion
\mathbb{N}
%EndExpansion
^{3}$ induced by automorphism $\Phi $\ of the category $\Theta ^{0}$
permutes the minimal set of generators of this monoid $E=\left\{
e_{1},e_{2},e_{3}\right\} $. We will prove that only the identity
permutation is possible. For this aim we will construct the embedding $\iota
:A=F\left( x_{1}^{\left( 1\right) },x_{2}^{\left( 1\right) },x_{1}^{\left(
2\right) },\ldots ,x_{n}^{\left( 2\right) }\right) \hookrightarrow B=F\left(
x_{1}^{\left( 1\right) },x_{2}^{\left( 1\right) },x_{1}^{\left( 2\right)
}\right) $ for every $n>1$. There exists homomorphism $\iota :A\rightarrow B$
such that $\iota \left( x_{i}^{\left( 1\right) }\right) =x_{i}^{\left(
1\right) }$, $i=1,2$, $\iota \left( x_{j}^{\left( 2\right) }\right)
=x_{2}^{\left( 1\right) }\ast \left( x_{1}^{\left( 1\right) }\ast \left(
\ldots \ast \left( x_{1}^{\left( 1\right) }\ast x_{1}^{\left( 2\right)
}\right) \right) \right) =y_{j}^{\left( 2\right) }$, where $x_{1}^{\left(
1\right) }$ appear $j$ times in the sequence $x_{2}^{\left( 1\right) }\ast
\left( x_{1}^{\left( 1\right) }\ast \left( \ldots \ast \left( x_{1}^{\left(
1\right) }\ast x_{1}^{\left( 2\right) }\right) \right) \right) $ and $1\leq
j\leq n$. This homomorphism is embedding, because for every element of the
subalgebra $\left\langle x_{1}^{\left( 1\right) },x_{2}^{\left( 1\right)
},y_{1}^{\left( 2\right) },\ldots ,y_{n}^{\left( 2\right) }\right\rangle
\subset F\left( x_{1}^{\left( 1\right) },x_{2}^{\left( 1\right)
},x_{1}^{\left( 2\right) }\right) $ we can uniquely calculate its coimage
according $\iota $.

If automorphism $\varphi $ of the additive monoid $%
%TCIMACRO{\U{2115} }%
%BeginExpansion
\mathbb{N}
%EndExpansion
^{3}$ induced by automorphism $\Phi $\ of the category $\Theta ^{0}$
permutes set $E$ by permutation $\left( 1,2\right) $, i., e., $\Phi \left(
F\left( x^{\left( 1\right) }\right) \right) =F\left( x^{\left( 2\right)
}\right) $, $\Phi \left( F\left( x^{\left( 2\right) }\right) \right)
=F\left( x^{\left( 1\right) }\right) $, $\Phi \left( F\left( x^{\left(
3\right) }\right) \right) =F\left( x^{\left( 3\right) }\right) $, then $\Phi
\left( \iota \right) $ must be an embedding $\Phi \left( A\right) =F\left(
x_{1}^{\left( 1\right) },\ldots ,x_{n}^{\left( 1\right) },x_{1}^{\left(
2\right) },x_{2}^{\left( 2\right) }\right) \hookrightarrow \Phi \left(
B\right) =F\left( x_{1}^{\left( 1\right) },x_{1}^{\left( 2\right)
},x_{2}^{\left( 2\right) }\right) $. But this embedding can not exist,
because $\left\vert \left( \Phi \left( A\right) \right) ^{\left( 1\right)
}\right\vert =n$, $\left\vert \left( \Phi \left( B\right) \right) ^{\left(
1\right) }\right\vert =1$. If $\varphi $ permutes $E$ by permutation $\left(
1,3\right) $, i., e., $\Phi \left( F\left( x^{\left( 1\right) }\right)
\right) =F\left( x^{\left( 3\right) }\right) $, $\Phi \left( F\left(
x^{\left( 3\right) }\right) \right) =F\left( x^{\left( 1\right) }\right) $, $%
\Phi \left( F\left( x^{\left( 2\right) }\right) \right) =F\left( x^{\left(
2\right) }\right) $, then $\Phi \left( \iota \right) $ must be an embedding $%
\Phi \left( A\right) =F\left( x_{1}^{\left( 2\right) },\ldots ,x_{n}^{\left(
2\right) },x_{1}^{\left( 3\right) },x_{2}^{\left( 3\right) }\right)
\hookrightarrow \Phi \left( B\right) =F\left( x_{1}^{\left( 2\right)
},x_{1}^{\left( 3\right) },x_{2}^{\left( 3\right) }\right) $. But this
embedding can not exist, because $\left\vert \left( \Phi \left( A\right)
\right) ^{\left( 2\right) }\right\vert =n$, $\left\vert \left( \Phi \left(
B\right) \right) ^{\left( 2\right) }\right\vert =1$. If $\varphi $ permutes $%
E$ by permutation $\left( 2,3\right) $, then $\Phi \left( \iota \right) $
must be an embedding $\Phi \left( A\right) =F\left( x_{1}^{\left( 1\right)
},x_{2}^{\left( 1\right) },x_{1}^{\left( 3\right) },\ldots ,x_{n}^{\left(
3\right) }\right) \hookrightarrow \Phi \left( B\right) =F\left(
x_{1}^{\left( 1\right) },x_{2}^{\left( 1\right) },x_{1}^{\left( 3\right)
}\right) $. But this embedding can not exist, because $\left\vert \left(
\Phi \left( A\right) \right) ^{\left( 3\right) }\right\vert =n$, $\left\vert
\left( \Phi \left( B\right) \right) ^{\left( 3\right) }\right\vert =1$. If $%
\varphi $ permutes $E$ by permutation $\left( 1,2,3\right) $, then $\Phi
\left( \iota \right) $ must be an embedding $\Phi \left( A\right) =F\left(
x_{1}^{\left( 2\right) },x_{2}^{\left( 2\right) },x_{1}^{\left( 3\right)
},\ldots ,x_{n}^{\left( 3\right) }\right) \hookrightarrow \Phi \left(
B\right) =F\left( x_{1}^{\left( 2\right) },x_{2}^{\left( 2\right)
},x_{1}^{\left( 3\right) }\right) $. But this embedding can not exist,
because $\left\vert \left( \Phi \left( A\right) \right) ^{\left( 3\right)
}\right\vert =n$, $\left\vert \left( \Phi \left( B\right) \right) ^{\left(
3\right) }\right\vert =1$. If $\varphi $ permutes $E$ by permutation $\left(
1,3,2\right) $, then $\Phi \left( \iota \right) $ must be an embedding $\Phi
\left( A\right) =F\left( x_{1}^{\left( 1\right) },\ldots ,x_{n}^{\left(
1\right) },x_{1}^{\left( 3\right) },x_{2}^{\left( 3\right) }\right)
\hookrightarrow \Phi \left( B\right) =F\left( x_{1}^{\left( 1\right)
},x_{1}^{\left( 3\right) },x_{2}^{\left( 3\right) }\right) $. But this
embedding can not exist, because $\left\vert \left( \Phi \left( A\right)
\right) ^{\left( 1\right) }\right\vert =n$, $\left\vert \left( \Phi \left(
B\right) \right) ^{\left( 1\right) }\right\vert =1$. So only the identity
permutation is possible. Therefore $\Phi \left( F\left( x^{\left( i\right)
}\right) \right) =F\left( x^{\left( i\right) }\right) $, $i=1,2,3$.
\end{proof}

\begin{theorem}
The automorphic equivalence in the variety of the all automatons coincides
with the geometric equivalence.
\end{theorem}

\begin{proof}
We will find systems of words $W=\left\{ w_{\ast },w_{\circ }\right\} $
subjects of conditions Op1) and Op2). But we have only one possibility: $%
w_{\ast }\left( x^{\left( 1\right) },x^{\left( 2\right) }\right) =x^{\left(
1\right) }\ast x^{\left( 2\right) }$, $w_{\circ }\left( x^{\left( 1\right)
},x^{\left( 2\right) }\right) =x^{\left( 1\right) }\circ x^{\left( 2\right)
} $. As in the proof of Theorem \ref{assTheorem} these words can not be
shorter or longer, because, otherwise, we can not construct a bijection $%
s_{F\left( x^{\left( 1\right) },x^{\left( 2\right) }\right) }:F\left(
x^{\left( 1\right) },x^{\left( 2\right) }\right) \rightarrow F\left(
x^{\left( 1\right) },x^{\left( 2\right) }\right) $ subject of condition Op2).
\end{proof}

\subsection{Representations of groups.}

In this Subsection we reprove one result of \cite{PlotkinZhitAutCat}. $%
\Theta $ will be a variety of the all representations of groups over linear
spaces over field $k$. We assume that $k$ has a characteristic $0$. $\Gamma
=\left\{ 1,2\right\} $: the first sort is a sort of elements of groups, the
second sort is a sort of vectors of linear spaces. $\Omega =\left\{
1,-1,\cdot ,0,-,\lambda \left( \lambda \in k\right) ,+,\circ \right\} $. $1$
is a $0$-ary operation of the taking a unit in the group, $\tau _{1}=\left(
1\right) $. $-1$ is an unary operation of the taking an inverse element in
the group, $\tau _{-1}=\left( 1;1\right) $. $\cdot $ is an operation of the
multiplication of elements of the group, $\tau _{\cdot }=\left( 1,1;1\right) 
$. $0$ is $0$-ary operation of the taking a zero vector in the linear space, 
$\tau _{0}=\left( 2\right) $. $-$ is an unary operation of the taking a
negative vector in the linear space, $\tau _{-}=\left( 2;2\right) $. For
every $\lambda \in k$ we have an unary operation: multiplication of vectors
from the linear space by scalar $\lambda $. We denote this operation by $%
\lambda $ and $\tau _{\lambda }=\left( 2;2\right) $. $+$ is an operation of
the addition of vectors of the linear space, $\tau _{-}=\left( 2;2\right) $. 
$\tau _{+}=\left( 2,2;2\right) $. $\circ $ is an operation of the action of
elements of the group on vectors from the linear space, $\tau _{\circ
}=\left( 1,2;2\right) $.

A free algebras $F\left( X\right) $ in our variety has this form: $\left(
F\left( X\right) \right) ^{\left( 1\right) }=G\left( X^{\left( 1\right)
}\right) $, where $G\left( X^{\left( 1\right) }\right) $ is a free group
with the set of free generators $X^{\left( 1\right) }$. $\left( F\left(
X\right) \right) ^{\left( 2\right) }=kG\left( X^{\left( 1\right) }\right)
\circ X^{\left( 2\right) }=\bigoplus\limits_{i\in I}\left( kG\left(
X^{\left( 1\right) }\right) \circ x_{i}^{\left( 2\right) }\right) $, where $%
kG\left( X^{\left( 1\right) }\right) $ is a group $k$-algebra of the group $%
G\left( X^{\left( 1\right) }\right) $ and $kG\left( X^{\left( 1\right)
}\right) \circ X^{\left( 2\right) }$ is a free $kG\left( X^{\left( 1\right)
}\right) $-module generated by the set $X^{\left( 2\right) }=\left\{
x_{i}^{\left( 2\right) }\mid i\in I\right\} $ of the free generators. We
understand $\left( \sum\limits_{g\in I}\lambda _{g}g\right) \circ
x_{i}^{\left( 2\right) }$, where $g\in G\left( X^{\left( 1\right) }\right) $%
, $\lambda _{g}\in k$, $I\subset G\left( X^{\left( 1\right) }\right) $, $%
\left\vert I\right\vert <\infty $, as $\sum\limits_{g\in I}\lambda
_{g}\left( g\circ x_{i}^{\left( 2\right) }\right) $.

\begin{proposition}
The variety of the all representations of groups over linear spaces over
field $k$ has an IBN propriety.
\end{proposition}

\begin{proof}
$\left\vert X^{\left( 1\right) }\right\vert $ is equal to the rang of the
free abelian group $G\left( X^{\left( 1\right) }\right) /\left[ G\left(
X^{\left( 1\right) }\right) ,G\left( X^{\left( 1\right) }\right) \right] $.

We consider in $kG\left( X^{\left( 1\right) }\right) $ the ideal of
augmentation $\Delta $. We denote $kG\left( X^{\left( 1\right) }\right)
\circ X^{\left( 2\right) }=M$, $G\left( X^{\left( 1\right) }\right) =G$. $%
\Delta \circ M$ will be a submodule of the module $M$ and quotient module $%
M/\Delta \circ M$ will be a $kG\left( X^{\left( 1\right) }\right) /\Delta $%
-module, i., e., a $k$-linear space. We will prove that $\mathrm{Sp}%
_{k}\left( x_{i}^{\left( 2\right) }+\Delta \circ M\mid i\in I\right)
=M/\Delta \circ M$. Indeed, for every $m\in M$ the $m=\sum\limits_{i\in
I}f_{i}\circ x_{i}^{\left( 2\right) }$ holds, where $f_{i}\in kG$. $%
f_{i}=\delta \left( f_{i}\right) 1+h_{i}$, where $\delta :kG\rightarrow k$
is a homomorphism of augmentation, $h_{i}\in \Delta $, $1\in G$. So $m\equiv
\sum\limits_{i\in I}\delta \left( f_{i}\right) x_{i}^{\left( 2\right)
}\left( \func{mod}\Delta \circ M\right) $. We will prove that $\left\{
x_{i}^{\left( 2\right) }+\Delta \circ M\mid i\in I\right\} $ is a linear
independent set. Indeed, if $\sum\limits_{i\in I}\alpha _{i}x_{i}^{\left(
2\right) }\in \Delta \circ M$, where $\left\{ \alpha _{i}\mid i\in I\right\} 
$, then $\sum\limits_{i\in I}\alpha _{i}x_{i}^{\left( 2\right)
}=\sum\limits_{i\in I}\alpha _{i}1\circ x_{i}^{\left( 2\right)
}=\sum\limits_{i\in I}h_{i}\circ x_{i}^{\left( 2\right) }$, where $h_{i}\in
\Delta $. $\left\{ x_{i}^{\left( 2\right) }\mid i\in I\right\} $ is a set of
the free generators of the module $M$, so $\alpha _{i}1\in \Delta $ and $%
\alpha _{i}=0$ for every $i\in I$. Therefore $\left\vert X^{\left( 2\right)
}\right\vert =\dim _{k}\left( M/\Delta \circ M\right) $.
\end{proof}

\begin{proposition}
Condition (\ref{monoiso}) fulfills in the variety of the all representations
of groups.
\end{proposition}

\begin{proof}
The proof of this Proposition is similar to the proof of the Proposition \ref%
{ActCond}: we use the fact that we can embed the free group generated by $n$
free generators to the free group generated by $2$ free generators and the
fact that we can not embed the $n$-dimension linear space to the $2$%
-dimension linear space when $n>2$.
\end{proof}

\begin{theorem}
\label{grouprepgroup}$\mathfrak{A/Y\cong }\mathrm{Aut}k$ for the variety of
the all representations of groups, where $\mathrm{Aut}k$ is a group of the
all automorphisms of the field $k$.
\end{theorem}

\begin{proof}
We will find the systems of words 
\begin{equation}
W=\left\{ w_{1},w_{-1},w_{\cdot },w_{0},w_{-},w_{\lambda }\left( \lambda \in
k\right) ,w_{+},w_{\circ }\right\}  \label{groupssignature}
\end{equation}%
subjects of conditions Op1) and Op2). It is clear that for $w_{1},w_{0}\in
F\left( \varnothing \right) $ we have only one possibility: $w_{1}=1$, $%
w_{0}=0$. More precisely - $w_{1}=1^{\left( 1\right) }$, $w_{0}=0^{\left(
2\right) }$. By \cite{Morimoto} we have two possibilities: $w_{\cdot }\left(
x_{1}^{\left( 1\right) },x_{2}^{\left( 1\right) }\right) =x_{1}^{\left(
1\right) }\cdot x_{2}^{\left( 1\right) }$ or $w_{\cdot }\left( x_{1}^{\left(
1\right) },x_{2}^{\left( 1\right) }\right) =x_{2}^{\left( 1\right) }\cdot
x_{1}^{\left( 1\right) }$. After this by \cite[Proposition 2.1]{TsurNilp} we
have that for $w_{-1}\left( x^{\left( 1\right) }\right) \in F\left(
x^{\left( 1\right) }\right) $ we have only one possibility: $w_{-1}\left(
x^{\left( 1\right) }\right) =\left( x^{\left( 1\right) }\right) ^{-1}$.

By consideration which was used in \cite{PlotkinZhitAutCat} and in \cite%
{TsurAutomEqLinAlg}, for $w_{+}\left( x_{1}^{\left( 2\right) },x_{2}^{\left(
2\right) }\right) \in F\left( x_{1}^{\left( 2\right) },x_{2}^{\left(
2\right) }\right) $ we have only one possibility: $w_{+}\left( x_{1}^{\left(
2\right) },x_{2}^{\left( 2\right) }\right) =x_{1}^{\left( 2\right)
}+x_{2}^{\left( 2\right) }$. After this it is clear that for $w_{-}\left(
x^{\left( 2\right) }\right) \in F\left( x^{\left( 2\right) }\right) $ we
have only one possibility: $w_{-}\left( x^{\left( 2\right) }\right)
=-x^{\left( 2\right) }$. By consideration which was used in \cite%
{PlotkinZhitAutCat} and in \cite{TsurAutomEqLinAlg}, for $w_{\lambda }\left(
x^{\left( 2\right) }\right) \in F\left( x^{\left( 2\right) }\right) $ the $%
w_{\lambda }\left( x^{\left( 2\right) }\right) =\varphi \left( \lambda
\right) x^{\left( 2\right) }$ holds, where $\varphi \in \mathrm{Aut}k$.

$w_{\circ }\left( x^{\left( 1\right) },x^{\left( 2\right) }\right) \in
\left( F\left( x^{\left( 1\right) },x^{\left( 2\right) }\right) \right)
^{\left( 2\right) }$.$\left( F\left( x^{\left( 1\right) },x^{\left( 2\right)
}\right) \right) ^{\left( 2\right) }$ $=kG\left( x^{\left( 1\right) }\right)
\circ x^{\left( 2\right) }$. Therefore $w_{\circ }\left( x^{\left( 1\right)
},x^{\left( 2\right) }\right) =w\left( x^{\left( 1\right) }\right) \circ
x^{\left( 2\right) }$, where $w\left( x^{\left( 1\right) }\right) \in
kG\left( x^{\left( 1\right) }\right) $. $s_{F\left( x^{\left( 1\right)
},x^{\left( 2\right) }\right) }:F\left( x^{\left( 1\right) },x^{\left(
2\right) }\right) \rightarrow \left( F\left( x^{\left( 1\right) },x^{\left(
2\right) }\right) \right) _{W}^{\ast }$ is an isomorphism, so for every $%
g_{1},g_{2}\in G\left( x^{\left( 1\right) }\right) $ the $w_{\circ }\left(
g_{1},w_{\circ }\left( g_{2},x^{\left( 2\right) }\right) \right) =w_{\circ
}\left( w_{\cdot }\left( g_{1},g_{2}\right) ,x^{\left( 2\right) }\right) $
must hold. $G\left( x^{\left( 1\right) }\right) $ is a commutative group, so 
$w_{\cdot }\left( g_{1},g_{2}\right) =g_{1}g_{2}$. Therefore $w_{\circ
}\left( g_{1},w_{\circ }\left( g_{2},x^{\left( 2\right) }\right) \right)
=w_{\circ }\left( g_{1}g_{2},x^{\left( 2\right) }\right) $. $w_{\circ
}\left( g_{1},w_{\circ }\left( g_{2},x^{\left( 2\right) }\right) \right)
=w_{\circ }\left( g_{1},w\left( g_{2}\right) \circ x^{\left( 2\right)
}\right) =w\left( g_{1}\right) \circ \left( w\left( g_{2}\right) \circ
x^{\left( 2\right) }\right) =\left( w\left( g_{1}\right) w\left(
g_{2}\right) \right) \circ x^{\left( 2\right) }$, $w_{\circ }\left(
g_{1}g_{2},x^{\left( 2\right) }\right) =w\left( g_{1}g_{2}\right) \circ
x^{\left( 2\right) }$. Hence $w\left( g_{1}\right) w\left( g_{2}\right)
=w\left( g_{1}g_{2}\right) $ and $w:G\left( x^{\left( 1\right) }\right)
\rightarrow kG\left( x^{\left( 1\right) }\right) $ is a multiplicative
homomorphism. If we define $w\left( \sum\limits_{i=n}^{m}\alpha _{i}\left(
x^{\left( 1\right) }\right) ^{i}\right) =\sum\limits_{i=n}^{m}\alpha
_{i}w\left( x^{\left( 1\right) }\right) ^{i}$, where $\alpha _{i}\in k$, $%
n,m\in 
%TCIMACRO{\U{2124} }%
%BeginExpansion
\mathbb{Z}
%EndExpansion
$, we achieve that $w:kG\left( x^{\left( 1\right) }\right) \rightarrow
kG\left( x^{\left( 1\right) }\right) $ is a homomorphism. If we define $%
w\left( \frac{f\left( x^{\left( 1\right) }\right) }{h\left( x^{\left(
1\right) }\right) }\right) =\frac{w\left( f\left( x^{\left( 1\right)
}\right) \right) }{w\left( h\left( x^{\left( 1\right) }\right) \right) }$,
where $f\left( x^{\left( 1\right) }\right) ,h\left( x^{\left( 1\right)
}\right) \in k\left[ x^{\left( 1\right) }\right] $, we can extend this $w$
to the endomorphism of the field $k\left( x^{\left( 1\right) }\right) $. The 
$w_{\circ }\left( 1,x^{\left( 2\right) }\right) =w\left( 1\right) \circ
x^{\left( 2\right) }=1\circ x^{\left( 2\right) }$ must hold, so $w\left(
1\right) =1$ and $\ker w=0$. $\mathrm{im}s_{F\left( x^{\left( 1\right)
},x^{\left( 2\right) }\right) }\supset \left( F\left( x^{\left( 1\right)
},x^{\left( 2\right) }\right) \right) ^{\left( 2\right) }=kG\left( x^{\left(
1\right) }\right) \circ x^{\left( 2\right) }$, so for every $f\left(
x^{\left( 1\right) }\right) \in kG\left( x^{\left( 1\right) }\right) \in
kG\left( x^{\left( 1\right) }\right) $ there exists $h\left( x^{\left(
1\right) }\right) \in kG\left( x^{\left( 1\right) }\right) $ such that $%
f\left( x^{\left( 1\right) }\right) \circ x^{\left( 2\right) }=s_{F\left(
x^{\left( 1\right) },x^{\left( 2\right) }\right) }\left( h\left( x^{\left(
1\right) }\right) \circ x^{\left( 2\right) }\right) =w_{\circ }\left(
h\left( x^{\left( 1\right) }\right) \circ x^{\left( 2\right) }\right)
=w\left( h\left( x^{\left( 1\right) }\right) \right) \circ x^{\left(
2\right) }$ and $f\left( x^{\left( 1\right) }\right) =w\left( h\left(
x^{\left( 1\right) }\right) \right) $. So $\mathrm{im}w\supset kG\left(
x^{\left( 1\right) }\right) $ and $\mathrm{im}w=k\left( x^{\left( 1\right)
}\right) $. Therefore $w$ is an automorphism of the $k\left( x^{\left(
1\right) }\right) $. So $w\left( x^{\left( 1\right) }\right) =\frac{%
ax^{\left( 1\right) }+b}{cx^{\left( 1\right) }+d}$, where $a,b,c,d\in k$, $%
ad-bc\neq 0$.

But $w\left( x^{\left( 1\right) }\right) =\sum\limits_{i=n}^{m}\alpha
_{i}\left( x^{\left( 1\right) }\right) ^{i}\in kG\left( x^{\left( 1\right)
}\right) $, $\alpha _{i}\in k$, $\alpha _{n},\alpha _{m}\neq 0$, $n,m\in 
%TCIMACRO{\U{2124} }%
%BeginExpansion
\mathbb{Z}
%EndExpansion
$. So $ax^{\left( 1\right) }+b=w\left( x^{\left( 1\right) }\right) \left(
cx^{\left( 1\right) }+d\right) $. If $a,b,c,d\neq 0$ then $m+1=1$, $n=0$, so 
$w\left( x^{\left( 1\right) }\right) \in k$. This is a contradiction with $%
ad-bc\neq 0$. If $a=0$ then $b\neq 0$, $c\neq 0$, $w\left( x^{\left(
1\right) }\right) =\frac{1}{\alpha x^{\left( 1\right) }+\beta }$ and $%
w\left( x^{\left( 1\right) }\right) \left( \alpha x^{\left( 1\right) }+\beta
\right) =1$, where $\alpha ,\beta \in k$, $\alpha \neq 0$. If in this case $%
\beta \neq 0$, then $m+1=0$, $n=0$. It is impossible. So $\beta =0$ and $%
w\left( x^{\left( 1\right) }\right) =\alpha ^{-1}\left( x^{\left( 1\right)
}\right) ^{-1}$. The $w\left( \left( x^{\left( 1\right) }\right) ^{2}\right)
=\left( w\left( x^{\left( 1\right) }\right) \right) ^{2}$ must holds, but $%
w\left( \left( x^{\left( 1\right) }\right) ^{2}\right) =\alpha ^{-1}\left(
x^{\left( 1\right) }\right) ^{-2}$, $\left( w\left( x^{\left( 1\right)
}\right) \right) ^{2}=\alpha ^{-2}\left( x^{\left( 1\right) }\right) ^{-2}$.
Hence $\alpha =1$ and $w\left( x^{\left( 1\right) }\right) =\left( x^{\left(
1\right) }\right) ^{-1}$. If $b=0$ then $d\neq 0$ and $w\left( x^{\left(
1\right) }\right) =\frac{ax^{\left( 1\right) }}{cx^{\left( 1\right) }+d}$.
So $w\left( x^{\left( 1\right) }\right) \left( cx^{\left( 1\right)
}+d\right) =ax^{\left( 1\right) }$ and $m+1=1$, $n=1$. It is impossible. If $%
c=0$ then $d\neq 0$, $a\neq 0$ and $w\left( x^{\left( 1\right) }\right)
=\alpha x^{\left( 1\right) }+\beta $, where $\alpha ,\beta \in k$, $\alpha
\neq 0$. From $w\left( \left( x^{\left( 1\right) }\right) ^{2}\right)
=\left( w\left( x^{\left( 1\right) }\right) \right) ^{2}$ we conclude that $%
\alpha \left( x^{\left( 1\right) }\right) ^{2}+\beta =\alpha ^{2}\left(
x^{\left( 1\right) }\right) ^{2}+2\alpha \beta x^{\left( 1\right) }+\beta
^{2}$ and $\beta =0$, $\alpha =1$. So $w\left( x^{\left( 1\right) }\right)
=x^{\left( 1\right) }$. If $d=0$ then then $c\neq 0$, $b\neq 0$ and $w\left(
x^{\left( 1\right) }\right) =\alpha \left( x^{\left( 1\right) }\right)
^{-1}+\beta $, where $\alpha ,\beta \in k$, $\alpha \neq 0$. As above from $%
w\left( \left( x^{\left( 1\right) }\right) ^{2}\right) =\left( w\left(
x^{\left( 1\right) }\right) \right) ^{2}$ we conclude that $w\left(
x^{\left( 1\right) }\right) =\left( x^{\left( 1\right) }\right) ^{-1}$.
Therefore we have two possibilities: $w_{\circ }\left( x^{\left( 1\right)
},x^{\left( 2\right) }\right) =x^{\left( 1\right) }\circ x^{\left( 2\right)
} $ or $w_{\circ }\left( x^{\left( 1\right) },x^{\left( 2\right) }\right)
=\left( x^{\left( 1\right) }\right) ^{-1}\circ x^{\left( 2\right) }$.

We will prove that the situation when $w_{\cdot }\left( x_{1}^{\left(
1\right) },x_{2}^{\left( 1\right) }\right) =x_{1}^{\left( 1\right) }\cdot
x_{2}^{\left( 1\right) }$, $w_{\circ }\left( x^{\left( 1\right) },x^{\left(
2\right) }\right) =\left( x^{\left( 1\right) }\right) ^{-1}\circ x^{\left(
2\right) }$ is impossible. Indeed, there is an isomorphism $s_{F\left(
x_{1}^{\left( 1\right) },x_{2}^{\left( 1\right) },x^{\left( 2\right)
}\right) }:F\left( x_{1}^{\left( 1\right) },x_{2}^{\left( 1\right)
},x^{\left( 2\right) }\right) \rightarrow \left( F\left( x_{1}^{\left(
1\right) },x_{2}^{\left( 1\right) },x^{\left( 2\right) }\right) \right)
_{W}^{\ast }$ , so the $w_{\circ }\left( x_{1}^{\left( 1\right) },w_{\circ
}\left( x_{2}^{\left( 1\right) },x^{\left( 2\right) }\right) \right)
=w_{\circ }\left( w_{\cdot }\left( x_{1}^{\left( 1\right) },x_{2}^{\left(
1\right) }\right) ,x^{\left( 2\right) }\right) $ must hold. But in our
situation 
\begin{equation*}
w_{\circ }\left( x_{1}^{\left( 1\right) },w_{\circ }\left( x_{2}^{\left(
1\right) },x^{\left( 2\right) }\right) \right) =\left( x_{1}^{\left(
1\right) }\right) ^{-1}\circ \left( \left( x_{2}^{\left( 1\right) }\right)
^{-1}\circ x^{\left( 2\right) }\right) =\left( \left( x_{1}^{\left( 1\right)
}\right) ^{-1}\cdot \left( x_{2}^{\left( 1\right) }\right) ^{-1}\right)
\circ x^{\left( 2\right) }
\end{equation*}
and 
\begin{equation*}
w_{\circ }\left( w_{\cdot }\left( x_{1}^{\left( 1\right) },x_{2}^{\left(
1\right) }\right) ,x^{\left( 2\right) }\right) =\left( x_{1}^{\left(
1\right) }\cdot x_{2}^{\left( 1\right) }\right) ^{-1}\circ x^{\left(
2\right) }=\left( \left( x_{2}^{\left( 1\right) }\right) ^{-1}\cdot \left(
x_{1}^{\left( 1\right) }\right) ^{-1}\right) \circ x^{\left( 2\right) }.
\end{equation*}

Also the situation when $w_{\cdot }\left( x_{1}^{\left( 1\right)
},x_{2}^{\left( 1\right) }\right) =x_{2}^{\left( 1\right) }\cdot
x_{1}^{\left( 1\right) }$, $w_{\circ }\left( x^{\left( 1\right) },x^{\left(
2\right) }\right) =x^{\left( 1\right) }\circ x^{\left( 2\right) }$ is
impossible, because in this situation $w_{\circ }\left( x_{1}^{\left(
1\right) },w_{\circ }\left( x_{2}^{\left( 1\right) },x^{\left( 2\right)
}\right) \right) =x_{1}^{\left( 1\right) }\circ \left( x_{2}^{\left(
1\right) }\circ x^{\left( 2\right) }\right) =\left( x_{1}^{\left( 1\right)
}\cdot x_{2}^{\left( 1\right) }\right) \circ x^{\left( 2\right) }$ and $%
w_{\circ }\left( w_{\cdot }\left( x_{1}^{\left( 1\right) },x_{2}^{\left(
1\right) }\right) ,x^{\left( 2\right) }\right) =\left( x_{2}^{\left(
1\right) }\cdot x_{1}^{\left( 1\right) }\right) \circ x^{\left( 2\right) }$.

So for systems of the words (\ref{groupssignature}) we have these
possibilities:%
\begin{equation*}
W=\{w_{1}=1^{\left( 1\right) },w_{-1}\left( x^{\left( 1\right) }\right)
=\left( x^{\left( 1\right) }\right) ^{-1},w_{\cdot }\left( x_{1}^{\left(
1\right) },x_{2}^{\left( 1\right) }\right) =x_{1}^{\left( 1\right) }\cdot
x_{2}^{\left( 1\right) },
\end{equation*}%
\begin{equation}
w_{0}=0^{\left( 2\right) },w_{-}\left( x^{\left( 2\right) }\right)
=-x^{\left( 2\right) },w_{\lambda }\left( x^{\left( 2\right) }\right)
=\varphi \left( \lambda \right) x^{\left( 2\right) }\left( \lambda \in
k\right) ,  \label{groupreallist1}
\end{equation}%
\begin{equation*}
w_{+}\left( x_{1}^{\left( 2\right) },x_{2}^{\left( 2\right) }\right)
=x_{1}^{\left( 2\right) }+x_{2}^{\left( 2\right) },w_{\circ }\left(
x^{\left( 1\right) },x^{\left( 2\right) }\right) =x^{\left( 1\right) }\circ
x^{\left( 2\right) }\}
\end{equation*}%
or%
\begin{equation*}
W=\{w_{1}=1^{\left( 1\right) },w_{-1}\left( x^{\left( 1\right) }\right)
=\left( x^{\left( 1\right) }\right) ^{-1},w_{\cdot }\left( x_{1}^{\left(
1\right) },x_{2}^{\left( 1\right) }\right) =x_{2}^{\left( 1\right) }\cdot
x_{1}^{\left( 1\right) },
\end{equation*}%
\begin{equation}
w_{0}=0^{\left( 2\right) },w_{-}\left( x^{\left( 2\right) }\right)
=-x^{\left( 2\right) },w_{\lambda }\left( x^{\left( 2\right) }\right)
=\varphi \left( \lambda \right) x^{\left( 2\right) }\left( \lambda \in
k\right) ,  \label{groupreallist2}
\end{equation}%
\begin{equation*}
w_{+}\left( x_{1}^{\left( 2\right) },x_{2}^{\left( 2\right) }\right)
=x_{1}^{\left( 2\right) }+x_{2}^{\left( 2\right) },w_{\circ }\left(
x^{\left( 1\right) },x^{\left( 2\right) }\right) =\left( x^{\left( 1\right)
}\right) ^{-1}\circ x^{\left( 2\right) }\},
\end{equation*}%
where $\varphi \in \mathrm{Aut}k$.

Now we will prove that conditions for Op1) and Op2) fulfills for all these
systems of the words. By direct computations we can prove that if $H\in
\Theta $ then, for all these systems of the words $W$, in the algebra $%
H_{W}^{\ast }$ fulfill all identities (axioms) which define our variety $%
\Theta $. So $H_{W}^{\ast }\in \Theta $. Therefore for every $F=F\left(
X\right) \in \mathrm{Ob}\Theta ^{0}$ there exists homomorphism $%
s_{F}:F\rightarrow F_{W}^{\ast }$ such that $\left( s_{F}\right) _{\mid
X}=id_{X}$.

We will prove that $s_{F}$ is an an isomorphism. We assume that $W$ has form
(\ref{groupreallist2}). First of all we prove that $s_{F}$ is a
monomorphism. We consider $g=\left( x_{i_{1}}^{\left( 1\right) }\right)
^{\varepsilon _{1}}\cdot \ldots \cdot \left( x_{i_{n}}^{\left( 1\right)
}\right) ^{\varepsilon _{n}}\in G\left( X^{\left( 1\right) }\right) $, where 
$x_{i_{1}}^{\left( 1\right) },\ldots ,x_{i_{n}}^{\left( 1\right) }\in
X^{\left( 1\right) }$, $\varepsilon _{1},\ldots ,\varepsilon _{n}\in \left\{
1,-1\right\} $, $n\in 
%TCIMACRO{\U{2115} }%
%BeginExpansion
\mathbb{N}
%EndExpansion
$ and in the word we can not make cancellations. We also can not make
cancellations in the word $s_{F}\left( g\right) =\left( x_{i_{n}}^{\left(
1\right) }\right) ^{\varepsilon _{n}}\cdot \ldots \cdot \left(
x_{i_{1}}^{\left( 1\right) }\right) ^{\varepsilon _{1}}$. So if $s_{F}\left(
g\right) =1^{\left( 1\right) }$ then $n=0$ and $g=1$. Now we consider $\sum
f_{i}\circ x_{i}^{\left( 2\right) }\in kG\left( X^{\left( 1\right) }\right)
\circ X^{\left( 2\right) }$, such that $s_{F}\left( \sum f_{i}\circ
x_{i}^{\left( 2\right) }\right) =\sum s_{F}\left( f_{i}\circ x_{i}^{\left(
2\right) }\right) =0$. If $f_{i}=\sum\limits_{g\in I}\lambda _{i,g}g\in
kG\left( X^{\left( 1\right) }\right) $, and $x_{i}^{\left( 2\right) }\in
X^{\left( 2\right) }$, then $s_{F}\left( f_{i}\circ x_{i}^{\left( 2\right)
}\right) =s_{F}\left( \sum\limits_{g\in I}\lambda _{i,g}\left( g\circ
x_{i}^{\left( 2\right) }\right) \right) =\sum\limits_{g\in I}\varphi \left(
\lambda _{i,g}\right) \left( g^{-1}\circ x_{i}^{\left( 2\right) }\right)
=\left( \sum\limits_{g\in I}\varphi \left( \lambda _{i,g}\right)
g^{-1}\right) \circ x_{i}^{\left( 2\right) }\in kG\left( X^{\left( 1\right)
}\right) \circ x_{i}^{\left( 2\right) }$, where $\varphi \in \mathrm{Aut}k$%
.. So, if $\sum s_{F}\left( f_{i}\circ x_{i}^{\left( 2\right) }\right) =0$
then the $s_{F}\left( f_{i}\circ x_{i}^{\left( 2\right) }\right) =0$ holds
for every $x_{i}^{\left( 2\right) }\in X^{\left( 2\right) }$ and $%
\sum\limits_{g\in I}\varphi \left( \lambda _{i,g}\right) g^{-1}=0$. Hence $%
\varphi \left( \lambda _{i,g}\right) =0$ and $\lambda _{i,g}=0$ holds for
every $g\in I$. Therefore $f_{i}\circ x_{i}^{\left( 2\right) }=0$ and $\sum
f_{i}\circ x_{i}^{\left( 2\right) }=0$. So $s_{F}$ is a monomorphism.

For every $g=\left( x_{i_{1}}^{\left( 1\right) }\right) ^{\varepsilon
_{1}}\cdot \ldots \cdot \left( x_{i_{n}}^{\left( 1\right) }\right)
^{\varepsilon _{n}}\in G\left( X^{\left( 1\right) }\right) $ the $%
s_{F}\left( \left( x_{i_{n}}^{\left( 1\right) }\right) ^{\varepsilon
_{n}}\cdot \ldots \cdot \left( x_{i_{1}}^{\left( 1\right) }\right)
^{\varepsilon _{1}}\right) =g$ holds. And for every $v=\sum f_{i}\circ
x_{i}^{\left( 2\right) }\in kG\left( X^{\left( 1\right) }\right) \circ
X^{\left( 2\right) }$, where $f_{i}=\sum\limits_{g\in I}\lambda _{i,g}g$,
the $s_{F}\left( \sum\limits_{g\in I}\varphi ^{-1}\left( \lambda
_{i,g}\right) \left( g^{-1}\circ x_{i}^{\left( 2\right) }\right) \right)
=f_{i}\circ x_{i}^{\left( 2\right) }$ holds. Therefore $v\in \mathrm{im}%
s_{F} $. Hence $s_{F}$ is an an isomorphism. If $W$ has form (\ref%
{groupreallist1}) we can even easier prove that $s_{F}$ is an isomorphism.

Now we know all elements of the group $\mathfrak{S}$: by Theorem \ref%
{methodofverbaloperations} they are all strongly stable automorphisms $\Phi
^{W}$, which correspond to the systems of words (\ref{groupreallist1}) and (%
\ref{groupreallist2}). Now we will study the multiplication in this group.
The automorphisms $\Phi ^{W}$, which correspond to the system of words (\ref%
{groupreallist1}) we will denote by $\Phi \left( \varphi ,1\right) $ and the
automorphisms $\Phi ^{W}$, which correspond to the system of words (\ref%
{groupreallist2}) we will denote by $\Phi \left( \varphi ,\sigma \right) $.
We consider $\Phi _{1},\Phi _{2}\in \mathfrak{S}$. We denote, as in the
Section \ref{verbal}, $S^{\Phi _{i}}=\left\{ s_{i,F}:F\rightarrow
F_{W_{i}}^{\ast }\mid F\in \mathrm{Ob}\Theta ^{0}\right\} $, where $%
W_{i}=W^{\Phi _{i}}$, $i=1,2$. The automorphism $\Phi _{2}\Phi _{1}$ acts on
morphisms of the category $\Theta ^{0}$ according the formula (\ref%
{biject_action}) by bijections $S^{\Phi _{2}\Phi _{1}}=\left\{
s_{2,F}s_{1,F}\mid F\in \mathrm{Ob}\Theta ^{0}\right\} $. The system of
words $W=W^{\Phi _{2}\Phi _{1}}$ can be calculated by the formula (\ref%
{wordformula}). If $\Phi _{i}=\Phi \left( \varphi _{i},\sigma \right) $, $%
i=1,2$, then $W=W^{\Phi _{2}\Phi _{1}}$ has form (\ref{groupssignature})
with $w_{\cdot }\left( x_{1}^{\left( 1\right) },x_{2}^{\left( 1\right)
}\right) =s_{2,F}s_{1,F}\left( x_{1}^{\left( 1\right) }\cdot x_{2}^{\left(
1\right) }\right) =s_{2,F}\left( x_{2}^{\left( 1\right) }\cdot x_{1}^{\left(
1\right) }\right) =x_{1}^{\left( 1\right) }\cdot x_{2}^{\left( 1\right) }$, $%
w_{\lambda }\left( x^{\left( 1\right) }\right) =s_{2,F}s_{1,F}\left( \lambda
x^{\left( 1\right) }\right) =s_{2,F}\left( \varphi _{1}\left( \lambda
\right) x^{\left( 1\right) }\right) =\varphi _{2}\varphi _{1}\left( \lambda
\right) x^{\left( 1\right) }$ and $w_{\circ }\left( x^{\left( 1\right)
},x^{\left( 2\right) }\right) =s_{2,F}s_{1,F}\left( x^{\left( 1\right)
}\circ x^{\left( 2\right) }\right) =s_{2,F}\left( \left( x^{\left( 1\right)
}\right) ^{-1}\circ x^{\left( 2\right) }\right) =x^{\left( 1\right) }\circ
x^{\left( 2\right) }$. Therefore $\Phi \left( \varphi _{2},\sigma \right)
\Phi \left( \varphi _{1},\sigma \right) =\Phi \left( \varphi _{2}\varphi
_{1},1\right) $. By similar calculation we prove that $\Phi \left( \varphi
_{2},\sigma \right) \Phi \left( \varphi _{1},1\right) =$ $\Phi \left(
\varphi _{2},1\right) \Phi \left( \varphi _{1},\sigma \right) =\Phi \left(
\varphi _{2}\varphi _{1},\sigma \right) $ and $\Phi \left( \varphi
_{2},1\right) \Phi \left( \varphi _{1},1\right) =\Phi \left( \varphi
_{2}\varphi _{1},1\right) $. Therefore $\mathfrak{S\cong }\mathrm{Aut}%
k\times 
%TCIMACRO{\U{2124} }%
%BeginExpansion
\mathbb{Z}
%EndExpansion
_{2}$.

Now we must calculate the group $\mathfrak{S\cap Y}$. If $\Phi =\Phi \left(
\varphi ,\sigma \right) $ or $\Phi =\Phi \left( \varphi ,1\right) $ and $%
\varphi \neq id_{k}$ then $\Phi $ is not inner automorphism. The proof of
this fact is even easier than proof of the similar fact from \cite[%
Proposition 4.1]{TsurAutomEqLinAlg}. If $\Phi =\Phi \left( id_{k},\sigma
\right) $ than for every $F\in \mathrm{Ob}\Theta ^{0}$ we define the mapping 
$\tau _{F}:F\rightarrow F$ by this way: if $g\in F^{\left( 1\right) }$ then $%
\tau _{F}\left( g\right) =g^{-1}$, if $v\in F^{\left( 2\right) }$ then $\tau
_{F}\left( v\right) =v$. We will that prove $\tau _{F}:F\rightarrow \left(
F\right) _{W}^{\ast }$, where $W=W^{\Phi }$, is an isomorphism. For every $%
g\in F^{\left( 1\right) }$ the $\tau _{F}\left( g^{-1}\right) =\left(
g^{-1}\right) ^{-1}=g=w_{-1}\left( \tau _{F}\left( g\right) \right) $ holds.
Also $g_{1},g_{2}\in F^{\left( 1\right) }$ the $\tau _{F}\left( g_{1}\cdot
g_{2}\right) =\left( g_{1}\cdot g_{2}\right) ^{-1}=g_{2}^{-1}\cdot
g_{1}^{-1}=w_{\cdot }\left( \tau _{F}\left( g_{1}\right) ,\tau _{F}\left(
g_{2}\right) \right) $ holds for every $g_{1},g_{2}\in F^{\left( 1\right) }$%
. And the $w_{\circ }\left( \tau _{F}\left( g\right) \circ \tau _{F}\left(
v\right) \right) =w_{\circ }\left( g^{-1}\circ v\right) =g\circ v=\tau
_{F}\left( g\circ v\right) $ holds for every $g\in F^{\left( 1\right) }$, $%
v\in F^{\left( 2\right) }$. Also it is clear that $\tau _{F}$ is a
bijection. So $\tau _{F}:F\rightarrow \left( F\right) _{W}^{\ast }$ is an
isomorphism. For every $F_{1},F_{2}\in \mathrm{Ob}\Theta ^{0}$ and every $%
\mu \in \mathrm{Mor}_{\Theta ^{0}}\left( F_{1},F_{2}\right) $ we have that
the $\tau _{F_{2}}\mu \left( g\right) =\mu \left( g\right) ^{-1}=\mu \tau
_{F_{1}}\left( g\right) $ holds for every $g\in F_{1}^{\left( 1\right) }$
and the $\tau _{F_{2}}\mu \left( v\right) =\mu \left( v\right) =\mu \tau
_{F_{1}}\left( v\right) $ holds for every $v\in F_{1}^{\left( 2\right) }$.
So, by Proposition \ref{intersectioncriterion}, $\Phi $ is an inner
automorphism. Therefore $\mathfrak{A/Y\cong S/S\cap Y\cong }\mathrm{Aut}k$.
\end{proof}

Now we will give an example of two representations of groups which are
automorphically equivalent but not geometrically equivalent. In \cite%
{PlotkinZhitAutCat} this matter was not discussed. We take $k=%
%TCIMACRO{\U{211a} }%
%BeginExpansion
\mathbb{Q}
%EndExpansion
\left( \theta _{1},\theta _{2}\right) $ - the transcendental extension of
degree $2$ of the field $%
%TCIMACRO{\U{211a} }%
%BeginExpansion
\mathbb{Q}
%EndExpansion
$. We consider a free representation $F=F\left( x^{\left( 1\right)
},x^{\left( 2\right) }\right) $. $F^{\left( 1\right) }=G\left( x^{\left(
1\right) }\right) $, $F^{\left( 2\right) }=kG\left( x^{\left( 1\right)
}\right) \circ x^{\left( 2\right) }$. In the representation $F$ we will
consider the congruence $T$ generated by pair $\left( x^{\left( 1\right)
}\circ x^{\left( 2\right) },\theta _{1}x^{\left( 2\right) }\right) $. The
pair of the normal subgroup of $F^{\left( 1\right) }$ and $kG\left(
x^{\left( 1\right) }\right) $-submodule of $F^{\left( 2\right) }$ which
corresponds to this congruence (see \cite{PlotkinVovsi}) is $\left( \left\{
1^{\left( 1\right) }\right\} ,\left( x^{\left( 1\right) }-\theta _{1}\right)
kG\left( x^{\left( 1\right) }\right) \circ x^{\left( 2\right) }\right) $. We
will denote the quotient representation $F/T=H$. We will denote the natural
epimorphism $\tau :F\rightarrow F/T$. The equalities $\tau \left( x^{\left(
1\right) }\right) ^{m}\circ \tau \left( x^{\left( 2\right) }\right) =\theta
_{1}^{m}\tau \left( x^{\left( 2\right) }\right) $ hold in the $H$ for every $%
m\in 
%TCIMACRO{\U{2124} }%
%BeginExpansion
\mathbb{Z}
%EndExpansion
$. We will consider the automorphism $\varphi $ of the field $k$ such that $%
\varphi \left( \theta _{1}\right) =\theta _{2}$, $\varphi \left( \theta
_{2}\right) =\theta _{1}$. $W$ will be the system of words which has form (%
\ref{groupreallist1}) with the mentioned $\varphi $. By Corollary \ref%
{hwautomequiv} from Proposition \ref{sbijection} the representations $H$ and 
$H_{W}^{\ast }$ are automorphically equivalent.

\begin{proposition}
\label{groupContrExample}The representations $H$ and $H_{W}^{\ast }$ are not
geometrically equivalent.
\end{proposition}

\begin{proof}
$\left( T\right) _{H}^{\prime \prime }=\bigcap\limits_{\substack{ \psi \in 
\mathrm{Hom}\left( F,H\right)  \\ T\subseteq \ker \psi }}\ker \psi $, $\ker
\tau =T$, so $\left( T\right) _{H}^{\prime \prime }=T$ and $T\in Cl_{H}(F)$.
If $H$ and $H_{W}^{\ast }$ are geometrically equivalent then $T\in
Cl_{H_{W}^{\ast }}(F)$ so, by Theorem \ref{alfaformula} and Corollary \ref%
{hwautomequiv} from Proposition \ref{sbijection}, $s_{F}\left( T\right) \in
Cl_{H}(F)$. The pair of the normal subgroup of $F^{\left( 1\right) }$ and $%
kG\left( x^{\left( 1\right) }\right) $-submodule of $F^{\left( 2\right) }$
which corresponds to the congruence $s_{F}\left( T\right) $ is $\left(
\left\{ 1^{\left( 1\right) }\right\} ,\left( x^{\left( 1\right) }-\theta
_{2}\right) kG\left( x^{\left( 1\right) }\right) \circ x^{\left( 2\right)
}\right) $. $\left( s_{F}\left( T\right) \right) _{H}^{\prime \prime
}=\bigcap\limits_{\substack{ \psi \in \mathrm{Hom}\left( F,H\right)  \\ %
s_{F}\left( T\right) \subseteq \ker \psi }}\ker \psi $. If $\psi \in \mathrm{%
Hom}\left( F,H\right) $ then by projective propriety of the free algebras
there exists $\alpha \in \mathrm{End}\left( F\right) $ such that $\tau
\alpha =\psi $. $\left( x^{\left( 1\right) }\circ x^{\left( 2\right)
},\theta _{2}x^{\left( 2\right) }\right) \in s_{F}\left( T\right) $, so if $%
s_{F}\left( T\right) \subseteq \ker \psi $ then $\tau \left( \alpha \left(
x^{\left( 1\right) }\right) \circ \alpha \left( x^{\left( 2\right) }\right)
\right) =\tau \left( \theta _{2}\alpha \left( x^{\left( 2\right) }\right)
\right) $ and $\left( \alpha \left( x^{\left( 1\right) }\right) \circ \alpha
\left( x^{\left( 2\right) }\right) ,\theta _{2}\alpha \left( x^{\left(
2\right) }\right) \right) \in T$. $\alpha \left( x^{\left( 1\right) }\right)
=\left( x^{\left( 1\right) }\right) ^{m}$, $\alpha \left( x^{\left( 2\right)
}\right) =f\left( x^{\left( 1\right) }\right) \circ x^{\left( 2\right) }$,
where $m\in 
%TCIMACRO{\U{2124} }%
%BeginExpansion
\mathbb{Z}
%EndExpansion
$, $f\left( x^{\left( 1\right) }\right) \in kG\left( x^{\left( 1\right)
}\right) $. So $\left( \left( x^{\left( 1\right) }\right) ^{m}f\left(
x^{\left( 1\right) }\right) \circ x^{\left( 2\right) },\theta _{2}f\left(
x^{\left( 1\right) }\right) \circ x^{\left( 2\right) }\right) \in T$. It
means that $\left( \left( x^{\left( 1\right) }\right) ^{m}-\theta
_{2}\right) f\left( x^{\left( 1\right) }\right) \circ x^{\left( 2\right)
}\in \left( x^{\left( 1\right) }-\theta _{1}\right) kG\left( x^{\left(
1\right) }\right) \circ x^{\left( 2\right) }$ or $\left( \left( x^{\left(
1\right) }\right) ^{m}-\theta _{2}\right) f\left( x^{\left( 1\right)
}\right) \in \left( x^{\left( 1\right) }-\theta _{1}\right) kG\left(
x^{\left( 1\right) }\right) $. The ideal $\left( x^{\left( 1\right) }-\theta
_{1}\right) kG\left( x^{\left( 1\right) }\right) $ is a prime ideal in the $%
kG\left( x^{\left( 1\right) }\right) $. Therefore or $\left( x^{\left(
1\right) }\right) ^{m}-\theta _{2}\in \left( x^{\left( 1\right) }-\theta
_{1}\right) kG\left( x^{\left( 1\right) }\right) $, or $f\left( x^{\left(
1\right) }\right) \in \left( x^{\left( 1\right) }-\theta _{1}\right)
kG\left( x^{\left( 1\right) }\right) $. If $\left( x^{\left( 1\right)
}\right) ^{m}-\theta _{2}\in \left( x^{\left( 1\right) }-\theta _{1}\right)
kG\left( x^{\left( 1\right) }\right) $ then $\left( \theta _{1}\right)
^{m}-\theta _{2}=0$, but it is not possible. If $f\left( x^{\left( 1\right)
}\right) \in \left( x^{\left( 1\right) }-\theta _{1}\right) kG\left(
x^{\left( 1\right) }\right) $ then $\psi \left( x^{\left( 2\right) }\right)
=0$. So $\left( s_{F}\left( T\right) \right) _{H}^{\prime \prime
}=\bigcap\limits_{\substack{ \psi \in \mathrm{Hom}\left( F,H\right)  \\ %
s_{F}\left( T\right) \subseteq \ker \psi }}\ker \psi \supseteq \Delta
_{F^{\left( 1\right) }}\cup \left( F^{\left( 2\right) }\right)
^{2}\supsetneqq s_{F}\left( T\right) $. Therefore $s_{F}\left( T\right)
\notin Cl_{H}(F)$. This contradiction finishes the prove.
\end{proof}

\subsection{Representations of Lie algebras.}

In this Subsection we reprove result of \cite{ShestTsur}. $\Theta $ will be
a variety of the all representations of Lie algebras over linear spaces over
field $k$. We assume that $k$ has a characteristic $0$. $\Gamma =\left\{
1,2\right\} $: the first sort is a sort of elements Lie algebras, the second
sort is a sort of vectors of linear spaces. 
\begin{equation*}
\Omega =\left\{ 0^{\left( 1\right) },-^{\left( 1\right) },\lambda ^{\left(
1\right) }\left( \lambda \in k\right) ,+^{\left( 1\right) },\left[ ,\right]
,0^{\left( 2\right) },-^{\left( 2\right) },\lambda ^{\left( 2\right) }\left(
\lambda \in k\right) ,+^{\left( 2\right) },\circ \right\} .
\end{equation*}%
$0^{\left( 2\right) }$, $-^{\left( 2\right) }$, $\lambda ^{\left( 2\right)
}\left( \lambda \in k\right) $, $+^{\left( 2\right) }$ are the operations in
the linear space, $0^{\left( 1\right) }$, $-^{\left( 1\right) }$, $\lambda
^{\left( 1\right) }\left( \lambda \in k\right) $, $+^{\left( 1\right) }$ are
the similar operations in the Lie algebra. $\left[ ,\right] $ is the Lie
brackets; this operation has type $\tau _{\left[ ,\right] }=\left(
1,1;1\right) $. $\circ $ is an operation of the action of elements of the
Lie algebra on vectors from the linear space, $\tau _{\circ }=\left(
1,2;2\right) $.

A free algebras $F\left( X\right) $ in our variety has this form: $\left(
F\left( X\right) \right) ^{\left( 1\right) }=L\left( X^{\left( 1\right)
}\right) $, where $L\left( X^{\left( 1\right) }\right) $ is a free Lie
algebra with the set of free generators $X^{\left( 1\right) }$. $\left(
F\left( X\right) \right) ^{\left( 2\right) }=A\left( X^{\left( 1\right)
}\right) \circ X^{\left( 2\right) }=\bigoplus\limits_{i\in I}\left( A\left(
X^{\left( 1\right) }\right) \circ x_{i}^{\left( 2\right) }\right) $, where $%
A\left( X^{\left( 1\right) }\right) $ is a free associative algebra with
unit generated by the set of free generators $X^{\left( 1\right) }$ and $%
\left( x_{i_{1}}^{\left( 1\right) }\ldots x_{i_{n}}^{\left( 1\right)
}\right) \circ v^{\left( 2\right) }$ we understand as $x_{i_{1}}^{\left(
1\right) }\circ \left( \ldots \circ \left( x_{i_{n}}^{\left( 1\right) }\circ
v^{\left( 2\right) }\right) \right) $ and so on by linearity, $%
x_{i_{1}}^{\left( 1\right) },\ldots ,x_{i_{n}}^{\left( 1\right) }\in
X^{\left( 1\right) }$, $v^{\left( 2\right) }\in \left( F\left( X\right)
\right) ^{\left( 2\right) }$.

By \cite[Theorem 5.1]{ShestTsur} the variety $\Theta $ has an IBN propriety.

\begin{proposition}
Condition (\ref{monoiso}) fulfills in the variety of the all representations
of Lie algebras.
\end{proposition}

\begin{proof}
The proof of this Proposition is similar to the proof of the Proposition \ref%
{ActCond}. We use the following two facts. The first is: we can embed the
free Lie algebra generated by $n$ free generators to the free Lie algebra
generated by $2$ free generators, for example, by \cite[2.4.2]%
{BahturinLieIdent} the elements $\left[ x_{1},x_{2}\right] $, $\left[ x_{1},%
\left[ x_{1},x_{2}\right] \right] $, $\left[ x_{1},\left[ x_{1},\left[
x_{1},x_{2}\right] \right] \right] $ and so on are free in the algebra $%
L\left( x_{1},x_{2}\right) $. And the second is: we can not embed the $n$%
-dimension linear space to the $2$-dimension linear space when $n>2$.
\end{proof}

\begin{theorem}
$\mathfrak{A/Y\cong }\mathrm{Aut}k$ for the variety of the all
representations of Lie algebras.
\end{theorem}

\begin{proof}
We will find the systems of words%
\begin{equation}
W=\left\{ w_{0^{\left( 1\right) }},w_{-^{\left( 1\right) }},w_{\lambda
^{\left( 1\right) }},w_{\lambda ^{\left( 2\right) }}\left( \lambda \in
k\right) ,w_{+^{\left( 1\right) }},w_{\left[ ,\right] },w_{0^{\left(
2\right) }},w_{-^{\left( 2\right) }},w_{+^{\left( 2\right) }},w_{\circ
}\right\}  \label{repLiesign}
\end{equation}%
subjects of conditions Op1) and Op2).

As in the proof of the Theorem \ref{grouprepgroup} we have that $%
w_{0^{\left( 1\right) }}=0^{\left( 1\right) }$, $w_{-^{\left( 1\right)
}}\left( x^{\left( 1\right) }\right) =-x^{\left( 1\right) }$, $w_{\lambda
^{\left( 1\right) }}\left( x^{\left( 1\right) }\right) =\varphi \left(
\lambda \right) x^{\left( 1\right) }$, $w_{+^{\left( 1\right) }}\left(
x_{1}^{\left( 1\right) },x_{2}^{\left( 1\right) }\right) =x_{1}^{\left(
1\right) }+x_{2}^{\left( 1\right) }$, $w_{0^{\left( 2\right) }}=0^{\left(
2\right) }$, $w_{-^{\left( 2\right) }}\left( x^{\left( 2\right) }\right)
=-x^{\left( 2\right) }$, $w_{\lambda ^{\left( 2\right) }}\left( x^{\left(
2\right) }\right) =\psi \left( \lambda \right) x^{\left( 2\right) }$, $%
w_{+^{\left( 2\right) }}\left( x_{1}^{\left( 2\right) },x_{2}^{\left(
2\right) }\right) =x_{1}^{\left( 2\right) }+x_{2}^{\left( 2\right) }$, where 
$\varphi ,\psi \in \mathrm{Aut}k$. By \cite[2.5]{PlotkinZhitAutCat} $w_{%
\left[ ,\right] }\left( x_{1}^{\left( 1\right) },x_{2}^{\left( 1\right)
}\right) =a\left[ x_{1}^{\left( 1\right) },x_{2}^{\left( 1\right) }\right] $%
, where $a\in k\setminus \left\{ 0\right\} $.

$w_{\circ }\left( x^{\left( 1\right) },x^{\left( 2\right) }\right) =f\left(
x^{\left( 1\right) }\right) \circ x^{\left( 2\right) }$, where $f\left(
x^{\left( 1\right) }\right) \in k\left[ x^{\left( 1\right) }\right] $. There
is an isomorphism $s_{F\left( x^{\left( 1\right) },x^{\left( 2\right)
}\right) }:F\left( x^{\left( 1\right) },x^{\left( 2\right) }\right)
\rightarrow \left( F\left( x^{\left( 1\right) },x^{\left( 2\right) }\right)
\right) _{W}^{\ast }$, so the $w_{\lambda ^{\left( 2\right) }}\left(
w_{\circ }\left( x^{\left( 1\right) },x^{\left( 2\right) }\right) \right)
=w_{\circ }\left( w_{\lambda ^{\left( 1\right) }}\left( x^{\left( 1\right)
}\right) ,x^{\left( 2\right) }\right) $ must hold for every $\lambda \in k$. 
$k$ has a characteristic $0$, so $2\in 
%TCIMACRO{\U{211a} }%
%BeginExpansion
\mathbb{Q}
%EndExpansion
\subseteq k$, $2^{i}=2$ if and only if $i=1$, and the $\varphi \left(
2\right) =2$ holds for every $\varphi \in \mathrm{Aut}k$. We denote $f\left(
x^{\left( 1\right) }\right) =\sum\limits_{i=0}^{n}\alpha _{i}\left(
x^{\left( 1\right) }\right) ^{i}$, where $\alpha _{i}\in k$. Then $%
w_{\lambda ^{\left( 2\right) }}\left( w_{\circ }\left( x^{\left( 1\right)
},x^{\left( 2\right) }\right) \right) =\psi \left( 2\right) f\left(
x^{\left( 1\right) }\right) \circ x^{\left( 2\right)
}=\sum\limits_{i=0}^{n}2\alpha _{i}\left( x^{\left( 1\right) }\right)
^{i}\circ x^{\left( 2\right) }$, $w_{\circ }\left( w_{\lambda ^{\left(
1\right) }}\left( x^{\left( 1\right) }\right) ,x^{\left( 2\right) }\right)
=\sum\limits_{i=0}^{n}\alpha _{i}\left( \varphi \left( 2\right) x^{\left(
1\right) }\right) ^{i}\circ x^{\left( 2\right)
}=\sum\limits_{i=0}^{n}2^{i}\alpha _{i}\left( x^{\left( 1\right) }\right)
^{i}\circ x^{\left( 2\right) }$. So $\alpha _{i}=0$ if $i\neq 1$ and $%
w_{\circ }\left( x^{\left( 1\right) },x^{\left( 2\right) }\right)
=bx^{\left( 1\right) }\circ x^{\left( 2\right) }$, where $b\in k\setminus
\left\{ 0\right\} $.

Also the%
\begin{equation*}
w_{\circ }\left( w_{\left[ ,\right] }\left( x_{1}^{\left( 1\right)
},x_{2}^{\left( 1\right) }\right) ,x^{\left( 2\right) }\right) =w_{\circ
}\left( x_{1}^{\left( 1\right) },w_{\circ }\left( x_{2}^{\left( 1\right)
},x^{\left( 2\right) }\right) \right) -w_{\circ }\left( x_{2}^{\left(
1\right) },w_{\circ }\left( x_{1}^{\left( 1\right) },x^{\left( 2\right)
}\right) \right)
\end{equation*}%
must hold in the $F\left( x_{1}^{\left( 1\right) },x_{2}^{\left( 1\right)
},x^{\left( 2\right) }\right) $.%
\begin{equation*}
w_{\circ }\left( w_{\left[ ,\right] }\left( x_{1}^{\left( 1\right)
},x_{2}^{\left( 1\right) }\right) ,x^{\left( 2\right) }\right) =ab\left( %
\left[ x_{1}^{\left( 1\right) },x_{2}^{\left( 1\right) }\right] \circ
x^{\left( 2\right) }\right) ,
\end{equation*}
\begin{equation*}
w_{\circ }\left( x_{1}^{\left( 1\right) },w_{\circ }\left( x_{2}^{\left(
1\right) },x^{\left( 2\right) }\right) \right) -w_{\circ }\left(
x_{2}^{\left( 1\right) },w_{\circ }\left( x_{1}^{\left( 1\right) },x^{\left(
2\right) }\right) \right) =
\end{equation*}%
\begin{equation*}
b^{2}\left( x_{1}^{\left( 1\right) }\circ \left( x_{2}^{\left( 1\right)
}\circ x^{\left( 2\right) }\right) -x_{2}^{\left( 1\right) }\circ \left(
x_{1}^{\left( 1\right) }\circ x^{\left( 2\right) }\right) \right) =
\end{equation*}%
\begin{equation*}
b^{2}\left( \left[ x_{1}^{\left( 1\right) },x_{2}^{\left( 1\right) }\right]
\circ x^{\left( 2\right) }\right) ,
\end{equation*}%
so $a=b$.

Also the $w_{\circ }\left( w_{\lambda ^{\left( 1\right) }}\left( x^{\left(
1\right) }\right) ,x^{\left( 2\right) }\right) =w_{\circ }\left( x^{\left(
1\right) },w_{\lambda ^{\left( 2\right) }}\left( x^{\left( 2\right) }\right)
\right) $ must hold for every $\lambda \in k$. $w_{\circ }\left( w_{\lambda
^{\left( 1\right) }}\left( x^{\left( 1\right) }\right) ,x^{\left( 2\right)
}\right) =a\varphi \left( \lambda \right) x^{\left( 1\right) }\circ
x^{\left( 2\right) }$, $w_{\circ }\left( x^{\left( 1\right) },w_{\lambda
^{\left( 2\right) }}\left( x^{\left( 2\right) }\right) \right) =a\psi \left(
\lambda \right) x^{\left( 1\right) }\circ x^{\left( 2\right) }$, so $\varphi
=\psi $. So for systems of the words (\ref{repLiesign}) we have these
possibilities:%
\begin{equation*}
W=\{w_{0^{\left( 1\right) }}=0^{\left( 1\right) },w_{-^{\left( 1\right)
}}\left( x^{\left( 1\right) }\right) =-x^{\left( 1\right) },w_{\lambda
^{\left( 1\right) }}\left( x^{\left( 1\right) }\right) =\varphi \left(
\lambda \right) x^{\left( 1\right) }\left( \lambda \in k\right) ,
\end{equation*}%
\begin{equation}
w_{+^{\left( 1\right) }}\left( x_{1}^{\left( 1\right) },x_{2}^{\left(
1\right) }\right) =x_{1}^{\left( 1\right) }+x_{2}^{\left( 1\right) },w_{ 
\left[ ,\right] }\left( x_{1}^{\left( 1\right) },x_{2}^{\left( 1\right)
}\right) =a\left[ x_{1}^{\left( 1\right) },x_{2}^{\left( 1\right) }\right] ,
\label{repLiereallist}
\end{equation}%
\begin{equation*}
w_{0^{\left( 2\right) }}=0^{\left( 2\right) },w_{-^{\left( 2\right) }}\left(
x^{\left( 2\right) }\right) =-x^{\left( 2\right) },w_{\lambda ^{\left(
2\right) }}\left( x^{\left( 2\right) }\right) =\varphi \left( \lambda
\right) x^{\left( 2\right) }\left( \lambda \in k\right) ,
\end{equation*}%
\begin{equation*}
w_{+^{\left( 2\right) }}\left( x_{1}^{\left( 2\right) },x_{2}^{\left(
2\right) }\right) =x_{1}^{\left( 2\right) }+x_{2}^{\left( 2\right)
},w_{\circ }\left( x^{\left( 1\right) },x^{\left( 2\right) }\right)
=ax^{\left( 1\right) }\circ x^{\left( 2\right) }\},
\end{equation*}%
where $\varphi \in \mathrm{Aut}k$, $a\in k^{\ast }$.

Now we will prove that conditions for Op1) and Op2) fulfills for all these
systems of the words. By direct computations we can prove that if $H\in
\Theta $ then, for all these systems of the words $W$, in the algebra $%
H_{W}^{\ast }$ fulfill all identities (axioms) which define our variety $%
\Theta $. So $H_{W}^{\ast }\in \Theta $. Therefore for every $F=F\left(
X\right) \in \mathrm{Ob}\Theta ^{0}$ there exists homomorphism $%
s_{F}:F\rightarrow F_{W}^{\ast }$ such that $\left( s_{F}\right) _{\mid
X}=id_{X}$.

We will prove that $s_{F}$ is an an isomorphism. First of all we prove that $%
s_{F}$ is a monomorphism. We consider a basis $E$ of the free Lie algebra $%
L\left( X^{\left( 1\right) }\right) $ such that the elements of $E$ are
monomials of $L\left( X^{\left( 1\right) }\right) $. If $l=\sum\limits_{e\in
E^{\prime }}\alpha _{e}e\in L$, where $\alpha _{e}\in k$, $E^{\prime
}\subset E$, $\left\vert E^{\prime }\right\vert <\infty $, and $s_{F}\left(
l\right) =\sum\limits_{e\in E^{\prime }}\varphi \left( \alpha _{e}\right)
a^{l\left( e\right) -1}e=0$, where $l\left( e\right) $ is a length of the
monomial $e$, then the $\varphi \left( \alpha _{e}\right) a^{l\left(
e\right) }=0$ holds for every $e\in E^{\prime }$. $a\neq 0$, so $\varphi
\left( \alpha _{e}\right) =0$ and $\alpha _{e}=0$ for every $e\in E^{\prime
} $. Hence $l=0$. Now we consider $v\in \left( F\left( X\right) \right)
^{\left( 2\right) }=A\left( X^{\left( 1\right) }\right) \circ X^{\left(
2\right) }$. $v=\sum\limits_{i\in I}f_{i}\circ x_{i}^{\left( 2\right) }$,
where $\left\{ x_{i}^{\left( 2\right) }\mid i\in I\right\} =X^{\left(
2\right) }$, $f_{i}\in A\left( X^{\left( 1\right) }\right) $. $s_{F}\left(
v\right) =\sum\limits_{i\in I}s_{F}\left( f_{i}\circ x_{i}^{\left( 2\right)
}\right) $. We denote $f_{i}=\sum\limits_{j\in J_{i}}\beta _{i,j}m_{j}$,
where $\beta _{i,j}\in k$, $m_{j}$ are monomials of $A\left( X^{\left(
1\right) }\right) $, $\left\vert J_{i}\right\vert <\infty $. $s_{F}\left(
f_{i}\circ x_{i}^{\left( 2\right) }\right) =\sum\limits_{j\in J_{i}}\varphi
\left( \beta _{i,j}\right) a^{\deg \left( m_{j}\right) }m_{j}\circ
x_{i}^{\left( 2\right) }\in A\left( X^{\left( 1\right) }\right) \circ
x_{i}^{\left( 2\right) }$. Hence if $s_{F}\left( v\right) =0$ then the $%
s_{F}\left( f_{i}\circ x_{i}^{\left( 2\right) }\right) =0$ holds for every $%
i\in I$. Therefore the $\varphi \left( \beta _{i,j}\right) =0$ and $\beta
_{i,j}=0$ holds for every $i\in I$ and every $j\in J_{i}$. So $v=0$ and $%
s_{F}$ is a monomorphism.

By previous calculation we can see that for every $f\in F\left( X\right) $
there exists $h\in F\left( X\right) $, such that $s_{F}\left( v\right) =f$.
So $s_{F}$ is an isomorphism.

Now we know all elements of the group $\mathfrak{S}$: they are all strongly
stable automorphisms $\Phi ^{W}$, which correspond to the systems of words (%
\ref{repLiereallist}). Now we will study the multiplication in this group.
The automorphism $\Phi ^{W}$, which correspond to the system of words (\ref%
{repLiereallist}) we will denote by $\Phi \left( \varphi ,a\right) $. We
consider $\Phi _{i}=\Phi \left( \varphi _{i},a_{i}\right) $, $i=1,2$. The
system of bijections corresponding to the automorphism $\Phi _{i}$ we denote 
$S^{\Phi _{i}}=\left\{ s_{i,F}:F\rightarrow F_{W_{i}}^{\ast }\mid F\in 
\mathrm{Ob}\Theta ^{0}\right\} $. The system of words $W=W^{\Phi _{2}\Phi
_{1}}$ can be calculated by the formula (\ref{wordformula}): the $W=W^{\Phi
_{2}\Phi _{1}}$ has form (\ref{repLiesign}) with $w_{\lambda ^{\left(
1\right) }}\left( x^{\left( 1\right) }\right) =s_{2,F}s_{1,F}\left( \lambda
x^{\left( 1\right) }\right) =s_{2,F}\left( \varphi _{1}\left( \lambda
\right) x^{\left( 1\right) }\right) =\varphi _{2}\varphi _{1}\left( \lambda
\right) x^{\left( 1\right) }$, similarly $w_{\lambda ^{\left( 2\right)
}}\left( x^{\left( 2\right) }\right) =\varphi _{2}\varphi _{1}\left( \lambda
\right) x^{\left( 2\right) }$, $w_{\left[ ,\right] }\left( x_{1}^{\left(
1\right) },x_{2}^{\left( 1\right) }\right) =s_{2,F}s_{1,F}\left[
x_{1}^{\left( 1\right) },x_{2}^{\left( 1\right) }\right] =s_{2,F}\left( a_{1}%
\left[ x_{1}^{\left( 1\right) },x_{2}^{\left( 1\right) }\right] \right)
=\varphi _{2}\left( a_{1}\right) s_{2,F}\left[ x_{1}^{\left( 1\right)
},x_{2}^{\left( 1\right) }\right] =a_{2}\varphi _{2}\left( a_{1}\right) %
\left[ x_{1}^{\left( 1\right) },x_{2}^{\left( 1\right) }\right] $, similarly 
$w_{\circ }\left( x^{\left( 1\right) },x^{\left( 2\right) }\right)
=a_{2}\varphi _{2}\left( a_{1}\right) x^{\left( 1\right) }\circ x^{\left(
2\right) }$. So $\Phi \left( \varphi _{2},a_{2}\right) \Phi \left( \varphi
_{1},a_{1}\right) =\Phi \left( \varphi _{2}\varphi _{1},a_{2}\varphi
_{2}\left( a_{1}\right) \right) $.

If $\Phi =\Phi \left( \varphi ,a\right) $ and $\varphi \neq id_{k}$ then $%
\Phi $ is not an inner automorphism. The proof of this fact is even easier
than proof of the similar fact from \cite[Proposition 4.1]{TsurAutomEqLinAlg}%
.

Now we consider the automorphism $\Phi =\Phi \left( id_{k},a\right) $. We
can define for every $F\in \mathrm{Ob}\Theta ^{0}$ a mapping $\tau
_{F}:F\rightarrow F_{W}^{\ast }$, such that $\tau _{F}\left( f\right)
=a^{-1}f$ for every $f\in F$. This mapping is a homomorphism: for example%
\begin{equation*}
\tau _{F}\left[ f_{1}^{\left( 1\right) },f_{2}^{\left( 1\right) }\right]
=a^{-1}\left[ f_{1}^{\left( 1\right) },f_{2}^{\left( 1\right) }\right] ,
\end{equation*}%
\begin{equation*}
w_{\left[ ,\right] }\left( \tau _{F}\left( f_{1}^{\left( 1\right) }\right)
,\tau _{F}\left( f_{2}^{\left( 1\right) }\right) \right) =a\left[ \tau
_{F}\left( f_{1}^{\left( 1\right) }\right) ,\tau _{F}\left( f_{2}^{\left(
1\right) }\right) \right] =
\end{equation*}%
\begin{equation*}
a\left[ a^{-1}f_{1}^{\left( 1\right) },a^{-1}f_{2}^{\left( 1\right) }\right]
=a^{-1}\left[ f_{1}^{\left( 1\right) },f_{2}^{\left( 1\right) }\right]
\end{equation*}%
and%
\begin{equation*}
\tau _{F}\left( f^{\left( 1\right) }\circ f^{\left( 2\right) }\right)
=a^{-1}\left( f^{\left( 1\right) }\circ f^{\left( 2\right) }\right) ,
\end{equation*}%
\begin{equation*}
w_{\circ }\left( \tau _{F}\left( f^{\left( 1\right) }\right) ,\tau
_{F}\left( f^{\left( 2\right) }\right) \right) =a\left( a^{-1}f^{\left(
1\right) }\circ a^{-1}f^{\left( 2\right) }\right) =a^{-1}\left( f^{\left(
1\right) }\circ f^{\left( 2\right) }\right)
\end{equation*}%
and so on. Obviously, that $\tau _{F}$ is an invertible mapping. So $\tau
_{F}$ is an isomorphism. For every $F_{1},F_{2}\in \mathrm{Ob}\Theta ^{0}$
and every $\mu \in \mathrm{Mor}_{\Theta ^{0}}\left( F_{1},F_{2}\right) $ the%
\begin{equation*}
\mu \tau _{F_{1}}\left( f\right) =\mu \left( a^{-1}f\right) =a^{-1}\mu
\left( f\right) =\tau _{F_{2}}\mu \left( f\right)
\end{equation*}%
holds for every $f\in F_{1}$. Therefore, by Proposition \ref%
{intersectioncriterion}, $\Phi $ is an inner automorphism. So $\mathfrak{%
A/Y\cong S/S\cap Y\cong }\mathrm{Aut}k$.
\end{proof}

Now we will give an example of two representations of Lie algebras which are
automorphically equivalent but not geometrically equivalent. This example
will by similar to the \cite[Example 3]{TsurAutomEqVarLinAlg}. Congruences
are the kernels of homomorphisms. So it is easy to see that for every
congruence in the free representation of the Lie algebra $F\left( X\right) $
corresponds subrepresentation $Q$, such that $Q^{\left( 1\right) }$ is an
ideal of the Lie algebra $L\left( X^{\left( 1\right) }\right) $, $Q^{\left(
2\right) }$ is a $A\left( X^{\left( 1\right) }\right) $-submodule of the
module $A\left( X^{\left( 1\right) }\right) \circ X^{\left( 2\right) }$ and
the $l^{\left( 1\right) }\circ v^{\left( 2\right) }\in Q^{\left( 2\right) }$
holds for every $l^{\left( 1\right) }\in Q^{\left( 1\right) }$ and every $%
v^{\left( 2\right) }\in A\left( X^{\left( 1\right) }\right) \circ X^{\left(
2\right) }$.

We will consider an arbitrary field $k$ with characteristic $0$ and strongly
stable automorphism $\Phi =\Phi \left( \varphi ,1\right) $ where $\varphi
\neq id_{k}$. So there is $\lambda \in k$ such $\varphi \left( \lambda
\right) \neq \lambda $. We will consider the free representation $F=F\left(
x_{1}^{\left( 1\right) },x_{2}^{\left( 1\right) },x^{\left( 2\right)
}\right) $. In $L=L\left( x_{1}^{\left( 1\right) },x_{2}^{\left( 1\right)
}\right) $ there are linear independent elements%
\begin{equation*}
\left[ x_{1}^{\left( 1\right) },\left[ x_{1}^{\left( 1\right) },\left[
x_{1}^{\left( 1\right) },\left[ x_{1}^{\left( 1\right) },x_{2}^{\left(
1\right) }\right] \right] \right] \right] =e_{1},\left[ x_{1}^{\left(
1\right) },\left[ x_{1}^{\left( 1\right) },\left[ \left[ x_{1}^{\left(
1\right) },x_{2}^{\left( 1\right) }\right] ,x_{2}^{\left( 1\right) }\right] %
\right] \right] =e_{2},
\end{equation*}%
\begin{equation*}
\left[ x_{1}^{\left( 1\right) },\left[ \left[ \left[ x_{1}^{\left( 1\right)
s},x_{2}^{\left( 1\right) }\right] ,x_{2}^{\left( 1\right) }\right]
,x_{2}^{\left( 1\right) }\right] \right] =e_{3},\left[ \left[ x_{1}^{\left(
1\right) },\left[ x_{1}^{\left( 1\right) },x_{2}^{\left( 1\right) }\right] %
\right] ,\left[ x_{1}^{\left( 1\right) },x_{2}^{\left( 1\right) }\right] %
\right] =e_{4},
\end{equation*}%
\begin{equation*}
\left[ \left[ x_{1}^{\left( 1\right) },x_{2}^{\left( 1\right) }\right] ,%
\left[ \left[ x_{1}^{\left( 1\right) },x_{2}^{\left( 1\right) }\right]
,x_{2}^{\left( 1\right) }\right] \right] =e_{5},\left[ \left[ \left[ \left[
x_{1}^{\left( 1\right) },x_{2}^{\left( 1\right) }\right] ,x_{2}^{\left(
1\right) }\right] ,x_{2}^{\left( 1\right) }\right] ,x_{2}^{\left( 1\right) }%
\right] =e_{6}.
\end{equation*}%
These elements will be linear independent also in $A=A\left( x_{1}^{\left(
1\right) },x_{2}^{\left( 1\right) }\right) $. In $A$ we will consider the
two sided ideal $\left\langle x_{1}^{\left( 1\right) },x_{2}^{\left(
1\right) }\right\rangle =N$ and the two sided ideal $I=\left\langle
t,N^{6}\right\rangle $, where $t=\lambda e_{2}+e_{4}$. In $F$ we will
consider the congruence $T$ corresponds to the subrepresentation $Q$, such
that $Q^{\left( 1\right) }=\left\{ 0\right\} $, $Q^{\left( 2\right) }=I\circ
x^{\left( 2\right) }$. We will denote $H=F/T$ and $W=W^{\Phi }$. As above $H$
and $H_{W}^{\ast }$ are automorphically equivalent.

\begin{proposition}
\label{LieRepContrExamp}The representations $H$ and $H_{W}^{\ast }$ are not
geometrically equivalent.
\end{proposition}

\begin{proof}
As above $T\in Cl_{H}(F)$. If $H$ and $H_{W}^{\ast }$ are geometrically
equivalent then $T\in Cl_{H_{W}^{\ast }}(F)$ and $s_{F}\left( T\right) \in
Cl_{H}(F)$, where $s_{F}:F\rightarrow F_{W}^{\ast }$ is an isomorphism
subject of condition Op2). We will calculate $\left( s_{F}\left( T\right)
\right) _{H}^{\prime \prime }$ with a view to prove that $\left( s_{F}\left(
T\right) \right) _{H}^{\prime \prime }\neq s_{F}\left( T\right) $. This
contradiction will finish the prove.

If $s_{F}\left( T\right) $ is a congruence then it corresponds to the
subrepresentation $s_{F}\left( Q\right) $, such that $\left( s_{F}\left(
Q\right) \right) ^{\left( 1\right) }=\left\{ 0\right\} $, $\left(
s_{F}\left( Q\right) \right) ^{\left( 2\right) }=s_{F}\left( I\right) \circ
x^{\left( 2\right) }$, where $s_{F}\left( I\right) =\left\langle s_{F}\left(
t\right) ,N^{6}\right\rangle $. We denote the natural epimorphism $\tau
:F\rightarrow F/T$. As above for every $\psi \in \mathrm{Hom}\left(
F,H\right) $ there exists $\alpha \in \mathrm{End}\left( F\right) $ such
that $\tau \alpha =\psi $. If $s_{F}\left( T\right) \subseteq \ker \psi $,
then $\alpha \left( s_{F}\left( t\circ x^{\left( 2\right) }\right) \right)
=\alpha s_{F}\left( t\right) \circ \alpha \left( x^{\left( 2\right) }\right)
\in I\circ x^{\left( 2\right) }$. $\alpha \left( x^{\left( 2\right) }\right)
=\left( \mu +n\left( x_{1}^{\left( 1\right) },x_{2}^{\left( 1\right)
}\right) \right) \circ x^{\left( 2\right) }$, where $\mu \in k$, $n\left(
x_{1}^{\left( 1\right) },x_{2}^{\left( 1\right) }\right) \in N$.%
\begin{equation*}
\alpha s_{F}\left( t\right) \circ \alpha \left( x^{\left( 2\right) }\right)
=\alpha s_{F}\left( t\right) \left( \mu +n\left( x_{1}^{\left( 1\right)
},x_{2}^{\left( 1\right) }\right) \right) \circ x^{\left( 2\right) }\in
I\circ x^{\left( 2\right) }
\end{equation*}%
if and only if%
\begin{equation}
\alpha s_{F}\left( t\right) \left( \mu +n\left( x_{1}^{\left( 1\right)
},x_{2}^{\left( 1\right) }\right) \right) \in I=\mathrm{sp}_{k}\left\{
t\right\} +N^{6}.  \label{RepLieKernCond}
\end{equation}%
$s_{F}\left( t\right) =\varphi \left( \lambda \right) e_{2}+e_{4}\in
L^{5}\subset N^{5}$, so $\alpha s_{F}\left( t\right) \in L^{5}\subset N^{5}$%
. If $\mu =0$, then $\alpha s_{F}\left( t\right) \left( \mu +n\left(
x_{1}^{\left( 1\right) },x_{2}^{\left( 1\right) }\right) \right) \in
N^{6}\subset I$. But, if $\mu =0$, the $\alpha \left( l\right) \left( \mu
+n\left( x_{1}^{\left( 1\right) },x_{2}^{\left( 1\right) }\right) \right)
\in N^{6}\subset I$ holds for every $l\in N^{5}$. In this case $\left( 
\mathrm{sp}_{k}\left\{ e_{1},\ldots ,e_{6}\right\} \right) \circ x^{\left(
2\right) }\subseteq \ker \psi $, where $\ker \psi $ we understand as
subrepresentation corresponding to this congruence. If $\mu \neq 0$, then
from (\ref{RepLieKernCond}) we conclude that $\alpha s_{F}\left( t\right)
\in I=\mathrm{sp}_{k}\left\{ t\right\} +N^{6}$. It fulfills, as we can see
from the calculation of the \cite[Example 3]{TsurAutomEqVarLinAlg}, if and
only if $\alpha \left( x_{1}^{\left( 1\right) }\right) $ and $\alpha \left(
x_{2}^{\left( 1\right) }\right) $ are liner depend modulo $L^{2}$. In this
case $\alpha \left[ x_{1}^{\left( 1\right) },x_{2}^{\left( 1\right) }\right]
\in L^{3}$, $\alpha \left( \mathrm{sp}_{k}\left\{ e_{1},\ldots
,e_{6}\right\} \right) \subseteq L^{6}\subset N^{6}\subset I$ and also $%
\left( \mathrm{sp}_{k}\left\{ e_{1},\ldots ,e_{6}\right\} \right) \circ
x^{\left( 2\right) }\subseteq \ker \psi $. Therefore $\left( s_{F}\left(
T\right) \right) _{H}^{\prime \prime }=\bigcap\limits_{\substack{ \psi \in 
\mathrm{Hom}\left( F,H\right)  \\ s_{F}\left( T\right) \subseteq \ker \psi }}%
\ker \psi \supseteq \left( \mathrm{sp}_{k}\left\{ e_{1},\ldots
,e_{6}\right\} \right) \circ x^{\left( 2\right) }$ and $\left( s_{F}\left(
T\right) \right) _{H}^{\prime \prime }\neq s_{F}\left( T\right) $.
\end{proof}

\section{Acknowledgements.}

I am thankful to Prof. I. P. Shestakov who was very heedful to my research.

I acknowledge the support by FAPESP - Funda\c{c}\~{a}o de Amparo \`{a}
Pesquisa do Estado de S\~{a}o Paulo (Foundation for Support Research of the
State S\~{a}o Paulo), project No. 2010/50948-2.

\end{document}